\documentclass[a4paper,10pt]{amsart}

\usepackage{amsthm,amsfonts,amsmath,amssymb,fixmath,mathtools,accents}
\usepackage[T1]{fontenc}
\usepackage[latin9]{inputenc}
\usepackage{lmodern}
\usepackage[all]{xy} % diagrams
\usepackage{tikz} % pictures
\usepackage{tikz-qtree} %tree
\usepackage{bbm} % nice sets N,Z,...
\usepackage{calrsfs} % calligraphic letters
\usepackage{enumerate} %lists
\usepackage{hyperref} %links in the pdf

\newtheorem{thm}{Theorem}[section]
\newtheorem{prop}[thm]{Proposition}
\newtheorem{lem}[thm]{Lemma}
\newtheorem{cor}[thm]{Corollary}

\theoremstyle{definition}
\newtheorem{defn}[thm]{Definition}
\newtheorem{exmp}[thm]{Example}

\theoremstyle{remark}
\newtheorem{rem}[thm]{Remark}
\newtheorem*{case}{Case}
\newtheorem*{question}{Question}

\numberwithin{equation}{section}

\author{Antonius Hase}
\address{Mathematics Department, Technion, Haifa 32000, Israel}
\email{hase@campus.technion.ac.il}
\title[The $\mathrm{Out}(F_n)$-action on $H^2_b(F_n, \mathbb{R})$]{Dynamics of $\mathrm{Out}(F_n)$ on the second bounded cohomology of $F_n$}

\begin{document}

\begin{abstract}
We study the $\mathrm{Out}(F_n)$-action on the second bounded cohomology $H^2_b(F_n, \mathbb{R})$, focusing on the countable-dimensional dense invariant subspace given by Brooks quasimorphisms. We show that this subspace has no finite-dimensional invariant subspaces, in particular no fixpoints, partially answering a question of Mikl\'{o}s Ab\'{e}rt. To this end we introduce a notion of speed of an element $g\in \mathrm{Out}(F_n)$, which measures the asymptotic growth rate of bounded cohomology classes under repeated application of $g$.
\end{abstract}

\maketitle

\section{Introduction}

In geometric group theory we study a group by studying its action on an appropriate space. In dynamics we study a space by studying what its elements do under the action of an appropriate group. Both points of view prompt us to study the action of the group of outer automorphisms of the free group $\mathrm{Out}(F_n)$ on the second bounded cohomology of the free group $H^2_b(F_n, \mathbb{R})$. While the $\mathrm{Out}(G)$-action on the group cohomology $H^k(G, \mathbb{R})$ of a group $G$ is a classical object of study in group theory, almost nothing is known about the corresponding $\mathrm{Out}(G)$-action on the bounded group cohomology $H^k_b(G, \mathbb{R})$. 

The study of $\mathrm{Out}(F_n)$ is a classical subject in geometric group theory going back to work of Nielsen \cite{Nielsen} and Whitehead \cite{Whitehead}. For further background on $\mathrm{Out}(F_n)$ see e.g. \cite{Laminations,Titsalternative,Traintracks,OuterSpace,Survey}.

The subject of bounded cohomology was popularized by the articles of Brooks \cite{Brooks}, Gromov \cite{Gromov} and Ivanov \cite{Ivanov}. For further background on bounded cohomology see e.g. \cite{BestvinaBrombergFujiwara2015,BestvinaBrombergFujiwara2016,Buehler-Foundations,BurgerMonod,scl,HullOsin,MonodThesis,Invitation}.

The starting point for our particular line of inquiry is a question asked by Mikl\'{o}s Ab\'{e}rt in \cite{Abert}.\begin{question}[47]\label{47} Can a non-Abelian free group $F$ have a nontrivial pseudocharacter that is invariant under $\mathrm{Aut}(F)$?\end{question} This question was recently answered in the positive for $F_2$ by Brandenbursky and Marcinkowski in \cite{Brandenbursky}, but remains open for $F_n$ with $n>2$. What Ab\'{e}rt calls a \emph{pseudocharacter}, we will call a \emph{homogeneous quasimorphism}. This definition is in line with a lot of the literature, in particular it is in line with the articles \cite{HS} and \cite{HT}, which form the basis of this article. In Section \ref{background} we will explain how we can identify the space of homogeneous quasimorphisms on the free group $\tilde{\mathcal{Q}}(F_n)$ up to homomorphisms with the second bounded cohomology $H^2_b(F_n, \mathbb{R})$. Tweaking Ab\'{e}rt's question only a tiny bit, we get the following question: Does the $\mathrm{Out}(F_n)$ action on $H^2_b(F_n, \mathbb{R})$ have a non-trivial fixpoint?

In this article we consider a countable-dimensional dense subspace $\tilde{\mathcal{B}}(F_n)$ of $\tilde{\mathcal{Q}}(F_n)$, which we call the Brooks space, and the corresponding subspace $H^2_b(F_n, \mathbb{R})_{\mathrm{fin}}$ of $H^2_b(F_n, \mathbb{R})$. By \cite{HS}, these spaces are invariant under the action of $\mathrm{Out}(F_n)$, and we will show the following theorem in Corollary \ref{halfAbert} and Corollary \ref{nofindiminv}.
\\

\begin{thm}
Let $n\geq 2$.
\begin{enumerate}
\item The action of $\mathrm{Out}(F_n)$ on the Brooks space $\tilde{\mathcal{B}}(F_n)$ has no non-trivial fixpoints.
\item The action of $\mathrm{Out}(F_n)$ on $H^2_b(F_n, \mathbb{R})_{\mathrm{fin}}$ admits no non-trivial finite-dimensional $\mathrm{Out}(F_n)$-invariant subspaces. In particular it has no non-trivial fixed point. 
\end{enumerate}
\end{thm}

The article is organized as follows: In Section \ref{background} we will explain the necessary background. In Section \ref{countfctquasi} we will present the Brooks space $\tilde{\mathcal{B}}(F_n)$ and introduce a norm on it. In Section \ref{OutFnaction} we will reproduce a proof of the fact that the Brooks space $\tilde{\mathcal{B}}(F_n)$ is invariant under the action of $\mathrm{Out}(F_n)$ on $\tilde{\mathcal{Q}}(F_n)$. We will also introduce the notion of the \emph{speed} of an element $X\in \mathrm{Out}(F_n)$ on an element $[f]$ of the Brooks space, which measures how fast the norm of $[f]$ grows asymptotically under repeated application of the element $X$. In Section \ref{Tsingle} we show that a special element $T^{-1}\in\mathrm{Out}(F_n)$ has linear speed on the Brooks space. In Section \ref{Tsums} we will show a way to compute the speed of $T^{-1}$ on any element of the Brooks space. This will help us to answer Ab\'{e}rt's question for the Brooks space in Section \ref{fixpts}: No non-trivial element of the Brooks space is fixed by the action of $\mathrm{Out}(F_n)$. In Section \ref{open} we will sketch questions and possible lines of further research that came up writing this article.

\section*{Acknowledgement}
This article is up to minor adjustments identical with my master's thesis at the Technion from August 2016. I want to thank my advisor Tobias Hartnick for his support in writing the thesis. I also want to acknowledge the generous financial support through the Department of Mathematics of the Technion and through ISF Grant No. 535/14.

\section{Background}\label{background}

In this section we will explain the background necessary for this article. We start in the first subsection by explaining group cohomology. We do this with the goal of explaining bounded group cohomology in the second subsection. Then we will present the $\mathrm{Out}(G)$ action on bounded group cohomology for any group $G$. We will end this section with a remark about $\mathrm{Out}(F_n)$. While we will present the definitions of (bounded) group cohomology for general coefficients, we will in low degrees only consider $\mathbb{R}$ with trivial $G$ action.

\subsection{Group cohomology}

We will mention three different view points on group cohomology - a topological, a categorical and a combinatorial one.

\subsubsection*{Topological definition}

We start with the topological definition of group cohomology, which is also historically the origin of this field of study.

\begin{defn}
We call a connected topological space $BG$ an \emph{Eilenberg-MacLane space} of a group $G$ if $\pi_1(BG)=G$ and $\pi_k(BG)=0$ for all $k>1$.
\end{defn}

\begin{prop}
For every group $G$ there exists an Eilenberg-MacLane space $BG$. It is unique up to weak homotopy equivalence. 
\end{prop}

\begin{proof}
See \cite{EML1} and \cite{EML2} for the original proof. Most textbooks on algebraic topology also give a proof, for example \cite{Hatcher} in Section 4.2.
\end{proof}

\begin{defn}\label{topdefn}
The \emph{group cohomology} of $G$ with coefficients in $\mathbb{R}$ is defined as the cohomology of the Eilenberg-MacLane space $BG$ with coefficients in $\mathbb{R}$, i.e. \[H^\bullet(G,\mathbb{R})=H^\bullet_{sing}(BG,\mathbb{R}).\]
\end{defn}

\begin{exmp}
Let $G=F_n$ be the free group with $n$ generators. Then $BF_n$ is a rose with $n$ petals. We get \begin{equation}H^k(F_n,\mathbb{R})=\begin{cases}
\mathbb{R},& \text{if } k=0\\
\mathbb{R}^n,& \text{if } k=1\\
0,& \text{if } k>1\\
\end{cases}.\end{equation}
\end{exmp}

\subsubsection*{Categorical definition}

Now we want to understand group cohomology as a derived functor. Let $R$ be a commutative ring and $G$ a group. We denote by $_{R}\mathrm{Mod}$ the category of left $R$-modules and by $_{RG}\mathrm{Mod}$ the category of representations of $G$ over $R$, i.e. $_{RG}\mathrm{Mod}$ is the category of left $RG$-modules, where $RG$ is the group ring. Since both categories are module categories, they are both abelian. Let $_{\epsilon}(-):_{R}\mathrm{Mod} \to _{RG}\mathrm{Mod}$ be the trivial module functor: The functor sends an object $V$ to $V$ with the trivial $G$-action and a morphism $f$ to $f$. The functor $_{\epsilon}(-)$ is left adjoint and we denote its right adjoint functor by $(-)^G$. The functor $(-)^G$ is left exact. Since module categories have enough injectives, we can consider the right-derived functors of $(-)^G$. We define the following:

\begin{defn}\label{catdefn}
The \emph{group cohomology} of $G$ with coefficients in a $RG$-module $V$ is defined by \[H^n(G,V)=R^n(-)^G(V).\]
\end{defn}

\begin{prop}
This categorical definition agrees with the topological one if $R=\mathbb{R}$ and $V=\mathbb{R}$ with the trivial $G$ action.
\end{prop}

\begin{proof}
See Section I.4 in \cite{Brown}.
\end{proof}

\subsubsection*{Bar complexes}

The categorical definition of group cohomology shows us that we can compute group cohomology from a lot of different complexes. Two complexes are especially well known and useful: The so called \emph{bar complexes} are so prominent that group cohomology is sometimes defined in their terms. We see them as providing a combinatorial point of view on group cohomology.

\begin{defn}
The \emph{homogeneous bar complex for $H^n(G,V)$} is \[{\xymatrix{0\ar[r] & \mathrm{Map}(G,V)^G\ar[r]^-{d^0} & \mathrm{Map}(G^2,V)^G\ar[r]^-{d^1} & \mathrm{Map}(G^3,V)^G\ar[r]^-{d^2} & ... }},\] where \begin{equation}
d^n f(g_0,...,g_{n+1})=\sum_{i=0}^{n+1} (-1)^i f(g_0,...,\hat{g_i},...,g_{n+1}).
\end{equation}
\end{defn}

\begin{prop}
We can compute the cohomology of a group $G$ using the homogeneous bar complex:\[H^\bullet(G,V)=H^\bullet(\mathrm{Map}(G^\bullet,V),d^\bullet)\]
\end{prop}

\begin{proof}
Let $B_n$ be the free $R$-module with basis $(g_0\mid...\mid g_n)$ for $g_i\in G$ and with the diagonal $G$ action. So $B_n$ is an object of $_{RG}\mathrm{Mod}$. We define a differential $d_n: B_n\to B_{n-1}$ by $d_n (g_0\mid...\mid g_{n})=\sum_{i=0}^{n} (-1)^i (g_0\mid...\mid\hat{g_i}\mid...\mid g_{n})$. Then $(B_\bullet,d_\bullet)$ is a free resolution of $_{\epsilon}(R)$ in $_{RG}\mathrm{Mod}$. See for example Section I.5 in \cite{Brown}.
\end{proof}

\begin{defn}
The \emph{inhomogeneous bar complex for $H^n(G,V)$} is \[{\xymatrix{0\ar[r] & \mathrm{Map}(G^0,V)\ar[r]^-{\partial^0} & \mathrm{Map}(G^1,V)\ar[r]^-{\partial^1} & \mathrm{Map}(G^2,V)\ar[r]^-{\partial^2} & ... }}\] where \begin{equation}\begin{split}
\partial^n f(g_0,...,g_n)=\
&g_0f(g_1,...,g_n)+(-1)^{n+1}f(g_0,...,g_{n-1})\\
&-\sum_{j=0}^{n-1} (-1)^{j+1} f(g_0,...,g_jg_{j+1},...,g_n).
\end{split}\end{equation}
\end{defn}

\begin{cor}\label{inhom}
We can compute the cohomology of a group $G$ using the inhomogeneous bar complex:\[H^\bullet(G,V)=H^\bullet(\mathrm{Map}(G^{\bullet-1},V),\partial^{\bullet-1})\]
\end{cor}

\begin{proof}
For this proof use an isomorphism between $\mathrm{Map}(G^{n+1},V)^G$ and $\mathrm{Map}(G^n,V)$. See again Section I.5 in \cite{Brown}.
\end{proof}

\subsubsection*{Low degrees}

The inhomogeneous bar complex is a very useful tool to compute group cohomology in low degrees. We get:\begin{enumerate}
\item $H^0(G,\mathbb{R})=\ker(\partial^0)=\mathbb{R}$,
\item $H^1(G,\mathbb{R})=\ker(\partial^1)/\text{im}(\partial^0)=\mathrm{Hom}(G,\mathbb{R})$.
\end{enumerate} To classify $H^2(G,\mathbb{R})$ more work is necessary.

\begin{defn}
A \emph{central extension} of a group $G$ by a group $A$ is a short exact sequence ${\xymatrix{0\ar[r] & A\ar[r] & \hat{G}\ar[r] & G\ar[r] & 1}}$ such that $A$ is in the center of $\hat{G}$. 
\end{defn}

\begin{prop}
The second cohomology $H^2(G,\mathbb{R})$ classifies central extensions of $G$ by $\mathbb{R}$ up to isomorphism.
\end{prop}

\begin{proof}
See for example Section IV.3 in \cite{Brown}.
\end{proof} There are more complicated group theoretic interpretations of the higher cohomology groups. For more information see \cite{Brown}.

\subsection{Bounded group cohomology}

As above we will mention the topological, categorical and combinatorial point of view on bounded cohomology. Since we are interested in $H^2_b(F_n, \mathbb{R})$, we will pay much more attention on low degrees.

\subsubsection*{Topological definition}

\begin{defn}% Ivanov
Let $X$ be a topological space. We denote by $S_n(x)$ the set of $n$-dimensional singular simplices in $X$. Then we call a bounded function $S_n(X)\to \mathbb{R}$ a \emph{bounded $n$-cochain}. Let $B^n(X)$ denote the space of bounded $n$-cochains. If $d$ is the usual differential of the cochain complex for singular cohomology, we get that $d^n(B^n(X))\subset B^{n+1}(X)$. So we find a cochain complex \[{\xymatrix{0\ar[r] & B^0(X)\ar[r]^-{d^0} & B^1(X)\ar[r]^-{d^1} & B^2(X)\ar[r]^-{d^2} & ...}},\] where $d$ is the usual differential of the cochain complex for singular cohomology. We call the cohomology of this complex the \emph{bounded cohomology} of $X$.
\end{defn}

\begin{defn}
The \emph{bounded cohomology} of $G$ with coefficients in $\mathbb{R}$ is the bounded cohomology of the Eilenberg-MacLane space of $G$.
\end{defn}

\subsubsection*{Categorical definition}

A categorical definition for bounded group cohomology was only found rather recently by B\"uhler. I will only be able to sketch his construction here, for the whole story see \cite{Buehler-Foundations}. The idea is to mimic the construction for group cohomology as long as possible. Since we have to make sense of the notion of boundedness, we can not consider all modules over a commutative ring $R$. Let $G$ be a group and for simplicity let $R$ be a field $k\in \{\mathbb{R},\mathbb{C}\}$. We denote by $\mathrm{Ban}_k$ the category of Banach spaces over $k$ and by $\mathrm{G-Ban}_k$ the category of isometric representations of $G$ on Banach spaces over $k$. More explicitely we have for $\mathrm{Ban}_k$: \begin{enumerate}
\item $V$ is an object of $\mathrm{Ban}_k$ if and only if $V$ is a Banach space over $k$.
\item $f:V_1 \to V_2$ is a morphism of $\mathrm{Ban}_k$ if and only if $f:V_1 \to V_2$ is $k$-linear and bounded.
\end{enumerate}and we have for $\mathrm{G-Ban}_k$:\begin{enumerate}
\item $V$ is an object of $\mathrm{G-Ban}_k$ if and only if $V$ is a Banach space over $k$ and $G$ acts on $V$ by $k$-linear isometries.
\item $f:V_1 \to V_2$ is a morphism of $\mathrm{G-Ban}_k$ if and only if $f:V_1 \to V_2$ is $k$-linear and bounded and $G$-equivariant.
\end{enumerate} Let $_{\epsilon}(-): \mathrm{Ban}_k \to \mathrm{G-Ban}_k$ be the trivial module functor. The functor $_{\epsilon}(-)$ is left adjoint and we denote the right adjoint functor of $_{\epsilon}(-)$ by $(-)^G$. The functor $(-)^G$ is left exact and now we would like to consider the right-derived functors of $(-)^G$. This is a priori not possible, since $\mathrm{Ban}_k$ and $\mathrm{G-Ban}_k$ are not abelian categories. But they are exact categories, which is almost as good. The following small digression in exact categories is from \cite{Buehler-Exact}.

\begin{defn}%Buehler Exact Categories, page 5
Let $\mathcal{A}$ be an additive category. A \emph{kernel-cokernel pair} $(i,p)$ in $\mathcal{A}$ is a pair of composable morphisms ${\xymatrix{A'\ar[r]^-{i} & A\ar[r]^-{p} & A''}}$ such that $i$ is a kernel of $p$ and $p$ is a cokernel of $i$. If a class $\mathcal{E}$ of kernel-cokernel pairs on $\mathcal{A}$ is fixed, an \emph{admissable monic} is a morphism $i$ for which there exists a morphism $p$ such that $(i,p)\in \mathcal{E}$.
We depict admissible monics by $\rightarrowtail$. If a class $\mathcal{E}$ of kernel-cokernel pairs on $\mathcal{A}$ is fixed, an \emph{admissable epic} is a morphism $p$ for which there exists a morphism $i$ such that $(i,p)\in \mathcal{E}$.
We depict admissible epics by $\twoheadrightarrow$. An \emph{exact structure} on $\mathcal{A}$ is a class $\mathcal{E}$ of kernel-cokernel pairs, which is closed under isomorphisms and satisfies the following axioms:

\begin{enumerate}
\item[$E0$] For all objects $A\in \mathcal{A}$, the identity morphism $1_A$ is an admissible monic.
\item[$E0^\text{op}$] For all objects $A\in \mathcal{A}$, the identity morphism $1_A$ is an admissable epic.
\item[$E1$] The class of admissible monics is closed under composition.
\item[$E1^\text{op}$] The class of admissible epics is closed under composition.
\item[$E2$] The push-out of an admissible monic along an arbitrary morphism exists and yields an admissible monic.
\item[$E2^\text{op}$] The pull-back of an admissible epic along an arbitrary morphism exists and yields an admissible epic.
\end{enumerate} An \emph{exact category} is a pair $(\mathcal{A},\mathcal{E})$ consisting of an additive category $\mathcal{A}$ and an exact strucure $\mathcal{E}$ on $\mathcal{A}$. Elements of $\mathcal{E}$ are called \emph{short exact sequences}.
\end{defn}

\begin{defn}
An object $I$ of an exact category $(\mathcal{A},\mathcal{E})$ is called \emph{injective} if every admissible monic $I\rightarrowtail A$ splits. An exact category $(\mathcal{A},\mathcal{E})$ has \emph{enough injectives} if for every object $A\in \mathcal{A}$ there exists an injective object $I$ and an admissible monic $A\rightarrowtail I$.
\end{defn}

\begin{prop}
If $(\mathcal{A},\mathcal{E})$ has enough injectives, then every object $A\in \mathcal{A}$ has an injective resolution $A\rightarrowtail I^\bullet$.
\end{prop}

\begin{proof}
See Proposition 12.2 in \cite{Buehler-Exact}.
\end{proof}

\begin{thm}
Let $(\mathcal{A},\mathcal{E})$ be an exact category with enough injectives and let $F:\mathcal{A}\rightarrow\mathcal{B}$ be an additive functor. Let $A$ be an object in $\mathcal{A}$ and let $A\rightarrowtail I^\bullet$ be an injective resolution. Then we can define the \emph{right derived functors} of $F$ as $R^iF(A)=H^i(F(I^\bullet))$.
\end{thm}

\begin{proof}
See Remark 12.11 in \cite{Buehler-Exact}.
\end{proof}

Now we can apply the theory of exact categories to our situation: Let $\mathcal{E}_{\text{max}}$ be the class of all kernel-cokernel pairs in $\mathrm{Ban}_k$ and let $\mathcal{E}^G_{\text{rel}}$ be the class of all kernel-cokernel pairs $(i,p)$ in $\mathrm{G-Ban}_k$ such that $(i,p)$ splits in $\mathrm{Ban}_k$. Then $(\mathrm{Ban}_k,\mathcal{E}_{\text{max}})$ and $(\mathrm{G-Ban}_k,\mathcal{E}^G_{\text{rel}})$ are exact categories. The exact category $(\mathrm{G-Ban}_k,\mathcal{E}^G_{\text{rel}})$ has enough injectives, so we can define right derived functors.

\begin{defn}
The \emph{bounded group cohomology} of $G$ with coefficients in a $G$-Banach space $V$ is defined by \[H_b^n(G,V)=R^n(-)^G(V),\] where $(-)^G: (\mathrm{G-Ban}_k,\mathcal{E}^G_{\text{rel}}) \to (\mathrm{Ban}_k,\mathcal{E}_{\text{max}})$ is the right adjoint of the trivial module functor $_{\epsilon}(-): (\mathrm{Ban}_k,\mathcal{E}_{\text{max}}) \to (\mathrm{G-Ban}_k,\mathcal{E}^G_{\text{rel}})$. 
\end{defn}

\begin{prop}
This categorical definition agrees with the topological one if $k=\mathbb{R}$ and $V=\mathbb{R}$ with the trivial $G$ action.
\end{prop}

\begin{proof}
See Proposition 11.3 in \cite{Buehler-Exact} and \cite{Ivanov}.
\end{proof}

\subsubsection*{Bar complexes}

Let $V$ be an object of $\mathrm{G-Ban}_k$. Then we can find a subcomplex of the homogeneous and inhomogeneous bar complex that computes bounded cohomology.

\begin{defn}
The \emph{bounded homogeneous bar complex for $H_b^n(G,V)$} is \[{\xymatrix{0\ar[r] & \mathrm{Map}_b(G,V)^G\ar[r]^-{d^0} & \mathrm{Map}_b(G^2,V)^G\ar[r]^-{d^1} & \mathrm{Map}_b(G^3,V)^G\ar[r]^-{d^2} & ... }},\] where \begin{equation}
d^n f(g_0,...,g_{n+1})=\sum_{i=0}^{n+1} (-1)^i f(g_0,...,\hat{g_i},...,g_{n+1}).
\end{equation}
\end{defn}

\begin{cor}
We can compute the bounded cohomology of a group $G$ using the bounded homogeneous bar complex:\[H_b^\bullet(G,V)=H^\bullet(\mathrm{Map}_b(G^\bullet,V),d^\bullet)\]
\end{cor}

\begin{proof}
See Proposition 11.3 in \cite{Buehler-Exact} and Lemma 3.2.2 in \cite{Ivanov}.
\end{proof}

\begin{defn}
The \emph{bounded inhomogeneous bar complex for $H_b^n(G,V)$} is \[{\xymatrix{0\ar[r] & \mathrm{Map}_{b}(G^0,V)\ar[r]^-{\partial^0} & \mathrm{Map}_{b}(G^1,V)\ar[r]^-{\partial^1} & \mathrm{Map}_{b}(G^2,V)\ar[r]^-{\partial^2} & \cdots }},\] where \begin{equation}\begin{split}
\partial^n f(g_0,...,g_n)=\
&g_0f(g_1,...,g_n)+(-1)^{n+1}f(g_0,...,g_{n-1})\\
&-\sum_{j=0}^{n-1} (-1)^{j+1} f(g_0,...,g_jg_{j+1},...,g_n).
\end{split}\end{equation}
\end{defn}

\begin{cor}
We can compute the bounded cohomology of a group $G$ using the bounded inhomogeneous bar complex:\[H_b^\bullet(G,V)=H^\bullet(\mathrm{Map}_b(G^{\bullet-1},V),\partial^{\bullet-1}).\]
\end{cor}

\begin{proof}
The isomorphism used in the proof of \ref{inhom} induces an isomorphism between $\mathrm{Map}_{b}(G^{n+1},V)^G$ and $\mathrm{Map}_{b}(G^n,V)$.
\end{proof}

\subsubsection*{Low degrees}

The bounded inhomogeneous bar complex is a very useful tool to compute bounded group cohomology in low degrees. We get:\begin{enumerate}
\item $H^0(G,\mathbb{R})=\ker(\partial^0)=\mathbb{R}$,
\item $H^1(G,\mathbb{R})=\ker(\partial^1)/\text{im}(\partial^0)=\mathrm{Hom}_b(G,\mathbb{R})=0$.
\end{enumerate} Since we are interested in $H_b^2(F_n,\mathbb{R})$, we will consider the group theoretic interpretation of $H_b^2(G,\mathbb{R})$ carefully. The space is $H_b^2(G,\mathbb{R})$ is related to the space of \emph{quasimorphisms} on $G$.

\begin{defn}
The chain map of inclusions $\mathrm{Map}_b(G^\bullet,\mathbb{R})\hookrightarrow \mathrm{Map}(G^\bullet,\mathbb{R})$ induces a map on the level of cohomology: \[c^\bullet: H_b^\bullet(G,\mathbb{R})\to H^\bullet(G,\mathbb{R}), [f]_b \mapsto [f].\] This map is called the \emph{comparison map}.
\end{defn}

One of the main results in the study of bounded cohomology is the fact that the comparison map is in general neither injective nor surjective. It is easy to see in an example that $c^\bullet$ does not have to be surjective.

\begin{exmp}
The map $c^1: H_b^1(\mathbb{Z},\mathbb{R})\to H^\bullet(\mathbb{Z},\mathbb{R})$ is not surjective. We already saw that $H_b^1(\mathbb{Z},\mathbb{R})=\mathrm{Hom}_b(\mathbb{Z},\mathbb{R})=0$ and $H^1(\mathbb{Z},\mathbb{R})=\mathrm{Hom}(\mathbb{Z},\mathbb{R})\neq 0$.
\end{exmp}

To find an example in which $c^\bullet$ is not injective is harder. What helps us is that we can find a nice description of the kernel of $c^2$. Here quasimorphisms enter the picture.

\begin{defn}
A function $f$ on a group $G$ is called a \emph{quasimorphism} if \begin{equation} \lVert \partial^1f \rVert_\infty = \sup_{g,h\in G} \lvert f(gh) - f(g) - f(h)\rvert =D(f)< \infty. \end{equation} $D(f)$ is called the \emph{defect} of the quasimorphism $f$. We denote by $\mathcal{Q}(G)$ the space of quasimorphisms on a group $G$. We define an equivalence relation on $\mathcal{Q}(G)$ by \[f_1\sim f_2\ \text{if and only if}\ \lVert f_1-f_2 \rVert_\infty < \infty.\] We denote by $\tilde{\mathcal{Q}}(G)=\mathcal{Q}(G)/\sim$ the space of equivalence classes. 
\end{defn}

Note that $f\sim 0$ for $f\in \mathcal{Q}(G)$ if and only if $f\in \mathrm{Map}_b(G,\mathbb{R})$. Instead of looking at equivalence classes of quasimorphisms, one can also pick a suitable representative for every equivalence class and consider the space of representatives.

\begin{defn}
A quasimorphism is called \emph{homogeneous} if for every $g\in G$ and $n\in \mathbb{N}$ we have $f(g^n)=nf(g)$.
\end{defn}

\begin{prop}\label{homquasi}
In every equivalence class $[f]\in \tilde{\mathcal{Q}}(G)$ there is exactly one homogeneous quasimorphism $\tilde{f}$. So we can identify $\tilde{\mathcal{Q}}(G)$ with the space of homogeneous quasimorphisms.
\end{prop}

\begin{proof}
We prove the proposition in four steps: Let $f\in \mathcal{Q}(G)$. Then \begin{enumerate}
\item $\tilde{f}(g)=\lim_{n}\frac{f(g^n)}{n}$ exists for every $g\in G$.
\item $\lVert \tilde{f}-f \rVert_\infty < \infty$. In particular $\tilde{f}$ is a quasimorphism.
\item $\tilde{f}$ is homogeneous.
\item $\tilde{f}$ is the unique homogeneous quasimorphism such that $\tilde{f}\sim f$.
\end{enumerate} Together these points prove the proposition.
\begin{enumerate}
\item Since $f$ is a quasimorphism, we know that $\lvert f(g^n)-nf(g)\rvert\leq (n-1)D(f)$ for every $g\in G$. So we know that $-(n-1)D(f)\leq f(g^n)-nf(g)\leq (n-1)D(f)$. After switching $f$ to $-f$ we can assume wlog that $f(g)\geq 0$. (If $\tilde{-f}$ exists, so does $\tilde{f}$.) Define $a_n=f(g^n)+(n+1)D(f)$. Then we have $a_n\geq nf(g)+2D(f)\geq 0$ and \begin{equation}\begin{split}
a_{m+n}=& f(g^{m+n})+(m+n+1)D(f)\\
\leq& f(g^m)+f(g^n)+D(f)+(m+n+1)D(f)\\
=& f(g^m)+(m+1)D(f)+f(g^n)+(n+1)D(f)\\
=& a_m+a_n.
\end{split}\end{equation} So we can apply Fekete's Lemma and get that $\lim_n \frac{a_n}{n}$ exists. Then it follows that $\tilde{f}(g)=\lim_n \frac{a_n}{n}-\frac{n+1}{n}D(f)$ exists as well.
\item For every $g\in G$ and $n\in \mathbb{N}$ we have $\lvert \frac{f(g^n)}{n}-f(g) \rvert = \frac{1}{n}\lvert f(g^n)-nf(g)\rvert\leq \frac{n-1}{n} D(f)$. So taking the limit we get that $\lVert \tilde{f}-f \rVert_\infty\leq D(f)<\infty$. 

Now we prove in general that $f_2$ is a quasimorphism, if $f_1$ is a quasimorphism and $\lVert f_1-f_2 \rVert_\infty$. This follows, since\begin{equation}\begin{split}
\lvert f_2(gh)-f_2(g)-f_2(h)\rvert=&\lvert f_2(gh)-f_1(gh)+f_1(gh)\\
&-f_2(g)-f_1(g)+f_1(g)\\
&-f_2(h)-f_1(h)+f_1(h)\rvert\\
\leq& \lvert f_1(gh)-f_1(g)-f_1(h)\rvert\\
&+ \lvert f_2(gh)-f_1(gh)\rvert\\
&+ \lvert f_2(g)-f_1(g)\rvert\\
&+\lvert f_2(h)-f_1(h)\rvert\\
\leq& D(f_1)+3\lVert f_1-f_2 \rVert_\infty\\
<&\infty.
\end{split}\end{equation}
\item Let $g\in G$ and $k\in \mathbb{N}$. Then we have $\tilde{f}(g^k)=\lim_{n}\frac{f(g^{kn})}{n}=k\lim_{n}\frac{f(g^{kn})}{kn}=k\tilde{f}(g)$.
\item Assume that $f_1$ and $f_2$ are homogeneous quasimorphisms and $f_1\sim f_2$. Then for every $g\in G$ and every $n\in \mathbb{N}$ we get that $\lvert f_1(g^n) - f_2(g^n)\rvert=n\lvert f_1(g) - f_2(g)\rvert<\infty$. So we have $\lvert f_1(g) - f_2(g)\rvert=0$ for every $g\in G$, so $f_1(g)=f_2(g)$.
\end{enumerate}
\end{proof}

In particular \ref{homquasi} shows us that different homomorphisms are in different equivalence classes. So in a slight abuse of notation we can consider $\mathrm{Hom}(G,\mathbb{R})$ as a subspace of $\tilde{\mathcal{Q}}(G)$. We will do so several times throughout this article. Now we finally explain the connection between quasimorphisms and (bounded) group cohomology.

\begin{prop}\label{iso}
The map \[\tilde{q}: \tilde{\mathcal{Q}}(G)/\mathrm{Hom}(G,\mathbb{R}) \xrightarrow{\sim} \ker(c^2), [f] \mapsto [\partial^1 f]_b\] is an isomorphism.
\end{prop}

\begin{proof}
We consider the map $q:\mathcal{Q}(G)\to \ker(c^2), f\mapsto [\partial^1 f]_b$. The statement follows if we can prove that $q$ is well defined, surjective and $\ker q=\mathrm{Map}_b(G,\mathbb{R})\oplus \mathrm{Hom}(G,\mathbb{R})$.

We start by proving that $q$ is well defined. Let $f\in \mathcal{Q}(G)$. Then $\partial^1 f$ is bounded, i.e. $\partial^1 f\in \mathrm{Map}_b(G^2,\mathbb{R})$. We have $\partial^2\circ\partial^1=0$, so $\partial^1 f\in \ker(\partial^2)$. So we have $[\partial^1 f]_b\in H_b^2(G,\mathbb{R})$. We get $c^2([\partial^1 f]_b)=[\partial^1 f]=0\in H^2(G,\mathbb{R})$, so $q$ is well defined.

We now prove that $q$ is surjective. Let $[f]_b\in \ker (c^2)$. Then $[f]=0\in H^2(G,\mathbb{R})$, so there is $h\in \mathrm{Map}(G,\mathbb{R})$ such that $\partial^1 h=f$. We get $\lVert \partial^1h \rVert_\infty=\lVert f \rVert_\infty<\infty$, since $f\in \mathrm{Map}_b(G^2,\mathbb{R})$. So $h\in \mathcal{Q}(G)$ and we get $q(h)=[f]_b$.

To finish the proof, we have to show that $\ker q=\mathrm{Map}_b(G,\mathbb{R})\oplus \mathrm{Hom}(G,\mathbb{R})$. We start by proving $\ker q\subset \mathrm{Map}_b(G,\mathbb{R})\oplus \mathrm{Hom}(G,\mathbb{R})$. Let $f\in \ker q$. Then $[\partial^1 f]_b=0$, so there is $f_1\in \mathrm{Map}_b(G,\mathbb{R})$ such that $\partial^1 f_1=\partial^1 f$. We have $\partial^1 (f-f_1)=0$ and $\ker (\partial^1)=\mathrm{Hom}(G,\mathbb{R})$, so there is $f_2\in \mathrm{Hom}(G,\mathbb{R})$ such that $f=f_1+f_2$. Only $\mathrm{Map}_b(G,\mathbb{R})\oplus \mathrm{Hom}(G,\mathbb{R})\subset \ker q$ is left to prove. Let $f=f_1+f_2$ with $f_1\in \mathrm{Map}_b(G,\mathbb{R})$ and $f_2\in \mathrm{Hom}(G,\mathbb{R})$. Then $q(f)=[\partial^1 f]_b=[\partial^1 f_1]_b+[\partial^1 f_2]_b=0\in H^2_b(G,\mathbb{R})$.
\end{proof}

Following in this direction, we can find a group theoretic interpretation of $H^2(G,\mathbb{R})$ (see \cite{primer}). But since we are interested in the case $G=F_n$ this is not necessary for us:

\begin{exmp}
We have $H_b^2(F_n,\mathbb{R})=\ker(c^2)=\tilde{\mathcal{Q}}(F_n)/\mathrm{Hom}(F_n,\mathbb{R})$.
\end{exmp}

Note that our characterization of the kernel $c^2$ is not enough to establish that $c^2$ is not injective: We have not proved yet that there are homogeneous quasimorphisms, which are not homomorphisms. But this is the case if $G=F_2$ as we will see in the Example \ref{noninj} in Section \ref{countfctquasi}.

The bounded cohomology of degree higher than 2 is still widely mysterious. Very little is known about these spaces: For free groups and degree 3 a notable exception is \cite{thirddegree}.

\subsection{The action of the outer automorphisms on (bounded) group cohomology}

The $\mathrm{Out}(G)$ action on (bounded) group cohomology can be looked at from the topological, the categorical and the combinatorial point of view. We will use the combinatorial point of view because of its relation to quasimorphisms.

\subsubsection*{Group cohomology}

We have an action $\mathrm{Aut}(G)\curvearrowright \mathrm{Map}(G^\bullet,\mathbb{R})$ by $X.f(g_0,...,g_n)=f(X^{-1}g_0,...,X^{-1}g_n)$. Using the inhomogeneous bar complex and \begin{equation}\label{equivariant}
\partial (f\circ X^{-1})=(\partial f) \circ X^{-1},
\end{equation} this gives us an action $\mathrm{Aut}(G)\curvearrowright H^\bullet(G,\mathbb{R})$.

\begin{prop}
The action $\mathrm{Aut}(G)\curvearrowright H^\bullet(G,\mathbb{R})$ factors through $\mathrm{Out}(G)$.
\end{prop}

\begin{proof}
See Section II.6 in \cite{Brown}.
\end{proof}

Unfortunately this action is not very interesting if $G=F_n$ is a free group.

\begin{exmp}
Let $G=F_n$ be the free group with $n$ generators. We get an action \begin{equation}\mathrm{Out}(F_n)\curvearrowright H^k(F_n,\mathbb{R})=\begin{cases}
\mathbb{R},& \text{if } k=0\\
\mathbb{R}^n,& \text{if } k=1\\
0,& \text{if } k>1\\
\end{cases}.\end{equation} This does not give us any information on $\mathrm{Out}(F_n)$, since \begin{enumerate}
\item $\mathrm{Out}(F_n)\curvearrowright H^0(F_n,\mathbb{R})=\mathbb{R}$ is the trivial action.
\item $\mathrm{Out}(F_n)\curvearrowright H^1(F_n,\mathbb{R})=\mathbb{R}^n$ factors through $GL_n(\mathbb{Z})\curvearrowright \mathbb{R}^n$.
\end{enumerate}
\end{exmp}

Fortunately things get a lot more interesting looking at bounded cohomology.

\subsubsection*{Bounded group cohomology}

We get the $\mathrm{Out}(G)$ action on bounded cohomology $H_b^k(G,\mathbb{R})$ for any group $G$ and $k\in \mathbb{N}_0$ in an analogous way. In this article we want to study this action in the case $G=F_n$ and $k=2$. Here we can use the description of $H_b^2(F_n,\mathbb{R})$ by equivalence classes of quasimorphisms to study the $\mathrm{Out}(F_n)$ action.

\begin{lem}
The $\mathrm{Aut}(F_n)$ action on $\tilde{\mathcal{Q}}(F_n)$ by $X[f]=[f\circ X^{-1}]$ factors through the action of $\mathrm{Out}(F_n)$.
\end{lem}

\begin{proof}
Let $w\in F_n$ and denote the corresponding inner automorphism by $X_w$. Let $[f]\in \tilde{\mathcal{Q}}(F_n)$. Then \begin{equation}\begin{split}
\lVert f\circ X_w^{-1}-f \rVert_\infty =& \sup_{v\in F_n} \lvert f(X_w^{-1}v)-f(v)\rvert\\
=& \sup_{v\in F_n} \lvert f(wvw^{-1})-f(v)\rvert\\
\leq& \sup_{v\in F_n} \lvert f(w)+f(vw^{-1})-f(v)\rvert+D(f)\\
\leq& \sup_{v\in F_n} \lvert f(w)+f(w^{-1})\rvert+2D(f)\\
<& \infty
\end{split}\end{equation}
So $X_w[f]=[f]$ for every $w\in F_n$ and $[f]\in \tilde{\mathcal{Q}}(F_n)$.
\end{proof}

\begin{prop}\label{sameaction}
The action of $\mathrm{Out}(F_n)$ on $H_b^2(F_n,\mathbb{R})$ is isomorphic to the $\mathrm{Out}(F_n)$ action on $\tilde{\mathcal{Q}}(F_n)/\mathrm{Hom}(F_n,\mathbb{R})$
by $X([f]+\mathrm{Hom}(F_n,\mathbb{R}))=[f\circ X^{-1}]+\mathrm{Hom}(F_n,\mathbb{R})$.
\end{prop}

\begin{proof}
The isomorphism given in \ref{iso} is $\mathrm{Out}(F_n)$ equivariant by \ref{equivariant}.
\end{proof}

\subsection{The outer automorphisms of the free group}

While we do not need a lot of background regarding $\mathrm{Out}(F_n)$, we need that it is generated by Nielsen transformations. Different authors mean different things when using the term \emph{Nielsen transformation}, so we will clarify what we mean now. Throughout this article we will always stick to the following convention:

\begin{defn}\label{Nielsen}
Let $F_n$ be the free group on the alphabet $S=\{a_1,a_2,...,a_n\}$, i.e. the collection of all reduced words over the extended alphabet $\bar{S}=\{a_1,...,a_n,a_1^{-1},...,a_n^{-1}\}$. We call an element $X\in \mathrm{Aut}(F_n)$ a \emph{Nielsen transformation} if it is one of the following:\begin{align}
P_1(a_1,...,a_n)&=(a_2,a_1,...,a_n),\\
P_2(a_1,...,a_n)&=(a_2,...,a_n,a_1),\\
H(a_1,...,a_n)&=(a_1^{-1},a_2,...,a_n),\\
T(a_1,...,a_n)&=(a_1a_2,a_2,...,a_n).
\end{align}
\end{defn}

\begin{thm}\label{NielsenOut}
The group of outer automorphism $\mathrm{Out}(F_n)$ is generated as a group by the Nielsen transformations.
\end{thm}

\begin{proof}
See \cite{Nielsen} for the original proof in German. For a proof in English see any textbook on combinatorial group theory, for example \cite{LS}.
\end{proof}

\section{Counting functions and counting quasimorphisms (up to bounded distance)}\label{countfctquasi}

\subsection{Counting functions}

Let $F_n$ be the free group on the alphabet $S=\{a_1,a_2,...,a_n\}$, i.e. the collection of all reduced words over the extended alphabet $\bar{S}=\{a_1,...,a_n,a_1^{-1},...,a_n^{-1}\}$. We fix the generating set $S$ for the rest of this article. Let $\lvert v\rvert_S$ be the word length for $v\in F_n$ with respect to $S$.

\begin{defn}%HT%
A reduced word $s_1...s_k\in F_n$ is a \emph{reduced subword} of a reduced word $r_1...r_l\in F_n$ if there exists $j\in \{1,...,l-k\}$ such that $s_i=r_{j+i}$ for all $i\in \{1,...,l\}$. 
\end{defn}

\begin{defn}%HT%
Given a reduced word $v=s_1...s_k\in F_n\setminus\{e\}$ and a reduced word $w=r_1...r_l\in F_n$, let $\#v(w)$ denote the number of $j\in \{1,...,l-k\}$ such that $s_i=r_{j+i}$ for all $i\in \{1,...,k\}$. This defines for every $v\in F_n\setminus\{e\}$ a function $\#v:F_n \to \mathbb{N}$, which we call the $v$-\emph{counting function}.
\end{defn}

\begin{exmp}
Let $F_2$ be the free group on the alphabet $S=\{a,b\}$. Then we get for example the following values for the $aba$-counting function:\begin{enumerate}
\item $\#aba(aba^3)=1$
\item $\#aba(aba^3b^{-4}a^2ba)=2$
\item $\#aba(a^3baba)=2$
\end{enumerate}
\end{exmp}

$\#v(w)$ counts how often $v$ appears as a reduced subword of $w$. As the third example shows, these reduced subwords are allowed to overlap.

\begin{defn}
Let $\mathcal{C}(F_n,S)$ be the space of real-valued functions on $F_n$ spanned by the counting functions $\{\#v\}_{v\in F_n\setminus\{e\}}$ with respect to the alphabet $S$.
\end{defn}

\begin{rem}
The space $\mathcal{C}(F_n,S)$ is defined a bit differently in \cite{HT}. But both definitions are equivalent.
\end{rem}

\begin{defn}%HT%
We call a finitely supported real-valued function $\alpha: F_n\setminus \{e\} \to \mathbb{R}$ a \emph{weight} on $F_n\setminus\{e\}$.
\end{defn}

It is clear that every element $f\in \mathcal{C}(F_n,S)$ can be written as $f=\sum_{v\in I}\alpha(v)\#v$ for some finite set $I\subset F_n\setminus\{e\}$ and some weight $\alpha$ that is supported on $I$. We now show now this representation of $f$ can be made unique.

\begin{prop}
The counting functions $\{\#v\}_{v\in F_n\setminus\{e\}}$ form a basis for $\mathcal{C}(F_n,S)$.
\end{prop}

\begin{proof}
The counting functions $\{\#v\}_{v\in F_n\setminus\{e\}}$ span $\mathcal{C}(F_n,S)$ by definition. So we have to prove that they are linearly independent. Assume that $0=\sum_{v\in I}\alpha(v)\#v$, where $\alpha$ is a weight supported on $I$. Choose $v_0\in I$ such that $\lvert v_0\rvert_S=\min_{v\in I} \lvert v\rvert_S$. Then $\sum_{v\in I}\alpha(v)\#v(v_0)=\alpha(v_0)\neq 0$. This is a contradiction.
\end{proof}

It follows that we can write every element $f\in \mathcal{C}(F_n,S)$ uniquely as $f=\sum_{v\in I}\alpha(v)\#v$, where $\alpha$ is a weight on $F_n\setminus\{e\}$ and $I=\mathrm{supp}(\alpha)$. In the following we will usually abandon this uniqueness in favour of more flexibility. So if we write $f=\sum_{v\in I}\alpha(v)\#v$, we only demand that $\mathrm{supp}(\alpha)\subset I$ unless stated otherwise.

\begin{defn}
We define the \emph{length} of an element $f=\sum_{v\in I}\alpha(v)\#v \in \mathcal{C}(F_n,S)$ by \[\lVert f\rVert_S= \max\{\lvert v\rvert_S\mid \alpha(v)\neq 0\}.\] If $\alpha(v)=0$ for all $v\in F_n\setminus \{e\}$, we set $\lVert f\rVert_S=0$.
\end{defn}

We will now see that length gives us a norm on $\mathcal{C}(F_n,S)$. Actually we get even more: We call a norm an \emph{ultrametric norm} if it induces an ultrametric distance. We denote by $\lvert -\rvert_0$ the trivial absolute value on $\mathbb{R}$, i.e. \begin{equation}
\lvert x\rvert_0=\begin{cases} 0,\text{if}\ x=0\\
1,\text{otherwise}\end{cases}.\end{equation}

\begin{prop}\label{normV}
Endow $\mathbb{R}$ with the trivial absolute value $\lvert -\rvert_0$. Then $\lVert -\rVert_S$ is an ultrametric norm on $\mathcal{C}(F_n,S)$.
\end{prop}

\begin{proof}
Let $f=\sum_{v\in I}\alpha(v)\#v$ and $g=\sum_{v\in J}\beta(v)\#v$. Then we have: \begin{enumerate}
\item We show that $\lVert -\rVert_S$ is absolutely homogeneous. If $a\neq 0$, then \begin{equation}\begin{split}\lVert af\rVert_S
&=\lVert \sum_{v\in I}a\alpha(v)\#v\rVert_S\\
&=\max\{\lvert v\rvert_S\mid v\in I, a\alpha(v)\neq 0\}\\
&=\max\{\lvert v\rvert_S\mid v\in I, \alpha(v)\neq 0\}\\
&=\lVert f\rVert_S\\
&=\lvert a\rvert_0\lVert f\rVert_S.\end{split}\end{equation} If $a=0$, then \begin{equation}\begin{split}\lVert af\rVert_S
&=\max\{\lvert v\rvert_S\mid v\in I, 0\alpha(v)\neq 0\}\\
&=0\\
&=\lvert a\rvert_0\lVert f\rVert_S.\end{split}\end{equation}
\item We show that the ultrametric triangle inequality holds for $\lVert -\rVert_S$. \begin{equation}\label{ultra}\begin{split}\lVert f+g\rVert_S
&=\lVert \sum_{v\in I\cup J}(\alpha(v)+\beta(v))\#v\rVert_S\\
&=\max\{\lvert v\rvert_S\mid v\in I\cup J, \alpha(v)+\beta(v)\neq 0\}\\
&\leq \max\{\lvert v\rvert_S\mid v\in I\cup J, \alpha(v)\neq 0 or \beta(v)\neq 0\}\\
&=\max(\lVert f\rVert_S,\lVert g\rVert_S).\end{split}\end{equation} So in particular $\lVert f+g\rVert_S \leq \lVert f\rVert_S+\lVert g\rVert_S.$
\item We show that $\lVert f\rVert_S=0$ if and only if $f=0$. This follows since both statements are equivalent to $\alpha(v)=0$ for all $v\in I$.
\end{enumerate}
\end{proof}

Note that $\lVert -\rVert_S$ takes values in the natural numbers $\mathbb{N}$. In particular the norm $\lVert -\rVert_S$ induces the discrete topology on $\mathcal{C}(F_n,S)$: Let $A\subset \mathcal{C}(F_n,S)$ be a subset and let $\{a_k\}\subset A$ be a sequence such that $a_k \to a\in \mathcal{C}(F_n,S)$. Then $\lVert a-a_k\rVert_S \to 0$. Since the values of $\lVert -\rVert_S$ are in $\mathbb{N}$, we get that there is $k_0 \in \mathbb{N}$ such that $\lVert a-a_k\rVert_S = 0$ for $k>k_0$. This shows $a\in A$, so $A$ is closed.

%HT%
How should we think of counting functions and their length? Let $T_n$ be the right Cayley tree of $F_n$ with respect to the generating set $S$. We think of $T_n$ as an edge coloured rooted tree with root $e$, where edges are coloured by elements of $\bar{S}$. An element $f=\sum_{v\in I}\alpha(v)\#v \in \mathcal{C}(F_n,S)$ corresponds to an vertex labeling of $T_n$ by the weight $\alpha$. We can visualize $f$ by drawing the weighted subtree spanned by $I\cup\{e\}$.

\begin{exmp}
Let $F_2$ be the free group on the alphabet $S=\{a,b\}$. Then we can visualize $f=5\#aa-3\#ab+\#b$ by
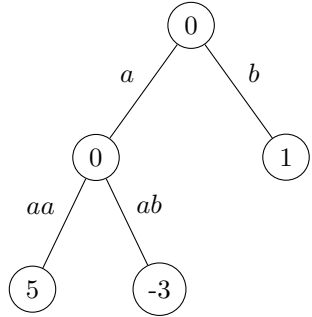
\begin{center}
\begin{tikzpicture} [every tree node/.style={draw,circle},
   level distance=1.75cm,sibling distance=1cm, 
   edge from parent path={(\tikzparentnode) -- (\tikzchildnode)}]
\Tree [.\node {0}; 
    \edge node[auto=right] {$a$};
    [.0  
        \edge node[auto=right] {$aa$}; [.5  ] 
        \edge node[auto=left] {$ab$}; [.-3 ] 
    ]
    \edge node[auto=left] {$b$};
    [.1
    ] ]
\end{tikzpicture}
\end{center}
\end{exmp}

We can visualize the length of an element $f \in \mathcal{C}(F_n,S)$ as the maximal distance of the root $e$ to a vertex not labeled $0$ in the Cayley tree $T_n$. It is the maximal depth of a vertex in the weighted subtree corresponding to $f$.

\begin{exmp}
Again let $f=5\#aa-3\#ab+\#b$. We have $\lVert f\rVert_S=2$. Both the vertex $aa$ and $ab$ are of distance $2$ from the root, the vertex $b$ is of distance $1$.
\end{exmp}

\subsection{Counting functions up to bounded distance}

\begin{defn}
We define an equivalence relation on $\mathcal{C}(F_n,S)$ by \[f_1\sim f_2\ \text{if and only if}\ \lVert f_1-f_2 \rVert_\infty < \infty.\] We denote by $\tilde{\mathcal{C}}(F_n,S)=\mathcal{C}(F_n,S)/\sim$ the space of equivalence classes and we denote by $\pi:\mathcal{C}(F_n,S)\to \tilde{\mathcal{C}}(F_n,S)$ the natural projection. We call $f\in \ker(\pi)$ a \emph{relation} and $\ker(\pi)$ the \emph{space of relations}.
\end{defn}

\begin{lem}\label{kerclosedsubspace}
The space of relations is a closed linear subspace of $(\mathcal{C}(F_n,S),\lVert -\rVert_S)$.
\end{lem}

\begin{proof}
Since $\lVert -\rVert_S$ induces the discrete topology on $\mathcal{C}(F_n,S)$, it is enough to prove that $\ker(\pi)$ is a linear subspace.\begin{enumerate}
\item Let $f\in \ker(\pi)$ and $a\in \mathbb{R}$. Then $\lVert f\rVert_\infty < C_f<\infty$ and $\lvert a\rvert<\infty$, where $\lvert -\rvert$ is the usual norm on $\mathbb{R}$. It follows that $\lVert af\rVert_\infty\leq \lvert a\rvert\lVert f\rVert_\infty < \lvert a\rvert C_f<\infty$. So $af\in \ker(\pi)$.
\item Let $f,g \in \ker(\pi)$. Then $\lVert f\rVert_\infty < C_f<\infty$ and $\lVert g\rVert_\infty < C_g<\infty$. It follows that $\lVert f+g\rVert_\infty\leq \lVert f\rVert_\infty+\lVert g\rVert_\infty < C_f+C_g<\infty$. So $f+g \in \ker(\pi)$.
\end{enumerate}
\end{proof}

Now we introduce a new notion of length adapted to our equivalence relation.

\begin{defn}
We define the \emph{reduced length} of an element $f\in \mathcal{C}(F_n,S)$ by \[\lvert f\rvert_S=\min\{\lVert f+h\rVert_S\mid h\in \ker(\pi)\}=\min\{\lVert g\rVert_S\mid f\sim g\}.\] We call $f \in \mathcal{C}(F_n,S)$ \emph{reduced}, if $\lvert f\rvert_S=\lVert f\rVert_S$. 
\end{defn}

In the next subsection we will discuss how to check if an element $f \in \mathcal{C}(F_n,S)$ is reduced. Since we do not know this yet, we have a hard time giving examples of reduced elements $f \in \mathcal{C}(F_n,S)$ at this stage. But we can consider elements with very low length.

\begin{exmp}\label{0red}
Let $f \in \mathcal{C}(F_n,S)$ have $\lVert f\rVert_S=0$. Then $f=0$ and $\lvert f\rvert_S=0$. So the element $f=0 \in \mathcal{C}(F_n,S)$ is reduced.
\end{exmp}

\begin{exmp}\label{1red}
Let $f \in \mathcal{C}(F_n,S)$ have $\lVert f\rVert_S=1$. Then $f=\sum_{v\in \bar{S}}\alpha(v)\#v$. Let $v_0\in I$ such that $\alpha(v_0)\neq 0$. Then $f(v_0^k)=k\alpha(v_0)$ for every $k\in \mathbb{N}$, so $\lvert f\rvert_S>0$. If follows that $\lvert f\rvert_S=1$. So every element $f \in \mathcal{C}(F_n,S)$ with $\lVert f\rVert_S=1$ is reduced.
\end{exmp}

\begin{cor}
Endow $\mathbb{R}$ with the trivial absolute value $\lvert -\rvert_0$. Then $\lvert -\rvert_S$ is a seminorm on $\mathcal{C}(F_n,S)$.
\end{cor}

\begin{proof}
This follows from \ref{normV} and \ref{kerclosedsubspace}.
\end{proof}

We can use this seminorm on $\mathcal{C}(F_n,S)$ to get a norm on $\tilde{\mathcal{C}}(F_n,S)$.

\begin{defn}
We define the \emph{length} of an element $[f] \in \tilde{\mathcal{C}}(F_n,S)$ by \[\lvert [f]\rvert_S=\lvert f\rvert_S.\]
\end{defn}

\begin{cor}\label{normontildeC}
Endow $\mathbb{R}$ with the trivial absolute value $\lvert -\rvert_0$. Then $\lvert -\rvert_S$ is an ultrametric norm on $\tilde{\mathcal{C}}(F_n,S)$.
\end{cor}

\begin{proof}
It follows from \ref{normV} and \ref{kerclosedsubspace} that $\lvert -\rvert_S$ is a norm on $\tilde{\mathcal{C}}(F_n,S)$. So it remains to show that $\lvert -\rvert_S$ induces an ultrametric distance on $\tilde{\mathcal{C}}(F_n,S)$: Let $f_1, f_2 \in \mathcal{C}(F_n,S)$. Then there are reduced elements $g_1, g_2 \in \mathcal{C}(F_n,S)$ such that $f_1\sim g_1$ and $f_2\sim g_2$. We get using \ref{ultra}:
\begin{equation}\begin{split}\lvert f_1+f_2\rvert_S
&=\min\{\lVert h\rVert_S\mid f_1+f_2\sim h\}\\
&\leq \min\{\lVert h_1+h_2\rVert_S\mid f_1\sim h_1, f_2\sim h_2\}\\
&\leq \lVert g_1+g_2\rVert_S\\
&\leq \max(\lVert g_1\rVert_S,\lVert g_2\rVert_S)\\
&=\max (\lvert [f_1]\rvert_S, \lvert [f_2]\rvert_S).\end{split}\end{equation}
\end{proof}

As in the subsection before we see that $\lvert -\rvert_S$ induces the discrete topology on $\tilde{\mathcal{C}}(F_n,S)$.

\subsection{Relations and recognizing reduced elements}

Visualizing elements of $\tilde{\mathcal{C}}(F_n,S)$ is not as straightforward as visualizing elements of $\mathcal{C}(F_n,S)$. In the same spirit as in the previous subsection we can think of $[f]$ as an equivalence class of weighted subtrees of $T_n$. While this gives us some intuition, at this stage we do not even have a way to decide if two weighted subtrees of $T_n$ represent the same equivalence class. Hartnick and Talambutsa solve this problem in \cite{HT}: They develop an algorithm to decide if $[f]=0\in \tilde{\mathcal{C}}(F_n,S)$. So for two elements $f,g \in \mathcal{C}(F_n,S)$ we can decide if $[f]=[g]\in \tilde{\mathcal{C}}(F_n,S)$ by applying the algorithm to $[f-g]\in \tilde{\mathcal{C}}(F_n,S)$. One major step towards this algorithm in \cite{HT} is the introduction of a condition on $f$ that guarantees $[f]\neq 0$. In this subsection we will show that the condition on $f$ in \cite{HT} even guarantees that $f$ is reduced. Then we will give a cruder criterion to show that $f$ is reduced, which is easier to check. This cruder criterion for reducedness will be one of our main tools in the rest of the article.

One key insight in \cite{HT} is a very good understanding of $\ker(\pi)$, which makes it possible to pass from one element $f \in \mathcal{C}(F_n,S)$ to an equivalent element $g \in \mathcal{C}(F_n,S)$. We will mention these results now, since they are also useful in the following sections.

\begin{defn}%HT%
Let $w=w_1...w_k\in F_n\setminus \{e\}$ be a reduced word. \begin{enumerate}[(i)]
\item Then we call $l_w=\#w-\sum_{s\in \bar{S}\setminus \{v_1^{-1}\}} \#sw$ a \emph{left-extension relation}.
\item Then we call $r_w=\#w-\sum_{s\in \bar{S}\setminus \{v_k^{-1}\}} \#ws$ a \emph{right-extension relation}.
\end{enumerate}
\end{defn}

\begin{lem}
Let $w=w_1...w_k\in F_n\setminus \{e\}$ be a reduced word. Then $l_w$ and $r_w$ are relations.
\end{lem}

\begin{proof}
See the proof of Lemma 2.2 in \cite{HT}.
\end{proof}

In the sections \ref{Tsingle}, \ref{Tsums} and \ref{fixpts} we will often want to pass from one element $f \in \mathcal{C}(F_n,S)$ to an equivalent element $g \in \mathcal{C}(F_n,S)$, which has better properties. We will always use only iterated left- and right-extension relations to do so. This is not a coincidence.

\begin{thm}[Calegari-Walker, Hartnick-Talambutsa]\label{relations}
The space of relations $\ker (\pi)$ is spanned by the set of left- and right-extension relation $\bigcup_{w\in F_n\setminus \{e\}} \{l_w,r_w\}$.
\end{thm}

\begin{proof}
See the proof of Theorem 1.3 in \cite{HT}. See also the preprint version of \cite{CalegariWalker}.
\end{proof}

Now we want to explain the sufficient condition on $f$ for $[f]\neq 0$ given in \cite{HT}. The following definitions are different from \cite{HT}, since Hartnick and Talambutsa actually work on the weighted trees which we only use for visualization, but they are essentially equivalent.

\begin{defn}
Let $v,w \in F_n$ and assume $\lvert v\rvert_S=\lvert w\rvert_S\geq 2$.\begin{enumerate}
\item We call $v$ and $w$ \emph{right brothers} if $\lvert v^{-1}w\rvert_S=2$.
\item We call $v$ and $w$ \emph{left brothers} if $\lvert vw^{-1}\rvert_S=2$.
\end{enumerate}
We also introduce a notation for the sets of right- and left-brothers.\begin{enumerate}
\item We call $\mathrm{rBr}(v)=\{w\in F_n \mid v,w\ \text{are right brothers}\}$ the \emph{right brotherhood} of $v$.
\item We call $\mathrm{lBr}(v)=\{w\in F_n \mid v,w\ \text{are left brothers}\}$ the \emph{left brotherhood} of $v$.
\end{enumerate}
\end{defn}

If $v,w$ are right brothers, the reduced words representing $v$ and $w$ differ only by the last letter. Right brothers are easy to visualize in the rooted right Cayley tree $T_n$. (Actually what we call a visualization is the definition of \emph{brothers} in \cite{HT}.) Two vertices $v$ and $w$ of $T_n$ are right-brothers if the vertices on the unique geodesic between $v$ and $e$ excluding $v$ are exactly the vertices on the unique geodesic between $w$ and $e$ excluding $w$. If $v,w$ are left brothers, the reduced words representing $v$ and $w$ differ only by the first letter. Left brothers are harder to visualize in the right-Cayley tree $T_n$. One can visualize them exactly as right brothers in the left Cayley tree of $F_n$.

\begin{defn}
Let $f=\sum_{v\in I} \alpha(v)\#v \in \mathcal{C}(F_n,S)$ and assume that $\lVert f\rVert_S\geq 2$. Then $f$ is called \emph{unbalanced} if there is $v_0\in I$ with $\lvert v_0\rvert_S=\lVert f\rVert_S$ and $\alpha(v_0)\neq 0$ such that \begin{enumerate}
\item There is $v_1 \in \mathrm{rBr}(v_0)$ such that $\alpha(v_1)\neq \alpha(v_0)$.
\item There is $v_2 \in \mathrm{lBr}(v_0)$ such that $\alpha(v_2)=0$ and $\alpha(w)=0$ for all $w\in \mathrm{rBr}(v_2)$.
\end{enumerate}
\end{defn}

\begin{rem}
Let $f=\sum_{v\in I} \alpha(v)\#v \in \mathcal{C}(F_n,S)$. Being unbalanced is a property of the value of $\alpha$ on words $v\in I$ with maximal word length $\lvert v\rvert_S=\lVert f\rVert_S$.
\end{rem}

The term \emph{unbalanced} comes from the visualization of $f\in \mathcal{C}(F_n,S)$. We will give two examples and hope this gives the intuition.

\begin{exmp}
Let $F_2$ be the free group on the alphabet $S=\{a,b\}$. Then for example $f=5\#aa-3\#ab+\#b$ is unbalanced. Its visualization is
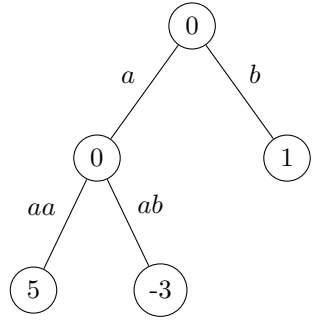
\begin{center}
\begin{tikzpicture} [every tree node/.style={draw,circle},
   level distance=1.75cm,sibling distance=1cm, 
   edge from parent path={(\tikzparentnode) -- (\tikzchildnode)}]
\Tree [.\node {0}; 
    \edge node[auto=right] {$a$};
    [.0  
        \edge node[auto=right] {$aa$}; [.5  ] 
        \edge node[auto=left] {$ab$}; [.-3 ] 
    ]
    \edge node[auto=left] {$b$};
    [.1
    ] ]
\end{tikzpicture}
\end{center}
\end{exmp}

\begin{exmp}
Let $F_2$ be the free group on the alphabet $S=\{a,b\}$. Then for example $f=\#a+4\#b^{-1}+5\#ab-2\#a^{-1}b-2\#ba+\#bb+\#b^{-1}a$ is balanced. Its visualization is
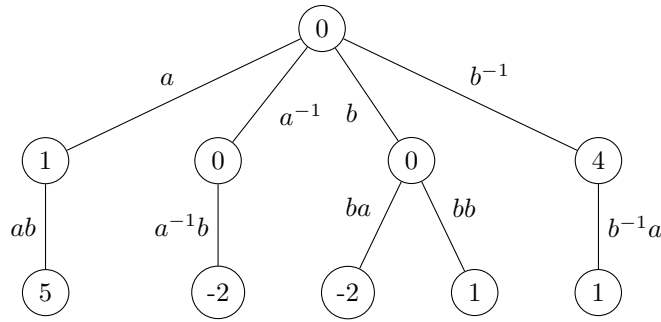
\begin{center}
\begin{tikzpicture} [every tree node/.style={draw,circle},
   level distance=1.75cm,sibling distance=1cm, 
   edge from parent path={(\tikzparentnode) -- (\tikzchildnode)}]
\Tree [.\node {0}; 
    \edge node[auto=right] {$a$};
    [.1  
        \edge node[auto=right] {$ab$}; [.5 ] 
    ]
    \edge node[auto=left] {$a^{-1}$};
    [.0
    		\edge node[auto=right] {$a^{-1}b$}; [.-2 ] 
    ]
    \edge node[auto=right] {$b$};
    [.0  
        \edge node[auto=right] {$ba$}; [.-2 ]
        \edge node[auto=left] {$bb$}; [.1 ] 
    ]
    \edge node[auto=left] {$b^{-1}$};
    [.4  
    		\edge node[auto=left] {$b^{-1}a$}; [.1 ]
    ] ]
\end{tikzpicture}
\end{center}
\end{exmp}

\begin{thm} [Hartnick, Talambutsa]\label{unbalancedtrivial}
Every unbalanced element in $\mathcal{C}(F_n,S)$ represents a non-trivial element in $\tilde{\mathcal{C}}(F_n,S)$.
\end{thm}

\begin{proof}
See the proof of Theorem 4.2 in \cite{HT}.
\end{proof}

\begin{cor}\label{unbalancedreduced}
Every unbalanced element in $\mathcal{C}(F_n,S)$ is reduced.
\end{cor}

\begin{proof}
Actually the proof of Theorem 4.2 in \cite{HT} already shows this implicitly, but we can also deduce this statement from \ref{unbalancedtrivial}: Assume $f=\sum_{v\in I} \alpha(v)\#v$ is unbalanced, but not reduced. Then there is $g=\sum_{v\in J} \beta(v)\#v$ such that $[f]=[g]$ and $\lVert g\rVert_S<\lVert f\rVert_S$. But $f-g$ is still unbalanced, since $\alpha(v)-\beta(v)=\alpha(v)$ for all $v$ such that $\lvert v\rvert_S=\lVert f\rVert_S$. By the theorem $[f-g]\neq 0$, so $[f]\neq [g]$. This is a contradiction.
\end{proof}

At this point we want to desribe the algorithm of \cite{HT} without going into any details: The algorithm starts with an element $f\in \mathcal{C}(F_n,S)$ of which we want to decide if $[f]\neq 0$. After every step $k$ we get an element $g_k\in \mathcal{C}(F_n,S)$ such that $[f]=[g_k]$. The algorithm ends in one of the three following ways:\begin{enumerate}
\item The element $g_k\in \mathcal{C}(F_n,S)$ is unbalanced.
\item The element $g_k\in \mathcal{C}(F_n,S)$ has $\lVert g_k\rVert_S=1$.
\item The element $g_k\in \mathcal{C}(F_n,S)$ is $0$.
\end{enumerate} In the first two cases the algorithm returns that $[f]=[g_k]\neq 0$, in the last case it returns $[f]=0$. By \ref{unbalancedreduced} and \ref{1red}, we can conclude that $g_k$ is actually reduced if the algorithm ends after $k$ steps. So the algorithm does not only decide if $[f]=0$, but it finds a reduced element in $[f]$. In general an equivalence class $[f]$ has more than one reduced representative.

Now we will introduce a rather crude criterion to check if $f \in \mathcal{C}(F_n,S)$ is unbalanced. This will be our main tool to prove that $f \in \mathcal{C}(F_n,S)$ is reduced in the rest of this article.

\begin{defn}
Let $w=s_1...s_k$ be a reduced word in $F_n$ and let $s\in S$. Then $w$ is called \emph{$s$-truncated} if $s_1,s_k\in \bar{S}\setminus\{s,s^{-1}\}$.
\end{defn}

\begin{defn}
Let $w=s_1...s_k$ be a reduced word in $F_n\setminus \{e\}$ and let $s\in S$. We define the \emph{truncation operator} $\tau_s$ the following way:\begin{enumerate}
\item Assume that $s_i\in \{s,s^{-1}\}$ for $1\leq i\leq k$. Then we set $\tau_s(w)=\varnothing$. 
\item Assume there is $1\leq i\leq k$ such that $s_i\notin \{s,s^{-1}\}$. Let \begin{align} i_0&=\min \{1\leq j\leq k\mid s_j\notin \{s,s^{-1}\}\}\\
i_1&=\max \{1\leq j\leq k\mid s_j\notin \{s,s^{-1}\}\}.\end{align} Then we set $\tau_s(w)=s_{i_0}...s_{i_1}$
\end{enumerate} In particular $w$ is $s$-truncated if $\tau_s(w)=w$.
\end{defn}

\begin{defn}
Given a finite set of words $I\subset F_n\setminus \{e\}$, we define: \[n_s(I)=\max \{\lvert v\rvert_S\mid v\in I, \tau_s(v)\neq v\}.\]  The number $n_s(I)$ is equal to the maximal length of $v\in I$, which is not $s$-truncated. We call $E_s(I)=\{v\in I\mid \lvert v\rvert_S>n_b(I)\}$ the \emph{s-truncated end} of $I$.
\end{defn} 

Using these definitions, we can finally formulate the promised criterion for reducedness.

\begin{prop}\label{redtree}
Let $f=\sum_{v\in I} \alpha(v)\#v$ be a sum of counting functions. Assume there exists $w\in E_s(I)$ such that $\alpha(w)\neq 0$. Then $f$ is reduced. In particular $\lvert f\rvert_S\geq \lvert w\rvert_S$.
\end{prop} 

To prove Proposition \ref{redtree} we want to show that $f=\sum_{v\in I} \alpha(v)\#v$ is unbalanced if there is $w\in E_s(I)$ such that $\alpha(w)\neq 0$. We need the following lemma.

\begin{lem}\label{lrbrotherstrunc}
Every $s$-truncated word $w$ with $\lvert w\rvert_S\geq 2$ has a left and a right brother, which are not $s$-truncated.
\end{lem}

\begin{proof}
Let $w=s_1s_2...s_{k-1}s_k$ be a $s$-truncated word with $k\geq 2$. Since $w$ is $s$-truncated, we know that $s_1,s_k\in \bar{S}\setminus\{s,s^{-1}\}$.\begin{enumerate}
\item If $s_{k-1}\neq s$, then $v=s_1s_2...s_{k-1}s^{-1}$ is a reduced word and not $s$-truncated. We have $\lvert v\rvert_S=k=\lvert w\rvert_S$ and since $s_k\neq s^{-1}$ we get that $\lvert w^{-1}v\rvert_S=\lvert s_k^{-1}s^{-1}\rvert_S=2$. So $w$ and $v$ are right brothers.
\item If $s_{k-1}\neq s^{-1}$, then $v=s_1s_2...s_{k-1}s$ is a reduced word and not $s$-truncated. We have $\lvert v\rvert_S=k=\lvert w\rvert_S$ and since $s_k\neq s$ we get that $\lvert w^{-1}v\rvert_S=\lvert s_k^{-1}s\rvert_S=2$. So $w$ and $v$ are right brothers.
\item If $s_2\neq s$, then $v=s^{-1}s_2...s_{k-1}s_k$ is a reduced word and not $s$-truncated. We have $\lvert v\rvert_S=k=\lvert w\rvert_S$ and since $s_1\neq s^{-1}$ we get that $\lvert wv^{-1}\rvert_S=\lvert s_1s\rvert_S=2$. So $w$ and $v$ are left brothers.
\item If $s_2\neq s^{-1}$, then $v=ss_2...s_{k-1}s_k$ is a reduced word and not $s$-truncated. We have $\lvert v\rvert_S=k=\lvert w\rvert_S$ and since $s_1\neq s$ we get that $\lvert wv^{-1}\rvert_S=\lvert s_1s^{-1}\rvert_S=2$. So $w$ and $v$ are left brothers.
\end{enumerate}
\end{proof} 

Now we can prove Proposition \ref{redtree} using the lemma.

\begin{proof}
Since $E_s(I)\neq \varnothing$, we can take $w_{max}\in E_s(I)$ such that $\alpha(w_{max})\neq 0$ and $\lvert w_{max} \rvert_S=\lVert f\rVert_S$. If $\lvert w_{max}\rvert_S=1$, then $\lVert f \rVert=1$, so $f$ is reduced by \ref{1red}. We can thus assume that $\lvert w_{max}\rvert_S=k\geq 2$. By \ref{lrbrotherstrunc} we know that $w_{max}$ has a right brother $w_1$ such that $w_1$ is not $s$-truncated. Since $w_{max}\in E_s(I)$, we get that $w_1\notin I$ and therefore that $\alpha(w_1)=0\neq \alpha(w_{max})$. Again by \ref{lrbrotherstrunc} we know that $w_{max}$ has a left brother $w_2$ such that $w_2$ is not $s$-truncated, thence $w_2^{-1}v$ is reduced for every reduced, $s$-truncated word $v$. So every right brother of $w_2$ is not $s$-truncated. Since $w_{max}\in E_s(I)$, it follows that $\alpha(w_2)=0$ and $\alpha(v)=0$ for all $v\in \mathrm{rBr}(w_2)$. We have proved that $f$ is unbalanced. By \ref{unbalancedreduced} it follows that $f$ is reduced. We get $\lvert f\rvert_S=\lvert w_{max}\rvert_S$.
\end{proof} 

We end this subsection by mentioning an easy property of $n_s(I)$ for future refence:

\begin{lem}\label{nsmax}
Let $I$ and $J$ be finite sets in $F_n\setminus\{e\}$. Then \begin{equation}
n_s(I\cup J)=\max (n_s(I),n_s(J)).
\end{equation}
\end{lem}

\begin{proof}
This follows by an easy computation: \begin{equation}\begin{split} n_s(I\cup J)&=\max \{\lvert v\rvert_S\mid v\in I\cup J, \tau_s(v)\neq v\}\\
&=\max (\max\{\lvert v\rvert_S\mid v\in I, \tau_s(v)\neq v\},\max \{\lvert v\rvert_S\mid v\in J, \tau_s(v)\neq v\})\\
&=\max (n_s(I),n_s(J))\end{split}\end{equation} 
\end{proof}

\subsection{Counting quasimorphisms (up to bounded distance)}

This whole section so far dealt with sums of counting functions and their equivalence classes. How does this help us to study the action of $\mathrm{Out}(F_n)$ on $H^2_b(F_n, \mathbb{R})$? One of the many reasons to care about counting functions is the following classical result:

\begin{prop}[Brooks]
The symmetrized counting functions $\phi v=\#v-\#v^{-1}$ for $v\in F_n\setminus \{e\}$ are quasimorphisms on $F_n$.
\end{prop}

\begin{proof}
Let $v$ be a reduced word in $F_n\setminus \{e\}$ and let $x$, $y$ be reduced words in $F_n$.
If $w=xy$ is a reduced word, then \begin{equation}
\lvert \#v(w)-\#v(x)-\#v(y)\rvert\leq \lvert v\rvert_S.
\end{equation} So we get \begin{equation}
\lvert \phi v(w)-\phi v(x)-\phi v(y)\rvert\leq 2\lvert v\rvert_S<\infty.
\end{equation} If $w=xy$ is not a reduced word, we do not have necessarily that $\lvert \#v(w)-\#v(x)-\#v(y)\rvert$ is bounded. (Otherwise counting functions would be quasimorphisms.) But we get that there are reduced word $x_1$, $y_2$ in $F_n$ and a reduced word $z$ in $F_n\setminus \{e\}$ such that $x=x_1z$, $y=z^{-1}y_2$ and $w=x_1y_2$ are reduced. Then \begin{equation}\begin{split}
\lvert \phi v(w)-\phi v(x)-\phi v(y)\rvert=& \lvert \phi v(x_1y_2)-\phi v(x_1z)-\phi v(z^{-1}y_2)\rvert\\
=& \lvert \phi v(x_1y_2)-\phi v(x_1)-\phi v(y_2)\\
&-\phi v(x_1z)+\phi v(x_1)+\phi v(z)\\
&-\phi v(z^{-1}y_2)-\phi v(z)+\phi v(y_2)\rvert\\
\leq& 2\lvert v\rvert_S+2\lvert v\rvert_S+2\lvert v\rvert_S\\
<&\infty.\end{split}\end{equation} So $\phi v$ is a quasimorphism.
\end{proof}

\begin{defn}
The quasimorphism $\phi_v$ for $v\in F_n\setminus \{e\}$ is called the $v$-\emph{counting (or Brooks) quasimorphism}. We denote by $\mathcal{B}(F_n,S)$ the space of real-valued functions on $F_n$ spanned by the counting quasimorphisms $\{\phi_v\}_{v\in F_n\setminus \{e\}}$. We define an equivalence relation on $\mathcal{B}(F_n,S)$ by \[f_1\sim f_2\ \text{if and only if}\ \lVert f_1-f_2 \rVert_\infty < \infty.\] We denote by $\tilde{\mathcal{B}}(F_n,S)=\mathcal{B}(F_n,S)/\sim$ the space of equivalence classes and call it the \emph{Brooks space}.
\end{defn}

\begin{exmp}\label{noninj}
Let $F_2$ be the free group on the alphabet $S=\{a,b\}$. Then the homogenization of the counting quasimorphism $\phi ab$ is not a homomorphism. We have $\tilde{\phi ab}(a)= \lim_{n}\frac{\phi ab(a^n)}{n}=0$ and $\tilde{\phi ab}(b)= \lim_{n}\frac{\phi ab(b^n)}{n}=0$, but $\tilde{\phi ab}(ab)= \lim_{n}\frac{\phi ab((ab)^n)}{n}=1$.
\end{exmp}

We can consider $\mathcal{B}(F_n,S)$ as a closed linear subspace of $\mathcal{C}(F_n,S)$ by identifying $\mathcal{B}(F_n,S)$ with $\{f\in \mathcal{C}(F_n,S)\mid f(w)=-f(w^{-1})\ \text{for all}\ w\in F_n\}$. This way we can see $\tilde{\mathcal{B}}(F_n,S)$ as a closed linear subspace of $\tilde{\mathcal{C}}(F_n,S)$. Therefore we can apply the definitions and results from the previous subsections of this section. In particular we obtain a norm $\lVert -\rVert_S$ on $\mathcal{B}(F_n,S)$ and a norm $\lvert -\rvert_S$ on $\tilde{\mathcal{B}}(F_n,S)$. A sum of counting quasimorphisms is reduced, if it is reduced as an element of $\mathcal{C}(F_n,S)$. We can use our Criterion \ref{redtree} to recognize reduced sums of counting quasimorphisms.

On the other hand we can consider $\mathcal{B}(F_n,S)$ as a closed linear subspace of $\mathcal{Q}(F_n)$ and $\tilde{\mathcal{B}}(F_n,S)$ as a closed linear subspace of $\tilde{\mathcal{Q}}(F_n)$. 

\begin{lem}\label{hom=1}
We have $\mathrm{Hom}(F_n,\mathbb{R})=\{f\in \mathcal{B}(F_n,S)\mid \lVert f\rVert_S\leq 1\}$.
\end{lem}

\begin{proof}
We prove $\mathrm{Hom}(F_n,\mathbb{R})\subset\{f\in \mathcal{B}(F_n,S)\mid \lVert f\rVert_S\leq 1\}$ first. Let $\psi \in \mathrm{Hom}(F_n,\mathbb{R})$. Then $\psi$ is determined by its values on $S$. We have $\psi=\sum_{s\in S}\psi(s)\phi s$ and therefore $\psi\in \mathcal{B}(F_n,S)$ and $\lVert \psi \rVert_S\leq 1$. Now we prove $\mathrm{Hom}(F_n,\mathbb{R})\supset\{f\in \mathcal{B}(F_n,S)\mid \lVert f\rVert_S\leq 1\}$. Let $f\in \mathcal{B}(F_n,S)$ and assume that $\lVert f\rVert_S\leq 1$. Then we can write $f=\sum_{s\in \bar{S}}\alpha(s)\phi s$. Since $\phi s=\phi s^{-1}$, we have $f=\sum_{s\in S}\beta(s)\phi s$ with $\beta(s)=\alpha(s)-\alpha(s^{-1})$. So $f\in \mathrm{Hom}(F_n,\mathbb{R})$.
\end{proof}

By the slight abuse of notation we justified in \ref{homquasi} we can consider $\mathrm{Hom}(F_n,\mathbb{R})$ as a subspace of $\tilde{\mathcal{B}}(F_n,S)$. This way we can define $$H^2_b(F_n, \mathbb{R})_{\mathrm{fin}}=\tilde{\mathcal{B}}(F_n,S)/\mathrm{Hom}(F_n,\mathbb{R})\subset H^2_b(F_n,\mathbb{R}).$$ We obtain a norm on this subspace of $H^2_b(F_n,\mathbb{R})$ from the norm of the Brooks space.

\begin{cor}
We have $\mathrm{Hom}(F_n,\mathbb{R})=\{[f]\in \tilde{\mathcal{B}}(F_n,S)\mid \lvert [f]\rvert_S\leq 1\}$.
\end{cor}

\begin{proof}
This follows from \ref{hom=1} and \ref{relations}.
\end{proof}

\begin{defn}
Let $[f]+\mathrm{Hom}(F_n,\mathbb{R})\in H^2_b(F_n, \mathbb{R})_{\mathrm{fin}}$. Then we define \begin{equation}
\lvert [f]+\mathrm{Hom}(F_n,\mathbb{R})\rvert_S=\begin{cases}
\lvert [f]\rvert_S,& \text{if } \lvert [f]\rvert_S\geq 2\\
0,& \text{otherwise}
\end{cases}.\end{equation} 
\end{defn}

\begin{cor}\label{H2bnorm}
Endow $\mathbb{R}$ with the trivial absolute value $\lvert -\rvert_0$. Then $\lvert -\rvert_S$ is an ultrametric norm on $H^2_b(F_n, \mathbb{R})_{\mathrm{fin}}$.
\end{cor}

\begin{proof}
This follows from \ref{normontildeC} and the fact that $\mathrm{Hom}(F_n,\mathbb{R})$ is a closed linear subspace of $\tilde{\mathcal{B}}(F_n,S)$.
\end{proof}

\section{The action of the outer automorphisms}\label{OutFnaction}

We saw in \ref{sameaction} that we can study the $\mathrm{Out}(F_n)$ action on $H_b^2(F_n, \mathbb{R})$ by studying the $\mathrm{Out}(F_n)$ action on $\tilde{\mathcal{Q}}(F_n)$. In the first subsection of this section we will follow \cite{HS} to show that the $\mathrm{Out}(F_n)$ action on $\tilde{\mathcal{Q}}(F_n)$ leaves the Brooks space $\tilde{\mathcal{B}}(F_n)\subset \tilde{\mathcal{Q}}(F_n)$ invariant. We will give a reason why it is sensible to restrict our attention to the action of $\mathrm{Out}(F_n)$ on the Brooks space $\tilde{\mathcal{B}}(F_n)$, which we will do from the second subsection onwards. In the second subsection we will introduce the notion of \emph{speed} of an element $X\in \mathrm{Out}(F_n)$ on an element $[f]$ of the Brooks space, which measures how fast the norm of $[f]$ grows asymptotically under repeated application of the element $X$. This is an interesting object in its own right and we will study it for $X=T^{-1}$ in the next sections. It will also be our main tool to answer Ab\'{e}rt's question \ref{47} for the Brooks space in Section \ref{fixpts}.

\subsection{Invariance of the Brooks space}\label{InvBrooks}

The main result of this subsection is the following:

\begin{thm}\label{Outinv}
The Brooks space $\tilde{\mathcal{B}}(F_n,S)$ is invariant under the action of $\mathrm{Out}(F_n)$ on $\tilde{\mathcal{Q}}(F_n)$ by $X[f]=[f\circ X^{-1}]$.
\end{thm}

This is Theorem 2.3 in \cite{HS}. Since we will use later many of the techniques appearing in the proof of this theorem, we will reproduce the proof now. This subsection rephrases pages 8 to 12 in \cite{HS}.

By \ref{NielsenOut} we know that $\mathrm{Out}(F_n)$ is generated as a group by the Nielsen transformations $\{P_1,P_2,H,T\}$. Note that $P_1$, $P_2$ and $H$ are of finite order, so if $X\in \{P_1,P_2,H\}$ we can write $X^{-1}=X^k$ for some $k\in \mathbb{N}$. Note also that $T=P_1HP_1T^{-1}P_1HP_1$. So it is enough to prove that the Brooks space is invariant under the action of $\{P_1,P_2,H\}$ and $T^{-1}$.

\begin{lem}\label{Xinv}
If $X\in \{P_1,P_2,H\}$, then $\phi w\circ X^{-1}=\phi (Xw)$.
\end{lem}

\begin{proof}
This follows from the definitions in \cite{Nielsen}.
\end{proof}

Now we concentrate on $T^{-1}$. (Essentially this is true for the rest of this article.) Since $a_1$ and $a_2$ play a special role for $T^{-1}$, we set $a=a_1$ and $b=a_2$. Also for convenience we define $S_b=\bar{S}\setminus \{b,b^{-1}\}$ for the rest of this article. A last useful introduction is the following representation of elements in $F_n$. Every reduced word $w\in F_n$ can be written as $w=b^{m_0}s_1b^{m_1}...s_kb^{m_k}$, where the following conditions are satified:\begin{enumerate}
\item $k\geq 0$,
\item $m_j \in \mathbb{Z}$ for all $0\leq j \leq k$,
\item $s_j \in S_b$ for all $1\leq j \leq k$,
\item $s_j\neq s_{j+1}^{-1}$ if $m_j=0$ for all $1\leq j \leq k-1$.
\end{enumerate} We use the following notation to denote this representation \[w\equiv_S (m_0,s_1,m_1,...,s_k,m_k).\]
If $w\equiv_S (m_0,s_1,m_1,...,s_k,m_k)$, we denote its \textit{$b$-length} by $\lvert w \rvert_b=k$.

Computations with $T$ and $T^{-1}$ are a lot easier in this representation: Let $w\in F_n$ be a reduced word. We can write $w\equiv_S (m_0,s_1,m_1,...,s_k,m_k)$. We set $s_0=s_{k+1}=e$ and define \begin{equation}\label{mj+}
m_j^+=m_j+\#a(s_j)-\#a^{-1}(s_{j+1}).
\end{equation} for all $0\leq j \leq k$. Then by the definition of $T$ we get \begin{equation}\label{Tw}Tw\equiv_S (m_0^+,s_1,...,s_k,m_k^+).\end{equation} In the same vein we can define \begin{equation}\label{mj-}
m_j^-=m_j-\#a(s_j)+\#a^{-1}(s_{j+1}).
\end{equation} for all $0\leq j \leq k$. Then by the definition of $T^{-1}$ we get \begin{equation}\label{T-1w}T^{-1}w\equiv_S (m_0^-,s_1,...,s_k,m_k^-).\end{equation} 

\begin{prop}\label{btrsubwords}
Another very convenient feature of this representation is that it simplifies describing $b$-truncated subwords of $w$ a lot: For $1\leq i\leq j\leq k$ we can define the word \begin{equation}w_{i,j}\equiv_S (s_i,m_i,...,s_j).\end{equation} Then $w_{i,j}$ is a $b$-truncated subword of $w$. On the other hand if $v$ is a $b$-truncated subword of $w$, then $v=w_{i,j}$ for some $1\leq i\leq j\leq k$. In particular we get the following bijections between the $b$-truncated subwords of $w$ and the $b$-truncated subwords of $Tw$ and $T^{-1}w$:\begin{enumerate}[(i)]
\item $w_{i,j} \leftrightarrow (Tw)_{i,j}=\tau_b(T(w_{i,j}))\equiv_S (s_i,m_i^+,...,s_j)$
\item $w_{i,j} \leftrightarrow (T^{-1}w)_{i,j}=\tau_b(T^{-1}(w_{i,j}))\equiv_S (s_i,m_i^-,...,s_j)$
\end{enumerate}
\end{prop}

\begin{proof}
This follows from \ref{Tw} and \ref{T-1w}.
\end{proof}

With this preparation we prove now that $T^{-1}\phi w\in \tilde{\mathcal{B}}(F_n,S)$ if $w$ is not a power of $b$.

% Fiddling around with spacing
\newlength{\lena}
\settowidth{\lena}{$\{b^{m_k}\}\cup\{b^{m_k-1}s\mid s\in S_b\setminus\{a^{-1}\}\}$,}

\begin{lem}\label{wisTinv}
Let $w\equiv_S (m_0,s_1,...,s_k,m_k)$ and assume that $w$ is not a power of $b$. Let $W_{1}(w)=W^{l}_1(w)\cdot W^{m}_1(w)\cdot W^{r}_1(w)$, where \begin{align*} 
W^{l}_1(w)&=\begin{cases}
\mathmakebox[\lena][l]{\{e\},}& \text{if } m_0=0\\
\{b^{m_0}\}\cup\{ab^{m_0-1}\},& \text{if } m_0>0, s_1\neq a^{-1}\\
\{b^{m_0+1}\}\cup\{ab^{m_0}\},& \text{if } m_0>0, s_1=a^{-1}\\
\{b^{m_0-1}\}\cup\{sb^{m_0}\mid s\in S_b\setminus\{a\}\},& \text{if } m_0<0, s_1\neq a^{-1}\\
\{b^{m_0}\}\cup\{sb^{m_0+1}\mid s\in S_b\setminus\{a\}\},& \text{if } m_0<0, s_1=a^{-1}
\end{cases},\\
W^{m}_1(w)&=\{\tau_b\circ T^{-1}\circ \tau_b(w)\},\\
W^{r}_1(w)&=\begin{cases}
\{e\},& \text{if } m_k=0\\
\{b^{m_k}\}\cup\{b^{m_k-1}s\mid s\in S_b\setminus\{a^{-1}\}\},& \text{if } m_k>0, s_k=a\\
\{b^{m_k+1}\}\cup\{b^{m_k}s\mid s\in S_b\setminus\{a^{-1}\}\},& \text{if } m_k>0, s_k\neq a\\
\{b^{m_k-1}\}\cup\{b^{m_k}a^{-1}\},& \text{if } m_k<0, s_k=a\\
\{b^{m_k}\}\cup\{b^{m_{l+1}}a^{-1}\},& \text{if } m_k<0, s_k\neq a
\end{cases}.
\end{align*}

Then \begin{equation}
\#w\circ T=\sum_{u\in W_{1}(w)}\#u.
\end{equation}
\end{lem}

\begin{proof}
Let $v\equiv_S (n_0,r_1,...,r_l,n_l)$ be a word in $F_n$. We will prove the the result in the following two steps:\begin{enumerate}
\item $\#w(Tv)\leq \sum_{u\in W_{1}(w)}\#u(v)$
\item $\sum_{u\in W_{1}(w)}\#u(v)\leq \#w(Tv)$
\end{enumerate}

We start with the first step. If $w$ is a reduced subword of $Tv$, then $\tau_b(w)$ is a reduced subword of $Tv$. Since $\tau_b(w)$ is $b$-truncated, we know that $\tau_b(w)$ is a reduced subword of $Tv$ if and only if there is $1\leq i\leq l-k$ such that $Tv_{i,i+k}=\tau_b(w)$. We can not necessarily find $w$ as a reduced subword, when we see $Tv_{i,i+k}=\tau_b(w)$. By \ref{Tw} we know that $Tv\equiv_S (n_0^+,r_1,...,r_k,n_k^+)$. So if we can extend $Tv_{i,i+k}=\tau_b(w)$ to $w$ or not depends on the values of $n_{i-1}^+$ and $n_{i+k}^+$. We get that $\#w(Tv)$ is the number of $1\leq i\leq l-k$ such that\begin{enumerate}
\item $Tv_{i,i+k}=\tau_b(w)\equiv_S (s_1,...,s_k)$,
\item $n_{i-1}^+\cdot m_0>0$ and $\lvert n_{i-1}^+\rvert > \lvert m_0\rvert$,
\item $n_{i+k}^+\cdot m_k>0$ and $\lvert n_{i+k}^+\rvert > \lvert m_k\rvert$.
\end{enumerate} By \ref{btrsubwords} the first condition is satisfied if and only if \begin{equation}
v_{i,i+k}=\tau_b(T^{-1}(\tau_b(w)))\equiv_S (s_1,m_i^-,...,s_k).
\end{equation} In other words, the first condition is satisfied if and only if $\{v_{i,i+k}\}=W^{m}_1(w)$. We set $r_0=r_{l+1}=e$ and get by \ref{mj+} that\begin{align}
n_{i-1}^+=& n_i+\#a(r_i)-\#a^{-1}(s_{1})\\
n_{i+k}^+=& n_{i+k}+\#a(s_k)-\#a^{-1}(r_{i+k+1}).
\end{align} This gives us that the second and third condition are satisfied if and only if $r_{i-1}b^{n_{i-1}}\in W^{l}_1(w)$ and $b^{n_{i+k}}r_{i+k+1}\in W^{r}_1(w)$ respectively. So whenever we find $w$ as a reduced subword of $Tv$, we find an element of $W_{1}(w)$ as a reduced subword of $v$. This proves $\#w(Tv)\leq \sum_{u\in W_{1}(w)}\#u(v)$.

We now prove the second step. If $u\in W_{1}(w)$ is a reduced subword of $v$, then by definition of $W_{1}(w)$, we get that $\tau_b\circ T^{-1}\circ \tau_b(w)$ is a reduced subword of $v$. So there is $1\leq i\leq l-k$ such that $v_{i,i+k}=\tau_b\circ T^{-1}\circ \tau_b(w)$. By turning the argument in the first step upside down we find that the three conditions above are satisfied. So we find $w$ as a reduced subword of $Tv$. It remains to show that different reduced subwords $u_1\neq u_2 \in W_{1}(w)$ of $v$, give us different occurrences of $w$ as a reduced subword of $Tv$, namely different $1\leq i_1,i_2\leq l-k$. This follows since no element $u_1\in W_{1}(w)$ is a reduced subword of another element $u_2\in W_{1}(w)$.
\end{proof}

\begin{cor}\label{phiwisTinv}
>From the Lemma \ref{wisTinv} we can read off that:\begin{enumerate}[(i)]
\item If $w\in F_n$ is not a power of $b$, then $T^{-1}\phi w=\sum_{u\in W_{1}(w)}\phi u$.
\item $T^{-1}$ is a bijection on the subspace of $\tilde{\mathcal{B}}(F_n,S)$ spanned by $\{[\phi v]\mid \tau_b(v)=v\}$.
\end{enumerate}
\end{cor}

To finish the proof of \ref{Outinv}, we have to show that $\phi b^m\circ T \in \tilde{\mathcal{B}}(F_n,S)$ for $m\in \mathbb{Z}\setminus \{0\}$.

\begin{lem}\label{powerofbtob}
Consider $\phi b^m$ for $m\in \mathbb{Z}\setminus \{0,1\}$. Then there is a weight $\alpha$ and a finite set $J$ such that no element of $J$ is a power of $b$ and \begin{equation}
\phi b^m\sim \alpha(b)\phi b + \sum_{y\in J}\alpha(y)\phi y.\end{equation}
\end{lem}

\begin{proof}
Since $\phi b^m=-\phi b^{-m}$, it is enough to prove the statement for $m>1$. We get by applying $m-1$ right-extension relations that \begin{equation}\phi b^{m}\sim\phi b-\sum_{i=1}^{m-1}\sum_{s\in S_b}\phi b^is.\end{equation}
\end{proof}

We know that $T^{-1}\phi b=\phi a+\phi b$, so we get that $\phi b^m\circ T\in \tilde{\mathcal{B}}(F_n,S)$ for $m\in \mathbb{Z}$. This finishes the proof of \ref{Outinv}.

Note that we actually did more than prove the invariance of the Brooks space $\tilde{\mathcal{B}}(F_n,S)$: We gave an explicit way to find an representative of $X[f]$ for every $[f]\in \tilde{\mathcal{B}}(F_n,S)$ and a set of $X$'s that generate $\mathrm{Out}(F_n)$ as a semigroup. So we can give a representative of $X[f]$ for $[f]\in \tilde{\mathcal{B}}(F_n,S)$ and $X \in \mathrm{Out}(F_n)$. Unfortunately this is very much a theoretical possibility for a general $X$: As it is easy to see, computations involving \ref{wisTinv} are not pleasant. Nevertheless we will use this procedure for special elements (namely powers of $T^{-1}$) in the following sections. Now we want to give a surprising consequence of \ref{Outinv}:

\begin{cor}
$\tilde{\mathcal{B}}(F_n,S)$ is independent of $S$.
\end{cor}

\begin{proof}
To prove this statement, we have to consider different free generating sets of $F_n$. Since our usual notation relies on one fixed free generating set, we have to adjust it for the sake of this proof. Let $v\in F_n\setminus \{e\}$ and let $S_i$ be a free generating set of $F_n$. Then we denote the $v$-counting quasimorphism with respect to $S_i$ by $\phi^{S_i}v$.

Now let $S_1$ and $S_2$ be different free generating sets of $F_n$. Then there is an element $X\in \mathrm{Out}(F_n)$ such that $XS_1=S_2$. Let $v\in F_n\setminus \{e\}$ and $w\in F_n$. Then we have \begin{equation}
\phi^{S_1}w(v)=\phi^{XS_1}_{Xw}(Xv)=\phi^{S_2}_{Xw}(Xv).
\end{equation} So we get that $\phi^{S_1}w\in \tilde{\mathcal{B}}(F_n,S_2)$ by \ref{Outinv}.
\end{proof}

This result makes it unnecessary to consider the generating set $S$ when talking about the Brooks space. So from now on we write $\tilde{\mathcal{B}}(F_n)=\tilde{\mathcal{B}}(F_n,S)$.

\begin{cor}
The subspace $H^2_b(F_n, \mathbb{R})_{\mathrm{fin}}\subset H_b^2(F_n,\mathbb{R})$ is invariant under the action of $\mathrm{Out}(F_n)$ on $H_b^2(F_n,\mathbb{R})$.
\end{cor}

\begin{proof}
This follows from \ref{sameaction} and \ref{Outinv}.
\end{proof}

>From now on we will concentrate on understanding the $\mathrm{Out}(F_n)$ action on $\tilde{\mathcal{B}}(F_n)$. By \ref{Outinv} this is possible and by \ref{dense} it is also reasonable: While we reduce the difficulty of the problem, we do not reduce scope of the problem by too much: 

\begin{thm}[Grigorchuk]\label{dense}
The space $\tilde{\mathcal{B}}(F_n)$ is dense in $\tilde{\mathcal{Q}}(F_n)$ with respect to the topology of pointwise convergence of homogeneous representatives. In particular the action of $\mathrm{Out}(F_n)$ on $H_b^2(F_n,\mathbb{R})$ is uniquely determined by the action of $\mathrm{Out}(F_n)$ on the space $H^2_b(F_n, \mathbb{R})_{\mathrm{fin}}$.
\end{thm}

\begin{proof}
See Theorem 2.3 in Hartnick-Schweitzer \cite{HS}, which relies on \cite{Grigorchuk}.
\end{proof}

This theorem shows that there is reasonable hope of extending any results we find in the next sections for $\tilde{\mathcal{B}}(F_n)$ to a result on $\tilde{\mathcal{Q}}(F_n)$. However we will not attempt to carry out this extension step in this article.

\subsection{Linear speed}\label{linspeed}

It is hard to compare $X[f]$ with $[f]$, where $X\in \mathrm{Out}(F_n)$ and $[f]\in \tilde{\mathcal{B}}(F_n)$. There are two ideas behind the notion of \emph{speed}, which we will introduce in this subsection:\begin{enumerate}
\item Instead of comparing elements, we compare the length of elements in $\tilde{\mathcal{B}}(F_n)$. 
\item Instead of comparing the length of $X[f]$ with the length of $[f]$, we will consider the asymptotic behaviour of the length of $X^n[f]$ under repeated applications of $X$.
\end{enumerate}

Since we look in the following sections at the speed of the special element $T^{-1}$, we concentrate here on \emph{linear speed}. But speed can be defined in a far more general fashion, which we give in the Appendix \ref{speed}.

\begin{defn} %Knuth
We use the following notations given by Knuth in \cite{Knuth}: Let $f$ and $g$ be functions from $\mathbb{N}$ to $\mathbb{R}$. Then:\begin{enumerate}
\item We write $g(n)=O(f(n))$ if there is $n_0\in \mathbb{N}$ and $C\in \mathbb{R}_{+}$ such that $\lvert g(n)\rvert \leq Cf(n)$ for all $n>n_0$.
\item We write $g(n)=\Theta(f(n))$ if there is $n_0\in \mathbb{N}$ and $C, C'\in \mathbb{R}_{+}$ such that $Cf(n)\leq g(n)\leq C'f(n)$ for all $n>n_0$.
\end{enumerate} Let $X\in \mathrm{Out}(F_n)$ and $[f]\in \tilde{\mathcal{B}}(F_n)$. We say that $X$ has \emph{linear speed on $[f]$} if \[\lvert X^{n}[f]\rvert_S=\Theta (n).\] We say that $X$ has \emph{linear speed} if it satisfies:\begin{enumerate}
\item $\lvert X^{n}[f]\rvert_S=O(n)$ for all $[f]\in \tilde{\mathcal{B}}(F_n)$.
\item There is $[f_0]\in \tilde{\mathcal{B}}(F_n)$ such that $X$ has linear speed on $[f_0]$.
\end{enumerate} We do not only care about the type of speed of $X$, we actually want to measure it. For this goal we introduce the following: \begin{enumerate}
\item We call $\mathrm{usp}_S(X,[f])=\limsup_n \frac{\lvert X^{n}[f]\rvert_S}{n}$ the \emph{upper speed} of $X$ on $[f]$.
\item We call $\mathrm{lsp}_S(X,[f])=\liminf_n \frac{\lvert X^{n}[f]\rvert_S}{n}$ the \emph{lower speed} of $X$ on $[f]$.
\item Assume $\lim_n \frac{\lvert X^{n}[f]\rvert_S}{n}\in \mathbb{R}$. Then we call $\mathrm{sp}_S(X,[f])=\lim_n \frac{\lvert X^{n}[f]\rvert_S}{n}$ the \emph{speed} of $X$ on $[f]$ and say that \emph{the speed of $X$ on $[f]$ exists}.
\end{enumerate}
\end{defn}

These (possibly) different speeds all inherit some, but not the same, properties from $\lvert -\rvert_S$. In particular we will see that the upper speed is a far nicer object than the lower speed.

\begin{prop}\label{speedseminorm}
Endow $\mathbb{R}$ with the trivial absolute value $\lvert -\rvert_0$ and let $X\in \mathrm{Out}(F_n)$. Then we get:\begin{enumerate}[(i)]
\item Assume that $\mathrm{usp}_S(X,[f])<\infty$ for all $[f]\in \tilde{\mathcal{B}}(F_n)$. Then the upper speed $\mathrm{usp}_S(X,-)$ is a seminorm on $\tilde{\mathcal{B}}(F_n)$.
\item Assume that $\mathrm{sp}_S(X,[f])$ exists for all $[f]\in \tilde{\mathcal{B}}(F_n)$. Then the speed $\mathrm{sp}_S(X,-)$ is a seminorm on $\tilde{\mathcal{B}}(F_n)$.
\end{enumerate}
\end{prop}

\begin{proof}
It is enough to prove that $\mathrm{usp}_S(X,-)$ is absolutely homogeneous and satisfies the triangle inequality. We will show that it even satisfies the ultrametric triangle inequality.\begin{enumerate}
\item Let $[f]\in \tilde{\mathcal{B}}(F_n)$ and $a\in \mathbb{R}$. Then we get by the absolute homogeneity of $\lvert -\rvert_S$ on $\tilde{\mathcal{B}}(F_n)$ that
\begin{equation}\label{speedhomogeneous}
\mathrm{usp}_S(X,[af])=\limsup_n \frac{\lvert X^{n}[af]\rvert_S}{n}=\limsup_n \frac{\lvert a\rvert_0 \lvert X^{n}[f]\rvert_S}{n}=\lvert a\rvert_0\mathrm{usp}_S(X,[f]).
\end{equation}
\item Let $[f], [g]\in \tilde{\mathcal{B}}(F_n)$. Then we get that \begin{equation}\label{ultraspeed}\begin{split}\mathrm{usp}_S(X,[f+g])
&=\limsup_n \frac{\lvert X^{n}[f+g]\rvert_S}{n}\\
&=\limsup_n \frac{\lvert X^{n}[f]+X^{n}[g]\rvert_S}{n}\\
&\leq \limsup_n \frac{\max(\lvert X^{n}[f]\rvert_S,\lvert X^{n}[g]\rvert_S)}{n}\\
&=\sup (\limsup_n\frac{\lvert X^{n}[f]\rvert_S}{n},\limsup_n\frac{\lvert X^{n}[g]\rvert_S}{n})\\
&=\sup (\mathrm{usp}_S(X,[f]),\mathrm{usp}_S(X,[g])).\end{split}\end{equation} So  $\mathrm{usp}_S(X,[f+g])\leq \mathrm{usp}_S(X,[f])+\mathrm{usp}_S(X,[g])$.
\end{enumerate} 
\end{proof}

Note that this is in general not true for the lower speed. While the lower speed is absolutely homogeneous, we can not prove the triangle inequality. The difference lies in the following fact for sequences $a_k$ and $b_k$ of real numbers:\begin{align}
\limsup_k \max(a_k,b_k)=&\sup (\limsup_k a_k,\limsup_k b_k)\\
\liminf_k \max(a_k,b_k)\geq& \sup (\liminf_k a_k,\liminf_k b_k)
\end{align}

Knowing the speed of $X\in \mathrm{Out}(F_n)$ on $[f]\in \tilde{\mathcal{B}}(F_n)$ gives us information on $X$ and on $[f]$. In this article we will mainly acquire the latter kind of information. One example is the following:

\begin{cor}\label{speedgiveslength}
Let $X\in \mathrm{Out}(F_n)$ and assume that the speed of $X$ exists on $[f_1], [f_2]\in \tilde{\mathcal{B}}(F_n)$. If $\mathrm{sp}_S(X,[f_1])>\mathrm{sp}_S(X,[f_2])$, then for every constant $C\in \mathbb{R}$ there is $N\in \mathbb{N}$ such that $\lvert X^{n}[f_1]\rvert_S>\lvert X^{n}[f_2]\rvert_S+C$ for $n>N$.
\end{cor}

\begin{proof}
Assume otherwise. Then there is a strictly increasing sequence of natural numbers $a_k$ such that $\lvert X^{a_k}[f_1]\rvert_S\leq \lvert X^{a_k}[f_2]\rvert_S+C$. Then \begin{equation}\begin{split}sp(X,f_1)
&=\lim_n \frac{\lvert X^{n}[f_1]\rvert_S}{n}\\
&=\lim_k \frac{\lvert X^{a_k}[f_1]\rvert_S}{a_k}\\
&\leq\lim_k \frac{\lvert X^{a_k}[f_2]\rvert_S+C}{a_k}\\
&=\lim_k \frac{\lvert X^{a_k}[f_2]\rvert_S}{a_k}\\
&=\lim_n \frac{\lvert X^{n}[f_2]\rvert_S}{n}\\
&=sp(X,f_2).
\end{split}\end{equation} This is a contradiction.
\end{proof}

At this stage we would like to show some examples of the (lower or upper) speed of some $X\in \mathrm{Out}(F_n)$ on some $[f]\in \tilde{\mathcal{B}}(F_n)$. Unfortunately the only example in our reach is quite depressing.

\begin{lem}
If $X\in \mathrm{Out}(F_n)$ has finite order, then $\mathrm{lsp}_S(X,[f])=\mathrm{usp}_S(X,[f])=0$ for any $[f]\in \tilde{\mathcal{B}}(F_n)$.
\end{lem}

\begin{proof}
Since $X\in \mathrm{Out}(F_n)$ is of finite order, $\lvert X^{n}[f]\rvert_S$ is bounded for any $[f]\in \tilde{\mathcal{B}}(F_n)$. So the statement follows by the definition of the upper and lower speed of $X$ on $[f]$.
\end{proof}

We will prove in the Sections \ref{Tsingle} and \ref{Tsums} that the speed of $T^{-1}$ on $[f]$ exists for every $[f]\in \tilde{\mathcal{B}}(F_n)$ and present a way to compute it. In particular we will prove that $T^{-1}$ has linear speed. We now sketch how we will do this: How can we check if $T^{-1}$ has linear speed on some $[f]\in \tilde{\mathcal{B}}(F_n)$? We will in praxis use something a bit stronger.

\begin{defn}
Assume that the speed of $X$ on $[f]$ exists and $\mathrm{sp}_S(X,[f])>0$. Then we say that $X$ has \emph{strongly linear speed on $[f]$}.
\end{defn}

\begin{lem}
If $X$ has strongly linear speed on $[f]$, then $X$ has linear speed on $[f]$.
\end{lem}

\begin{proof}
Let $0<\epsilon<\mathrm{sp}_S(X,[f])$. Then there is $N\in \mathbb{N}$ such that for all $n>N$\begin{equation}\mathrm{sp}_S(X,[f])-\epsilon\leq \frac{\lvert X^{n}[f]\rvert_S}{n}\leq \mathrm{sp}_S(X,[f])+\epsilon.\end{equation} This implies for all $n>N$ \begin{equation}(\mathrm{sp}_S(X,[f])-\epsilon)n\leq \lvert X^{n}[f]\rvert_S\leq (\mathrm{sp}_S(X,[f])+\epsilon)n.\end{equation} So $X$ has linear speed on $[f]$.
\end{proof}

Of course this only shifts the question: How can we check that $T^{-1}$ has strongly linear speed on $[f]\in \tilde{\mathcal{B}}(F_n)$? In other words how can we check if the speed of $X$ on $[f]$ exists and how can we compute it, if it exists? We will do that in two steps: First we will consider the counting quasimorphism $[\phi w]$ for every $w\in F_n$ and get our hands dirty. This will happen in Section \ref{Tsingle}. In a second step in Section \ref{Tsums} we will reduce the case of sums of counting quasimorphisms $[f]$ to the case of a single quasimorphism $[\phi w]$. For a special set of elements in $\mathcal{B}(F_n,S)$ this reduction is very easy:

\begin{defn}
Let $f=\sum_{v\in I} \alpha(v)\phi v$ be a sum of counting quasimorphisms. Then $f$ is \emph{speed reduced} for $X$ if $\mathrm{sp}_S(X,[f])=\sup_{v\in \mathrm{supp}(\alpha)} \mathrm{sp}_S(X,[\phi v])$ and both sides of the equation exist.
\end{defn}

\begin{rem}
Note that being speed reduced is a property of the element $f\in \mathcal{B}(F_n,S)$ and not of the equivalence class $[f]\in \tilde{\mathcal{B}}(F_n)$.
\end{rem}

In the spirit of this definition we can reformulate the problem of Section \ref{Tsums}: We want to find a speed reduced element in every equivalence class. The following statement will be our tool to recognize speed reduced elements:

\begin{lem}\label{howprovespeedred}
Let $f=\sum_{v\in I} \alpha(v)\phi v$ be a sum of counting quasimorphisms. Then $f$ is speed reduced for $X$ if the speed of $X$ on $[\phi v]$ exists for all $v\in \mathrm{supp}(\alpha)$ and \begin{equation}
\mathrm{lsp}_S(X,[f])\geq\sup_{v\in \mathrm{supp}(\alpha)} \mathrm{usp}_S(X,[\phi v]).
\end{equation}
\end{lem}

\begin{proof}
Let $f=\sum_{v\in I} \alpha(v)\phi v$ and assume for convenience that $I=\mathrm{supp}(\alpha)$. Then we have by \ref{ultraspeed} that \begin{equation}\begin{split}
\mathrm{lsp}_S(X,[f])\leq & \mathrm{usp}_S(X,[f])\\
\leq & \sup_{v\in I} \mathrm{usp}_S(X,[\phi v])\\
\leq & \mathrm{lsp}_S(X,[f]).\end{split}\end{equation} Since the speed of $X$ on $[\phi v]$ exists for all $v\in I$, it follows that the speed of $X$ on $[f]$ exists and that $\mathrm{sp}_S(X,[f])=\sup_{v\in I} \mathrm{sp}_S(X,[\phi v])$. So $f$ is speed reduced for $X$.
\end{proof}

\section{The speed of $T^{-1}$ on a counting quasimorphism}\label{Tsingle}

In this section we consider the action of $T^{-1}$ on a single counting quasimorphism $[\phi w]$. If $w$ is not a power of $b$, we will show that the speed of $T^{-1}$ on $[\phi w]$ exists and give a formula to compute it. Our strategy is the following: First we find a good element in the equivalence class $T^{-n}[\phi w]$. We call this element the \emph{n-representative} of $[\phi w]$. The n-representative is good in two aspects: It is a reduced element of $\mathcal{B}(F_n,S)$ and we can say something about its length. Using this information we can say enough about the reduced length of $T^{-n}[\phi w]$ to compute the speed of $T^{-1}$ on $[\phi w]$.

\subsection{Finding the n-representative}

In this subsection we define the n-representative of a single counting quasimorphism $[\phi w]$, where $w$ is not a power of $b$.

\newlength{\lenc}
\settowidth{\lenc}{$\{b^{m_k+n}\}\cup\bigcup_{i=1}^n\{b^{m_k+i-1}s\mid s\in \bar{S}\setminus\{a^{-1},b,b^{-1}\}\}$,}

\begin{cor}\label{induction}
Let $w\equiv_S (m_0,s_1,...,s_k,m_k)$ assume that $w$ is not a power of $b$. Let $W^{\star}_{n}(w)=W^{l}_n(w)\cdot W^{m}_n(w)\cdot W^{r}_n(w)$, where

\begin{align*}
W^{l}_n(w)&=\begin{cases}
\mathmakebox[\lenc][l]{\{e\},}& \text{if } m_0=0\\
\{b^{m_0}\}\cup\bigcup_{i=1}^n\{ab^{m_0-i}\},& \text{if } m_0>0, s_1\neq a^{-1}\\
\{b^{m_0+n}\}\cup\bigcup_{i=1}^n\{ab^{m_0+i-1}\},& \text{if } m_0>0, s_1=a^{-1}\\
\{b^{m_0-n}\}\cup\bigcup_{i=1}^n\{sb^{m_0-i+1}\mid s\in \bar{S}\setminus\{a,b,b^{-1}\}\},& \text{if } m_0<0, s_1\neq a^{-1}\\
\{b^{m_0}\}\cup\bigcup_{i=1}^n\{sb^{m_0+i}\mid s\in \bar{S}\setminus\{a,b,b^{-1}\}\},& \text{if } m_0<0, s_1=a^{-1}
\end{cases},\\
W^{m}_n(w)&=\{(\tau_b\circ T^{-1}\circ \tau_b)^n(w)\},\\
W^{r}_n(w)&=\begin{cases}
\mathmakebox[\lenc][l]{\{e\},}& \text{if } m_k=0\\
\{b^{m_k}\}\cup \bigcup_{i=1}^n\{b^{m_k-i}s\mid s\in \bar{S}\setminus\{a^{-1},b,b^{-1}\}\},& \text{if } m_k>0, s_k=a\\
\{b^{m_k+n}\}\cup\bigcup_{i=1}^n\{b^{m_k+i-1}s\mid s\in \bar{S}\setminus\{a^{-1},b,b^{-1}\}\},& \text{if } m_k>0, s_k\neq a\\
\{b^{m_k-n}\}\cup\bigcup_{i=1}^n\{b^{m_k-i+1}a^{-1}\},& \text{if } m_k<0, s_k=a\\
\{b^{m_k}\}\cup \bigcup_{i=1}^n\{b^{m_k+i}a^{-1}\},& \text{if } m_k<0, s_k\neq a
\end{cases}.
\end{align*}
Then \begin{equation}
\phi w\circ T^n=\sum_{u\in W^{\star}_{n}(w)}\phi u.
\end{equation} 
\end{cor}

\begin{proof}
We prove this using induction on $n$. The base case $n=1$ is exactly \ref{phiwisTinv}.
Now we assume that the statement is true for $n$. We want to prove it for $n+1$. We know by assumption and the base case that \begin{equation}\begin{split}
\phi w\circ T^{n+1} &= \phi w\circ T^n\circ T\\
&= \sum_{u\in W^{\star}_{n}(w)}\phi u\circ T\\
&= \sum_{u\in W^{\star}_{n}(w)}\sum_{v\in W^{\star}_{1}(u)} \phi v
\end{split}\end{equation} So it is enough to prove \begin{equation}W^{\star}_{n+1}(w)=\bigcup_{u\in W^{\star}_{n}(w)}W^{\star}_{1}(u)\end{equation} This follows by considering all 25 possible cases for $W^{\star}_{n}$ separately.\end{proof}

The sum of counting quasimorphisms $\sum_{u\in W^{\star}_{n}(w)}\phi u$ is a representative of $T^{-n}[\phi w]$. But it is not a good representative: We can not use our Criterion \ref{redtree} to show reducedness. So we use left- and right-extension relations to find a better representative. In the following statement our notation relies on two conventions:\begin{enumerate}
\item The product of the empty set with another set is the empty set.
\item Any sum over the elements of the empty set is $0$.
\end{enumerate}

\newlength{\lend}
\settowidth{\lend}{$\{b^{m_k}\}\cup \bigcup_{i=1}^n\{b^{m_k-i}s\mid s\in S_b\setminus\{a^{-1}\}\}$,}

\begin{defn}\label{def:phiwn}
Let $w\equiv_S (m_0,s_1,...,s_k,m_k)$ and assume that $w$ is not a power of $b$. For $n\geq 1$ we define the sets\begin{enumerate}
\item $W^{+}_{n}(w)=L^{+}_{n}(w)\cdot M_{n}(w)\cdot R^{+}_{n}(w) \cup L^{-}_{n}(w)\cdot M_{n}(w)\cdot R^{-}_{n}(w)$ and
\item $W^{-}_{n}(w)=L^{+}_{n}(w)\cdot M_{n}(w)\cdot R^{-}_{n}(w) \cup L^{-}_{n}(w)\cdot M_{n}(w)\cdot R^{+}_{n}(w)$,\end{enumerate}where \begin{align*}
L^{+}_{n}(w)&=\begin{cases}
\mathmakebox[\lend][l]{\{e\},}& \text{if } m_0=0\\
\{b^{m_0}\}\cup\bigcup_{i=1}^n\{ab^{m_0-i}\},& \text{if } m_0>0, s_1\neq a^{-1}\\
\{b^{m_0}\},& \text{if } m_0>0, s_1=a^{-1}\\
\{b^{m_0}\},& \text{if } m_0<0, s_1\neq a^{-1}\\
\{b^{m_0}\}\cup\bigcup_{i=1}^n\{sb^{m_0+i}\mid s\in S_b\setminus\{a\}\},& \text{if } m_0<0, s_1=a^{-1}
\end{cases},\\
L^{-}_{n}(w)&=\begin{cases}
\mathmakebox[\lend][l]{\varnothing,}& \text{if } m_0=0\\
\varnothing,& \text{if } m_0>0, s_1\neq a^{-1}\\
\bigcup_{i=1}^{n}\{sb^{m_0+i-1}\mid s\in S_b\setminus\{a\}\},& \text{if } m_0>0, s_1=a^{-1}\\
\bigcup_{i=1}^{n}\{ab^{m_0-i+1}\},& \text{if } m_0<0, s_1\neq a^{-1}\\
\varnothing,& \text{if } m_0<0, s_1=a^{-1}
\end{cases},\\
M_{n}(w)&=\{(\tau_b\circ T^{-1}\circ \tau_b)^n(w)\},\\
R^{+}_{n}(w)&=\begin{cases}
\mathmakebox[\lend][l]{\{e\},}& \text{if } m_k=0\\
\{b^{m_k}\}\cup \bigcup_{i=1}^n\{b^{m_k-i}s\mid s\in S_b\setminus\{a^{-1}\}\},& \text{if } m_k>0, s_k=a\\
\{b^{m_k}\},& \text{if } m_k>0, s_k\neq a\\
\{b^{m_k}\},& \text{if } m_k<0, s_k=a\\
\{b^{m_k}\}\cup \bigcup_{i=1}^n\{b^{m_k+i}a^{-1}\},& \text{if } m_k<0, s_k\neq a
\end{cases},\\
R^{-}_{n}(w)&=\begin{cases}
\mathmakebox[\lend][l]{\varnothing,}& \text{if } m_k=0\\
\varnothing,& \text{if } m_k>0, s_k=a\\
\bigcup_{i=1}^n\{b^{m_k+i-1}a^{-1}\},& \text{if } m_k>0, s_k\neq a\\
\bigcup_{i=1}^n\{b^{m_k-i+1}s\mid s\in S_b\setminus\{a^{-1}\}\},& \text{if } m_k<0, s_k=a\\
\varnothing,& \text{if } m_k<0, s_k\neq a
\end{cases}.
\end{align*}
Then we call \[(\phi w)_n=\sum_{u\in W^{+}_{n}(w)}\phi u - \sum_{u\in W^{-}_{n}(w)}\phi u\] the \emph{n-representative} of $[\phi w]$.
\end{defn}

\begin{prop}\label{phiwnequiv}
Let $w\equiv_S (m_0,s_1,...,s_k,m_k)$ and assume that $w$ is not a power of $b$. Then $\phi w\circ T^n$ is equivalent to $(\phi w)_n$.
\end{prop}

\begin{proof}
By \ref{induction} it is enough to prove that $\sum_{u\in W^{\star}_{n}(w)}\phi u$ is equivalent to $(\phi w)_n$. First we consider the middle of the words: Note that $W^{m}_n(w)=M_{n}(w)$. We fix for the proof $\{m\}=M_{n}(w)$. Now we consider the left side of the words: Let $x=x_1...x_t$ be a reduced word and assume that $x_1\notin \{b,b^{-1}\}$. We want to show that \begin{equation}\label{leftind}\sum_{l\in W^{l}_{n}(w)}\phi lx \sim \sum_{l^{+}\in L^{+}_{n}(w)}\phi l^{+}x - \sum_{l^{-}\in L^{-}_{n}(w)}\phi l^{-}x.\end{equation} We have to consider three cases:\begin{enumerate}
\item If $m_0=0$ or $m_0>0$ and $s_1\neq a^{-1}$ or $m_0<0$ and $s_1=a^{-1}$, then we get that $L^{+}_{n}(w)=W^{l}_{n}(w)$ and $L^{-}_{n}(w)=\varnothing$, so we have \ref{leftind}.
\item If $m_0>0$ and $s_1=a^{-1}$, we get using left-extension relations $n$ times that \begin{equation}\begin{split}
\sum_{l\in W^{l}_{n}(w)}\phi lx = & \phi b^{m_0+n}x + \sum_{i=1}^n \phi ab^{m_0+i-1}x\\
\sim & \phi b^{m_0+n-1}x + \sum_{i=1}^{n-1}\phi ab^{m_0+i-1}x - \sum_{s\in S_b\setminus\{a\}} \phi sb^{m_0+n-1}x\\
\sim & ...\\
\sim & \phi b^{m_0}x - \sum_{i=1}^{n}\sum_{s\in S_b\setminus\{a\}}\phi sb^{m_0+i-1}x\\
= & \sum_{l^{+}\in L^{+}_{n}(w)}\phi l^{+}x - \sum_{l^{-}\in L^{-}_{n}(w)}\phi l^{-}x
\end{split}\end{equation}
\item If $m_0<0$ and $s_1\neq a^{-1}$, we get using left-extension relations $n$ times that \begin{equation}\begin{split}
\sum_{l\in W^{l}_{n}(w)}\phi lx =& \phi b^{m_0-n}x + \sum_{i=1}^n\sum_{s\in S_b\setminus\{a\}}\phi sb^{m_0-i+1}x\\
\sim& \phi b^{m_0-n+1}x + \sum_{i=1}^{n-1}\sum_{s\in S_b\setminus\{a\}}\phi sb^{m_0-i+1}x -\phi ab^{m_0-n+1}x\\
\sim& ...\\
\sim& \phi b^{m_0}x - \sum_{i=1}^{n}\phi ab^{m_0-i+1}x\\
=&\sum_{l^{+}\in L^{+}_{n}(w)}\phi l^{+}x - \sum_{l^{-}\in L^{-}_{n}(w)}\phi l^{-}x
\end{split}\end{equation}
\end{enumerate}Now we consider the right side of the words: Let $x=x_1...x_t$ be a reduced word and assume that $x_t\notin \{b,b^{-1}\}$. We want to show that \begin{equation}\label{rightind}\sum_{r\in W^{r}_{n}(w)}\phi xr \sim \sum_{r^{+}\in R^{+}_{n}(w)}\phi xr^{+} - \sum_{r^{-}\in R^{-}_{n}(w)}\phi xr^{-}.\end{equation} Again we have to consider three cases: \begin{enumerate}
\item If $m_k=0$ or $m_k>0$ and $s_k=a$ or $m_k<0$ and $s_k\neq a$, then we get that $R^{+}_{n}(w)=W^{r}_{n}(w)$ and $R^{-}_{n}(w)=\varnothing$, so we have \ref{rightind}.
\item If $m_k>0$ and $s_k\neq a$, we get using right-extension relations $n$ times that \begin{equation}\begin{split} \sum_{r\in W^{r}_{n}(w)}\phi xr =& \phi xb^{m_k+n} + \sum_{i=1}^n\sum_{s\in S_b\setminus\{a^{-1}\}} \phi xb^{m_k+i-1}s\\
\sim& \phi xb^{m_k+n-1} + \sum_{i=1}^{n-1}\sum_{s\in S_b\setminus\{a^{-1}\}} \phi xb^{m_k+i-1}s - \phi xb^{m_k+n-1}a^{-1}\\
\sim& ...\\
\sim& \phi xb^{m_k} - \sum_{i=1}^{n} \phi xb^{m_k+i-1}a^{-1}\\
=&\sum_{r^{+}\in R^{+}_{n}(w)}\phi xr^{+} - \sum_{r^{-}\in R^{-}_{n}(w)}\phi xr^{-}\end{split}\end{equation}
\item If $m_k<0$ and $s_k=a$, we get using right-extension relations $n$ times that \begin{equation}\begin{split} \sum_{r\in W^{r}_{n}(w)}\phi xr =& \phi xb^{m_k-n} + \sum_{i=1}^n \phi xb^{m_k-i+1}a^{-1}\\
\sim& \phi xb^{m_k-n+1} + \sum_{i=1}^{n-1} \phi xb^{m_k-i+1}a^{-1} - \sum_{s\in S_b\setminus\{a^{-1}\}}\phi xb^{m_k-n+1}s\\
\sim& ...\\
\sim& \phi xb^{m_k} - \sum_{i=1}^{n}\sum_{s\in S_b\setminus\{a^{-1}\}}\phi xb^{m_k-i+1}s\\
=&\sum_{r^{+}\in R^{+}_{n}(w)}\phi xr^{+} - \sum_{r^{-}\in R^{-}_{n}(w)}\phi xr^{-}\end{split}\end{equation}
\end{enumerate} Using \ref{leftind} and \ref{rightind} and the fact that $m$ is $b$-truncated, we finally get that\begin{equation}\begin{split}
\sum_{u\in W^{\star}_{n}(w)}\phi u= &\sum_{r\in W^{r}_{n}(w)} \sum_{l\in W^{l}_{n}(w)}\phi lmr\\
\sim& \sum_{r\in W^{r}_{n}(w)} (\sum_{l^{+}\in L^{+}_{n}(w)}\phi l^{+}mr - \sum_{l^{-}\in L^{-}_{n}(w)}\phi l^{-}mr)\\
=& \sum_{l^{+}\in L^{+}_{n}(w)}\sum_{r\in W^{r}_{n}(w)}\phi l^{+}mr - \sum_{l^{-}\in L^{-}_{n}(w)}\sum_{r\in W^{r}_{n}(w)}\phi l^{-}mr\\
\sim& \sum_{l^{+}\in L^{+}_{n}(w)}(\sum_{r^{+}\in R^{+}_{n}(w)}\phi l^{+}mr^{+} - \sum_{r^{-}\in R^{-}_{n}(w)}\phi l^{+}mr^{-})\\
&- \sum_{l^{-}\in L^{-}_{n}(w)}(\sum_{r^{+}\in R^{+}_{n}(w)}\phi l^{-}mr^{+} - \sum_{r^{-}\in R^{-}_{n}(w)}\phi l^{-}mr^{-})\\
=&\sum_{u\in W^{+}_{n}(w)}\phi u - \sum_{u\in W^{-}_{n}(w)}\phi u\\
=&(\phi w)_n. \end{split}\end{equation}
\end{proof}

The look of this representative suggests that we are able to apply our Criterion \ref{redtree} if $n$ is big enough. This turns out to be true. To handle $(\phi w)_n$ more conveniently we want to write it as one sum and not as the difference of two sums. We can do this using the following:

\begin{prop}\label{disjoint+-}
Let $w$ be a reduced word and assume that $w$ is not a power of $b$. Then $W^{+}_{n}(w)\cap W^{-}_{n}(w)=\varnothing$.
\end{prop} Thus if we set $W_{n}(w)=W^{+}_{n}(w)\cup W^{-}_{n}(w)$ and define the weight \[w_n(u)=\begin{cases}
1,& \text{if } u\in W^{+}_{n}(w)\\
-1,& \text{if } u\in W^{-}_{n}(w)\\
0,& \text{otherwise}
\end{cases},\] we obtain: \begin{equation}\label{phiwn}
(\phi w)_n=\sum_{u\in W_{n}(w)}w_n(u)\phi u.
\end{equation} We call $W_{n}(w)$ the \emph{n-support} of $w$. Note that $W_{n}(w)=L_{n}(w)\cdot M_{n}(w)\cdot R_{n}(w)$, where\begin{enumerate}
\item $L_{n}(w)=L^{+}_{n}(w)\cup L^{-}_{n}(w)$,
\item $R_{n}(w)=R^{+}_{n}(w)\cup R^{-}_{n}(w)$.
\end{enumerate} We call $L_{n}(w)$ the \emph{left n-support} of $w$ and $R_{n}(w)$ be the \emph{right n-support} of $w$. For later reference we prove the following lemma

\begin{lem}
Let $w$ be a reduced word and assume that $w$ is not a power of $b$. Then $$W_{n}(w)^{-1}=W_{n}(w^{-1}).$$
\end{lem}

\begin{proof}
Since $\tau_b(w)^{-1}=\tau_b(w^{-1})$ and $T^{-1}$ is an automorphism, we get $M_{n}(w)^{-1}=M_{n}(w^{-1})$. We see that $L_{n}(w)^{-1}=R_{n}(w^{-1})$ by explicit comparison in all cases.
\end{proof}

Now we want to prove the Proposition \ref{disjoint+-}. Looking at \ref{def:phiwn} we can observe the following:

\begin{cor}\label{lrlook}
Let $w\equiv_S (m_0,s_1,...,s_k,m_k)$ and assume that $w$ is not a power of $b$. Let $l\in L_{n}(w)$. \begin{enumerate}[(i)]
\item If $m_0=0$, then $l=e$.
\item If $m_0\neq 0$, then $l=b^{m_0}$ or $l=sb^i$ for $s\in S_b\setminus \{s_1^{-1}\}$ and $i\in \mathbb{Z}$.
\end{enumerate} Let $r\in R_{n}(w)$. Then\begin{enumerate}[(i)]
\item If $m_k=0$, then $r=e$.
\item If $m_k\neq 0$, then $r=b^{m_k}$ or $r=b^js$ for $s\in S_b\setminus \{s_k^{-1}\}$ and $j\in \mathbb{Z}$.
\end{enumerate}
\end{cor}

The following statement is not only useful for reformulating the n-representative, but fundamental for our later computations of its length.

\begin{lem}\label{lmrred}
Let $w\equiv_S (m_0,s_1,...,s_k,m_k)$ and assume that $w$ is not a power of $b$. Let $l\in L_{n}(w)$, $\{m\}=M_{n}(w)$ and $r\in R_{n}(w)$. Then $u=lmr\in W_n(w)$ is a reduced word.
\end{lem}

\begin{proof}
We know by \ref{btrsubwords} that $m\equiv_S(s_1,m_1^-,...,m_{k-1}^-,s_k)$, where \begin{equation}
m_j^-=m_j-n\#a(s_j)+n\#a^{-1}(s_{j+1}).\end{equation} So the statement follows from \ref{lrlook}.
\end{proof}

Now we can finally prove the Proposition \ref{disjoint+-}:

\begin{proof}
By \ref{def:phiwn} we have that $L^{+}_{n}(w)\cap L^{-}_{n}(w)=\varnothing$ and $R^{+}_{n}(w)\cap R^{-}_{n}(w)=\varnothing$. So we get by \ref{lmrred} that $W^{+}_{n}(w)\cap W^{-}_{n}(w)=\varnothing$.
\end{proof}

Until now we only found the n-representative for words that are not a power of $b$. If $w=b^k$ for some $k\in \mathbb{Z}\setminus \{0,1\}$ this will have to wait till Section \ref{Tsums}, since we regard them by \ref{powerofbtob} as sums of counting quasimorphisms. However we will check the case $w=b$ right away.

\begin{defn}\label{phibn}
We define the \emph{n-support} of $b$ by $W_n(b)=\{a,b\}$. We define the \emph{n-representative} of $\phi b$ by $(\phi b)_n=\sum_{u\in W_{n}(b)}b_n(u)\phi u$ with the weight \[b_n(u)=\begin{cases}
n,& \text{if } u=a\\
1,& \text{if } u=b\\
0,& \text{otherwise}
\end{cases}.\]
\end{defn}

\begin{lem}\label{phibnequiv}
With the definitions above we get that $\phi b\circ T^n$ is equivalent to $(\phi b)_n$.
\end{lem}

\begin{proof}
We know that $\phi b\circ T=\phi a+\phi b$. By \ref{def:phiwn} we get that $(\phi a)_n=\phi a$. The statement follows.
\end{proof}

Note that the n-support of $w$ is the support of the weight $w_n$. This gives us the possibility to study $\lVert T^{-n}[\phi w]\rVert_S$ and eventually $\lvert T^{-n}[\phi w]\rvert_S$ by looking at the n-support of $w$, which will in turn allow us to compute the speed of $T^{-1}$ on $[\phi w]$. In the case $w=b$ this is really easy:

\begin{cor}\label{speedonphib}
We get that $\lvert T^{-n}[\phi b]\rvert_S=1$ for all $n\in \mathbb{N}$. In particular the speed of $T^{-1}$ on $[\phi b]$ exists and $sp(T^{-1},[\phi b])=0$.
\end{cor}

\begin{proof}
By \ref{phibnequiv} we get that $\lvert T^{-n}[\phi b]\rvert_S=\lvert (\phi b)_n\rvert_S$. By \ref{1red} we get $\lvert (\phi b)_n\rvert_S=\lVert (\phi b)_n\rVert_S=1$. The rest of the statement follows by the definition of speed.
\end{proof}

If $w\neq b$ we have a lot more work to do. We will need the next two subsections to get a formula for $sp(T^{-1},[\phi w])$ and show that $T^{-1}$ has linear speed.

\subsection{Study of the n-support}\label{studynsupportw}

Let $w\equiv_S (m_0,s_1,...,s_k,m_k)$ be a reduced word and assume that $w$ is not a power of $b$. In this subsection we prove that the n-representative of $[\phi w]$ is reduced if $n$ is big enough. We want to apply our Criterion \ref{redtree}, so we need to show that the $b$-truncated end of the n-support is not empty if $n$ is big enough. We start our study of the n-support by defining a partition on it:

\begin{defn}\label{partition}
Let $w\equiv_S (m_0,s_1,...,s_k,m_k)$ and assume that $w$ is not a power of $b$. We extend the definitions of the left and right n-support of $w$ by\begin{enumerate}
\item $L_0(w)=\{b^{m_0}\}$ and $L_{-1}(w)=\varnothing$,
\item $R_0(w)=\{b^{m_k}\}$ and $R_{-1}(w)=\varnothing$.
\end{enumerate}Note that we have for $i\geq 0$\begin{enumerate}
\item $L_{i-1}(w)\subset L_{i}(w)$,
\item $R_{i-1}(w)\subset R_{i}(w)$.
\end{enumerate}So we can define the following sets for $i\geq 0$\begin{enumerate}
\item $dL_{i}(w)=L_{i}(w)\setminus L_{i-1}(w)$,
\item $dR_{i}(w)=R_{i}(w)\setminus R_{i-1}(w)$.
\end{enumerate}We get the following partitions:\begin{enumerate}
\item $L_{n}(w)=\amalg_{i=0}^n dL_{i}(w)$,
\item $R_{n}(w)=\amalg_{i=0}^n dR_{i}(w)$.
\end{enumerate}These partitions give us a partition of the n-support of $w$: \begin{equation}\label{partitionn}
W_{n}(w)=\amalg_{i=1}^n \amalg_{j=0}^n dL_{i}(w)\cdot M_{n}(w)\cdot dR_{j}(w)
\end{equation}
\end{defn}

\begin{defn}
Let $w$ be a reduced word and assume that $w$ is not a power of $b$. We say that $u=lmr\in W_n(w)$ is \emph{of type $(i,j)$} if $l\in dL_{i}(w)$ and $r\in dL_{j}(w)$. The type of an element $u\in W_n(w)$ is well defined by \ref{partitionn}.
\end{defn} 

We will now spend some time to explain our program for the rest of this subsection: Let $w\equiv_S (m_0,s_1,...,s_k,m_k)$ and assume that $w$ is not a power of $b$. As a first step we visualize the partition \ref{partitionn} of the n-support the following way: We depict $W_n(w)$ as an rectangle partitioned into squares. Each square is denoted by a pair $(i,j)$ with $0\leq i,j\leq n$ and is thought of containing the words of type $(i,j)$. We call $W_n(w)$ a \emph{rectangle} and $\{v\in W_n(w)\mid v\ \text{is of type}\ (i,j)\}$ a \emph{square}. We see in \ref{4shapes} that there are four different shapes of rectangles depending on whether $m_0=0$ and $m_k=0$. We call them \emph{cell}, \emph{column}, \emph{row} and \emph{box} respectively.

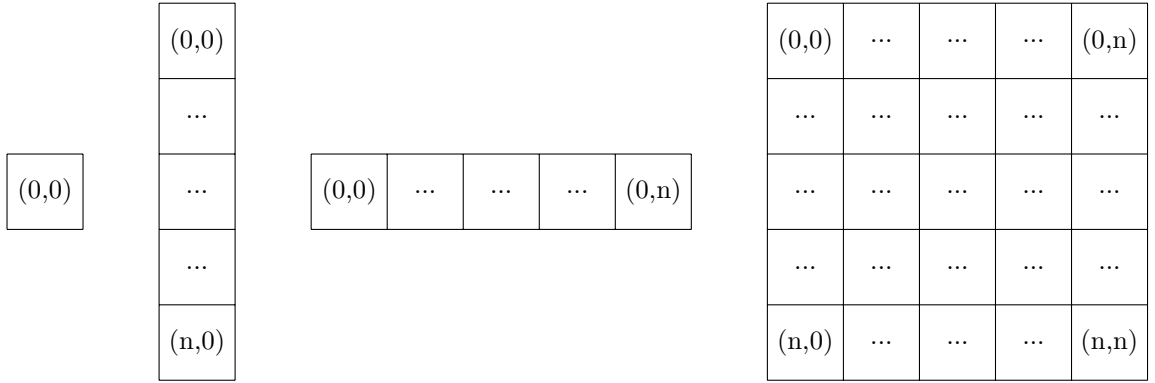
\begin{center}
\begin{tikzpicture}
\draw[step=1cm] (0,2) grid (1,3);
\node at (0.5,2.5) {(0,0)};

\draw[step=1cm] (2,0) grid (3,5);
\node at (2.5,0.5) {(n,0)};
\node at (2.5,1.5) {...};
\node at (2.5,2.5) {...};
\node at (2.5,3.5) {...};
\node at (2.5,4.5) {(0,0)};

\draw[step=1cm] (4,2) grid (9,3);
\node at (4.5,2.5) {(0,0)};
\node at (5.5,2.5) {...};
\node at (6.5,2.5) {...};
\node at (7.5,2.5) {...};
\node at (8.5,2.5) {(0,n)};

\draw[step=1cm] (10,0) grid (15,5);

\node at (10.5,4.5) {(0,0)};
\node at (11.5,4.5) {...};
\node at (12.5,4.5) {...};
\node at (13.5,4.5) {...};
\node at (14.5,4.5) {(0,n)};

\node at (10.5,3.5) {...};
\node at (11.5,3.5) {...};
\node at (12.5,3.5) {...};
\node at (13.5,3.5) {...};
\node at (14.5,3.5) {...};

\node at (10.5,2.5) {...};
\node at (11.5,2.5) {...};
\node at (12.5,2.5) {...};
\node at (13.5,2.5) {...};
\node at (14.5,2.5) {...};

\node at (10.5,1.5) {...};
\node at (11.5,1.5) {...};
\node at (12.5,1.5) {...};
\node at (13.5,1.5) {...};
\node at (14.5,1.5) {...};

\node at (10.5,0.5) {(n,0)};
\node at (11.5,0.5) {...};
\node at (12.5,0.5) {...};
\node at (13.5,0.5) {...};
\node at (14.5,0.5) {(n,n)};
\end{tikzpicture}
\end{center}

In a surprising twist the most comfortable way to study the length of elements in the n-support is to forget about elements. Thinking about squares instead we get the results we need and actually a nice geometric intuition. We prove in \ref{ijlength} that the length of an element $u\in W_n(w)$ depends only on its type $(i,j)$ and introduce in \ref{def:ijlength} the \emph{length of a square} $(i,j)$, which is just the length of an element $u\in W_n(w)$ of type $(i,j)$. From there on we will concentrate on the length of squares.

We see in \ref{ijbendWnw} that the elements of $W_n(w)$, which are not $b$-truncated, sit in the leftmost or highest squares. We colour these squares black in the visualization.

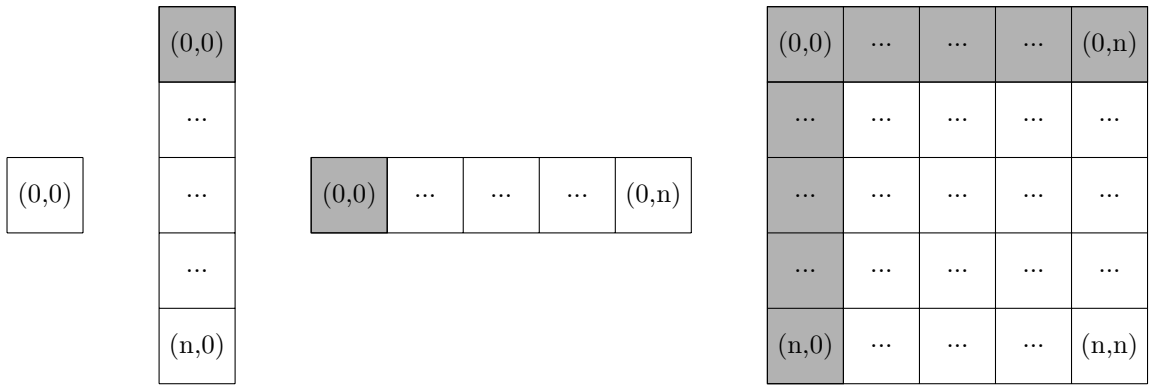
\begin{center}
\begin{tikzpicture}
\draw[step=1cm] (0,2) grid (1,3);
\node at (0.5,2.5) {(0,0)};

\filldraw[fill=black!30!white] (2,4) rectangle (3,5);
\draw[step=1cm] (2,0) grid (3,5);
\node at (2.5,0.5) {(n,0)};
\node at (2.5,1.5) {...};
\node at (2.5,2.5) {...};
\node at (2.5,3.5) {...};
\node at (2.5,4.5) {(0,0)};

\filldraw[fill=black!30!white] (4,2) rectangle (5,3);
\draw[step=1cm] (4,2) grid (9,3);
\node at (4.5,2.5) {(0,0)};
\node at (5.5,2.5) {...};
\node at (6.5,2.5) {...};
\node at (7.5,2.5) {...};
\node at (8.5,2.5) {(0,n)};

\filldraw[fill=black!30!white] (10,4) rectangle (15,5);
\filldraw[fill=black!30!white] (10,4) rectangle (11,0);
\draw[step=1cm] (10,0) grid (15,5);

\node at (10.5,4.5) {(0,0)};
\node at (11.5,4.5) {...};
\node at (12.5,4.5) {...};
\node at (13.5,4.5) {...};
\node at (14.5,4.5) {(0,n)};

\node at (10.5,3.5) {...};
\node at (11.5,3.5) {...};
\node at (12.5,3.5) {...};
\node at (13.5,3.5) {...};
\node at (14.5,3.5) {...};

\node at (10.5,2.5) {...};
\node at (11.5,2.5) {...};
\node at (12.5,2.5) {...};
\node at (13.5,2.5) {...};
\node at (14.5,2.5) {...};

\node at (10.5,1.5) {...};
\node at (11.5,1.5) {...};
\node at (12.5,1.5) {...};
\node at (13.5,1.5) {...};
\node at (14.5,1.5) {...};

\node at (10.5,0.5) {(n,0)};
\node at (11.5,0.5) {...};
\node at (12.5,0.5) {...};
\node at (13.5,0.5) {...};
\node at (14.5,0.5) {(n,n)};
\end{tikzpicture}
\end{center}

We want to show that for $n$ big enough, there is a square with bigger length then any black square. We will use two lemmas for this. The first lemma \ref{LRgrow} gives us the following for $n>N=\max (\lvert m_0\rvert, \lvert m_k\rvert)$: If $m_0\neq 0$, we draw a red line under the squares whose horizontal coordinate is $\lvert m_0\rvert$. If you want to walk from the upper, left corner to the lower, right corner of the rectangle, you have to cross the red line. Once you crossed the red line, the length of the squares you walk on grows with every step down by 1. If $m_k\neq 0$, we draw a blue line right of the squares whose vertical coordinate is $\lvert m_k\rvert$. If you want to walk from the upper, left corner to the lower, right corner of the rectangle, you have to cross the blue line. Once you crossed the blue line, the length of the squares you walk on grows with every step to the right by 1.

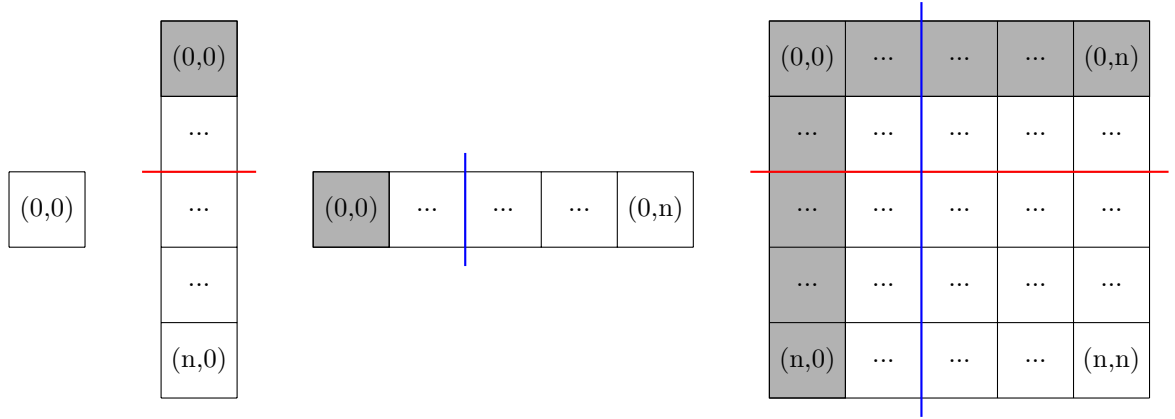
\begin{center}
\begin{tikzpicture}
\draw[step=1cm] (0,2) grid (1,3);
\node at (0.5,2.5) {(0,0)};

\filldraw[fill=black!30!white] (2,4) rectangle (3,5);
\draw[step=1cm] (2,0) grid (3,5);

\node at (2.5,0.5) {(n,0)};
\node at (2.5,1.5) {...};
\node at (2.5,2.5) {...};
\node at (2.5,3.5) {...};
\node at (2.5,4.5) {(0,0)};
\draw[thick, red] (1.75,3)--(3.25,3);

\filldraw[fill=black!30!white] (4,2) rectangle (5,3);
\draw[step=1cm] (4,2) grid (9,3);

\node at (4.5,2.5) {(0,0)};
\node at (5.5,2.5) {...};
\node at (6.5,2.5) {...};
\node at (7.5,2.5) {...};
\node at (8.5,2.5) {(0,n)};
\draw[thick, blue] (6,1.75)--(6,3.25);

\filldraw[fill=black!30!white] (10,4) rectangle (15,5);
\filldraw[fill=black!30!white] (10,4) rectangle (11,0);
\draw[step=1cm] (10,0) grid (15,5);

\node at (10.5,4.5) {(0,0)};
\node at (11.5,4.5) {...};
\node at (12.5,4.5) {...};
\node at (13.5,4.5) {...};
\node at (14.5,4.5) {(0,n)};

\node at (10.5,3.5) {...};
\node at (11.5,3.5) {...};
\node at (12.5,3.5) {...};
\node at (13.5,3.5) {...};
\node at (14.5,3.5) {...};

\node at (10.5,2.5) {...};
\node at (11.5,2.5) {...};
\node at (12.5,2.5) {...};
\node at (13.5,2.5) {...};
\node at (14.5,2.5) {...};

\node at (10.5,1.5) {...};
\node at (11.5,1.5) {...};
\node at (12.5,1.5) {...};
\node at (13.5,1.5) {...};
\node at (14.5,1.5) {...};

\node at (10.5,0.5) {(n,0)};
\node at (11.5,0.5) {...};
\node at (12.5,0.5) {...};
\node at (13.5,0.5) {...};
\node at (14.5,0.5) {(n,n)};
\draw[thick, red] (9.75,3)--(15.25,3);
\draw[thick, blue] (12,-0.25)--(12,5.25);
\end{tikzpicture}
\end{center}

The second lemma \ref{LRanlauf} gives us the following for $n>2N=2\max (\lvert m_0\rvert, \lvert m_k\rvert)$: If $m_0\neq 0$, we colour the squares red whose horizontal coordinate is $2\lvert m_0\rvert$. If you want to walk from the upper, left corner to the lower, right corner of the rectangle, you have to step on a red square. Once you walked over the red square, the length of any square you stand on is bigger than the length of every square above it with the same vertical coordinate. If $m_k\neq 0$, we colour the squares blue whose vertical coordinate is $2\lvert m_k\rvert$. If you want to walk from the upper, left corner to the lower, right corner of the rectangle, you have to step on a blue square. Once you walked over the blue square, the length of any square you stand on is bigger than the length of every square left of it with the same horizontal coordinate.

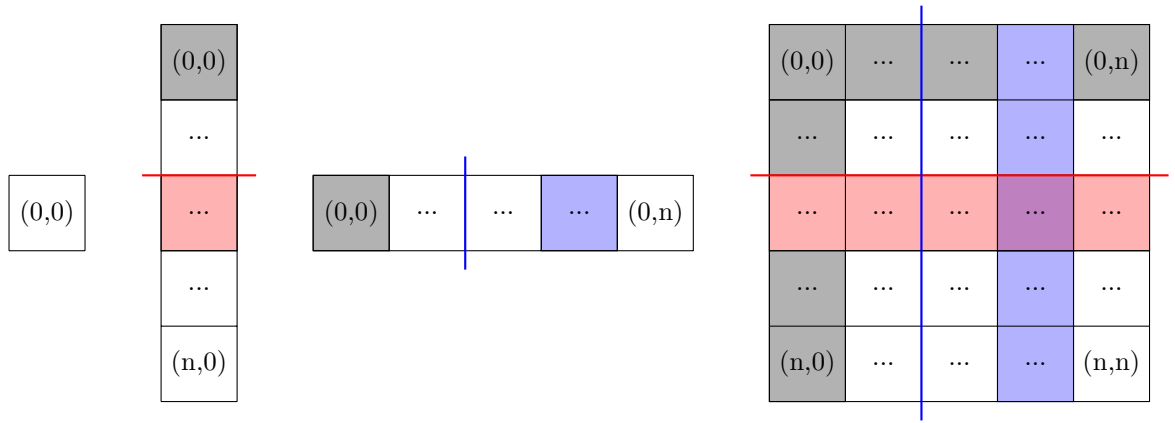
\begin{center}
\begin{tikzpicture}
\draw[step=1cm] (0,2) grid (1,3);
\node at (0.5,2.5) {(0,0)};

\filldraw[fill=red!30!white] (2,2) rectangle (3,3);
\filldraw[fill=black!30!white] (2,4) rectangle (3,5);
\draw[step=1cm] (2,0) grid (3,5);
\node at (2.5,0.5) {(n,0)};
\node at (2.5,1.5) {...};
\node at (2.5,2.5) {...};
\node at (2.5,3.5) {...};
\node at (2.5,4.5) {(0,0)};
\draw[thick, red] (1.75,3)--(3.25,3);

\filldraw[fill=blue!30!white] (7,2) rectangle (8,3);
\filldraw[fill=black!30!white] (4,2) rectangle (5,3);
\draw[step=1cm] (4,2) grid (9,3);
\node at (4.5,2.5) {(0,0)};
\node at (5.5,2.5) {...};
\node at (6.5,2.5) {...};
\node at (7.5,2.5) {...};
\node at (8.5,2.5) {(0,n)};
\draw[thick, blue] (6,1.75)--(6,3.25);

\filldraw[fill=black!30!white] (10,4) rectangle (15,5);
\filldraw[fill=black!30!white] (10,4) rectangle (11,0);
\filldraw[fill=red!30!white] (10,2) rectangle (13,3);
\filldraw[fill=red!30!white] (14,2) rectangle (15,3);
\filldraw[fill=blue!30!white] (13,0) rectangle (14,2);
\filldraw[fill=blue!30!white] (13,3) rectangle (14,5);
\filldraw[fill=blue!50!red!50!white] (13,2) rectangle (14,3);
\draw[step=1cm] (10,0) grid (15,5);

\node at (10.5,4.5) {(0,0)};
\node at (11.5,4.5) {...};
\node at (12.5,4.5) {...};
\node at (13.5,4.5) {...};
\node at (14.5,4.5) {(0,n)};

\node at (10.5,3.5) {...};
\node at (11.5,3.5) {...};
\node at (12.5,3.5) {...};
\node at (13.5,3.5) {...};
\node at (14.5,3.5) {...};

\node at (10.5,2.5) {...};
\node at (11.5,2.5) {...};
\node at (12.5,2.5) {...};
\node at (13.5,2.5) {...};
\node at (14.5,2.5) {...};

\node at (10.5,1.5) {...};
\node at (11.5,1.5) {...};
\node at (12.5,1.5) {...};
\node at (13.5,1.5) {...};
\node at (14.5,1.5) {...};

\node at (10.5,0.5) {(n,0)};
\node at (11.5,0.5) {...};
\node at (12.5,0.5) {...};
\node at (13.5,0.5) {...};
\node at (14.5,0.5) {(n,n)};
\draw[thick, red] (9.75,3)--(15.25,3);
\draw[thick, blue] (12,-0.25)--(12,5.25);
\end{tikzpicture}
\end{center}

Once we have proved all these results, we will see in \ref{btruncendWnw} that the $b$-truncated end of the $W_n(w)$ is not empty if $n>2N=2\max (\lvert m_0\rvert, \lvert m_k\rvert)$. Then we can apply our Criterion \ref{redtree} and get in \ref{tentacle} that $(\phi w)_n$ is reduced.

Now we start our program by proving that there are four different shapes of rectangles. We get an explicit description of the sets $dL_{i}(w)$ and $dR_{i}(w)$ by looking at \ref{def:phiwn}:

\newlength{\lene}
\settowidth{\lene}{$\{b^{m_k-i+1}s\mid s\in \bar{S}\setminus\{a^{-1},b,b^{-1}\}\}$,}

\begin{lem}\label{LRexpl}
Let $w\equiv_S (m_0,s_1,...,s_k,m_k)$ and assume that $w$ is not a power of $b$. We get $dL_{0}(w)=\{b^{m_0}\}$ and $dR_{0}(w)=\{b^{m_k}\}$. For $i\geq 1$ and $j\geq 1$ we have:

$dL_{i}(w)=\begin{cases}
\mathmakebox[\lene][l]{\varnothing,}& \text{if } m_0=0\\
\{ab^{m_0-i}\},& \text{if } m_0>0, s_1\neq a^{-1}\\
\{sb^{m_0+i-1}\mid s\in S_b\setminus\{a\}\},& \text{if } m_0>0, s_1=a^{-1}\\
\{ab^{m_0-i+1}\},& \text{if } m_0<0, s_1\neq a^{-1}\\
\{sb^{m_0+i}\mid s\in S_b\setminus\{a\}\},& \text{if } m_0<0, s_1=a^{-1}
\end{cases}$,

$dR_{j}(w)=\begin{cases}
\mathmakebox[\lene][l]{\varnothing,}& \text{if } m_k=0\\
\{b^{m_k-j}s\mid s\in S_b\setminus\{a^{-1}\}\},& \text{if } m_k>0, s_k=a\\
\{b^{m_k+j-1}a^{-1}\},& \text{if } m_k>0, s_k\neq a\\
\{b^{m_k-j+1}s\mid s\in S_b\setminus\{a^{-1}\}\},& \text{if } m_k<0, s_k=a\\
\{b^{m_k+j}a^{-1}\},& \text{if } m_k<0, s_k\neq a
\end{cases}$.
\end{lem}

\begin{proof}
This follows from \ref{def:phiwn}.
\end{proof} 

We see that how the sets $dL_{i}(w)$ and $dR_{i}(w)$ look like depends primarily on the values of $m_0$ and $m_k$. This motivates the following definition:

\begin{defn}
Let $w\equiv_S (m_0,s_1,...,s_k,m_k)$ and assume that $w$ is not a power of $b$. We introduce the following terms:
\begin{enumerate}
\item $w$ is $b$-\emph{left} if $m_0\neq 0$ and $m_k=0$.
\item $w$ is \emph{right}-$b$ if $m_0=0$ and $m_k\neq 0$.
\item $w$ is $b$-\emph{and}-$b$ if $m_0\neq 0$ and $m_k\neq0$.
\end{enumerate}
\end{defn}

Note that every word in $F_n\setminus \{e\}$ is either a power of $b$ or $b$-truncated or $b$-left or right-$b$ or $b$-and-$b$. Now we can read the four different shapes of rectangles off from the Lemma \ref{LRexpl}

\begin{cor}\label{4shapes}
Let $w$ be a reduced word and assume that $w$ is not a power of $b$.\begin{enumerate}[(i)]
\item Assume $w$ is $b$-truncated. Then every element $v\in W_n(w)$ is of type $(0,0)$. We call $W_n(w)$ a \emph{cell}.
\item Assume $w$ is $b$-left. Then every element $v\in W_n(w)$ is of type $(i,0)$ for $0\leq i\leq n$. We call $W_n(w)$ a \emph{column}.
\item Assume $w$ is right-$b$. Then every element $v\in W_n(w)$ is of type $(0,j)$ for $0\leq j\leq n$. We call $W_n(w)$ a \emph{row}.
\item Assume $w$ is $b$-and-$b$. Then every element $v\in W_n(w)$ is of type $(i,j)$ for $0\leq i,j\leq n$. We call $W_n(w)$ a \emph{box}.
\end{enumerate}
\end{cor}

We now give a concrete example: We pick four words such that one is $b$-truncated, one $b$-left, one right-$b$ and one $b$-and-$b$. Then we will visualize their n-support for $n=4$ and $n=8$.

\begin{exmp}
Let $w_1=abab^3a$. Then $w_1$ is $b$-truncated and we visualize $W_4(w_1)$ as a cell:

\begin{center}
\begin{tikzpicture}
\draw[step=1cm] (0,0) grid (1,1);
\node at (0.5,0.5) {(0,0)};
\end{tikzpicture}
\end{center}The same way we visualize $W_8(w_1)$ as:

\begin{center}
\begin{tikzpicture}
\draw[step=1cm] (0,0) grid (1,1);
\node at (0.5,0.5) {(0,0)};
\end{tikzpicture}
\end{center}Let $w_2=bab^2a$. Then $w_2$ is $b$-left and we visualize $W_4(w)$ as a column:

\begin{center}
\begin{tikzpicture}
\draw[step=1cm] (0,0) grid (1,5);
\node at (0.5,0.5) {(4,0)};
\node at (0.5,1.5) {(3,0)};
\node at (0.5,2.5) {(2,0)};
\node at (0.5,3.5) {(1,0)};
\node at (0.5,4.5) {(0,0)};
\end{tikzpicture}
\end{center}The same way we visualize $W_8(w_2)$ as:

\begin{center}
\begin{tikzpicture}
\draw[step=1cm] (0,0) grid (1,9);
\node at (0.5,0.5) {(8,0)};
\node at (0.5,1.5) {(7,0)};
\node at (0.5,2.5) {(6,0)};
\node at (0.5,3.5) {(5,0)};
\node at (0.5,4.5) {(4,0)};
\node at (0.5,5.5) {(3,0)};
\node at (0.5,6.5) {(2,0)};
\node at (0.5,7.5) {(1,0)};
\node at (0.5,8.5) {(0,0)};
\end{tikzpicture}
\end{center}Let $w_3=ab^2ab^2$. Then $w_3$ is right-$b$ and we visualize $W_4(w)$ as a row:

\begin{center}
\begin{tikzpicture}
\draw[step=1cm] (0,0) grid (5,1);
\node at (0.5,0.5) {(0,0)};
\node at (1.5,0.5) {(0,1)};
\node at (2.5,0.5) {(0,2)};
\node at (3.5,0.5) {(0,3)};
\node at (4.5,0.5) {(0,4)};
\end{tikzpicture}
\end{center}The same way we visualize $W_8(w_3)$ as:

\begin{center}
\begin{tikzpicture}
\draw[step=1cm] (0,0) grid (9,1);
\node at (0.5,0.5) {(0,0)};
\node at (1.5,0.5) {(0,1)};
\node at (2.5,0.5) {(0,2)};
\node at (3.5,0.5) {(0,3)};
\node at (4.5,0.5) {(0,4)};
\node at (5.5,0.5) {(0,5)};
\node at (6.5,0.5) {(0,6)};
\node at (7.5,0.5) {(0,7)};
\node at (8.5,0.5) {(0,8)};
\end{tikzpicture}
\end{center}Let $w_4=bab^2$. Then $w_4$ is $b$-and-$b$ and we visualize $W_4(w)$ as a box:

\begin{center}
\begin{tikzpicture}
\draw[step=1cm] (0,0) grid (5,5);

\node at (0.5,4.5) {(0,0)};
\node at (1.5,4.5) {(0,1)};
\node at (2.5,4.5) {(0,2)};
\node at (3.5,4.5) {(0,3)};
\node at (4.5,4.5) {(0,4)};

\node at (0.5,3.5) {(1,0)};
\node at (1.5,3.5) {(1,1)};
\node at (2.5,3.5) {(1,2)};
\node at (3.5,3.5) {(1,3)};
\node at (4.5,3.5) {(1,4)};

\node at (0.5,2.5) {(2,0)};
\node at (1.5,2.5) {(2,1)};
\node at (2.5,2.5) {(2,2)};
\node at (3.5,2.5) {(2,3)};
\node at (4.5,2.5) {(2,4)};

\node at (0.5,1.5) {(3,0)};
\node at (1.5,1.5) {(3,1)};
\node at (2.5,1.5) {(3,2)};
\node at (3.5,1.5) {(3,3)};
\node at (4.5,1.5) {(3,4)};

\node at (0.5,0.5) {(4,0)};
\node at (1.5,0.5) {(4,1)};
\node at (2.5,0.5) {(4,2)};
\node at (3.5,0.5) {(4,3)};
\node at (4.5,0.5) {(4,4)};
\end{tikzpicture}
\end{center}The same way we visualize $W_8(w_4)$ as:

\begin{center}
\begin{tikzpicture}
\draw[step=1cm] (0,0) grid (9,9);

\node at (0.5,8.5) {(0,0)};
\node at (1.5,8.5) {(0,1)};
\node at (2.5,8.5) {(0,2)};
\node at (3.5,8.5) {(0,3)};
\node at (4.5,8.5) {(0,4)};
\node at (5.5,8.5) {(0,5)};
\node at (6.5,8.5) {(0,6)};
\node at (7.5,8.5) {(0,7)};
\node at (8.5,8.5) {(0,8)};

\node at (0.5,7.5) {(1,0)};
\node at (1.5,7.5) {(1,1)};
\node at (2.5,7.5) {(1,2)};
\node at (3.5,7.5) {(1,3)};
\node at (4.5,7.5) {(1,4)};
\node at (5.5,7.5) {(1,5)};
\node at (6.5,7.5) {(1,6)};
\node at (7.5,7.5) {(1,7)};
\node at (8.5,7.5) {(1,8)};

\node at (0.5,6.5) {(2,0)};
\node at (1.5,6.5) {(2,1)};
\node at (2.5,6.5) {(2,2)};
\node at (3.5,6.5) {(2,3)};
\node at (4.5,6.5) {(2,4)};
\node at (5.5,6.5) {(2,5)};
\node at (6.5,6.5) {(2,6)};
\node at (7.5,6.5) {(2,7)};
\node at (8.5,6.5) {(2,8)};

\node at (0.5,5.5) {(3,0)};
\node at (1.5,5.5) {(3,1)};
\node at (2.5,5.5) {(3,2)};
\node at (3.5,5.5) {(3,3)};
\node at (4.5,5.5) {(3,4)};
\node at (5.5,5.5) {(3,5)};
\node at (6.5,5.5) {(3,6)};
\node at (7.5,5.5) {(3,7)};
\node at (8.5,5.5) {(3,8)};

\node at (0.5,4.5) {(4,0)};
\node at (1.5,4.5) {(4,1)};
\node at (2.5,4.5) {(4,2)};
\node at (3.5,4.5) {(4,3)};
\node at (4.5,4.5) {(4,4)};
\node at (5.5,4.5) {(4,5)};
\node at (6.5,4.5) {(4,6)};
\node at (7.5,4.5) {(4,7)};
\node at (8.5,4.5) {(4,8)};

\node at (0.5,3.5) {(5,0)};
\node at (1.5,3.5) {(5,1)};
\node at (2.5,3.5) {(5,2)};
\node at (3.5,3.5) {(5,3)};
\node at (4.5,3.5) {(5,4)};
\node at (5.5,3.5) {(5,5)};
\node at (6.5,3.5) {(5,6)};
\node at (7.5,3.5) {(5,7)};
\node at (8.5,3.5) {(5,8)};

\node at (0.5,2.5) {(6,0)};
\node at (1.5,2.5) {(6,1)};
\node at (2.5,2.5) {(6,2)};
\node at (3.5,2.5) {(6,3)};
\node at (4.5,2.5) {(6,4)};
\node at (5.5,2.5) {(6,5)};
\node at (6.5,2.5) {(6,6)};
\node at (7.5,2.5) {(6,7)};
\node at (8.5,2.5) {(6,8)};

\node at (0.5,1.5) {(7,0)};
\node at (1.5,1.5) {(7,1)};
\node at (2.5,1.5) {(7,2)};
\node at (3.5,1.5) {(7,3)};
\node at (4.5,1.5) {(7,4)};
\node at (5.5,1.5) {(7,5)};
\node at (6.5,1.5) {(7,6)};
\node at (7.5,1.5) {(7,7)};
\node at (8.5,1.5) {(7,8)};

\node at (0.5,0.5) {(8,0)};
\node at (1.5,0.5) {(8,1)};
\node at (2.5,0.5) {(8,2)};
\node at (3.5,0.5) {(8,3)};
\node at (4.5,0.5) {(8,4)};
\node at (5.5,0.5) {(8,5)};
\node at (6.5,0.5) {(8,6)};
\node at (7.5,0.5) {(8,7)};
\node at (8.5,0.5) {(8,8)};
\end{tikzpicture}
\end{center}
\end{exmp}

Note that the rectangles grow with $n$ if $w$ is not $b$-truncated. This follows since $L_{n_1}(w)\subsetneq L_{n_2}(w)$ and $R_{n_1}(w)\subsetneq R_{n_2}(w)$ for $n_1<n_2$ and $w$ not $b$-truncated. But note also that in general $W_{n_1}(w)$ is not a subset of $W_{n_2}(w)$, because $M_{n_1}(w)\neq M_{n_2}(w)$. In the above example $W_{4}(w_4)\subset W_{8}(w_4)$, but $W_{4}(w_i)\not\subset W_{8}(w_i)$ for $i\in \{1,2,3\}$.

Now we examine the length of elements in the n-support. Let $w\equiv_S (m_0,s_1,...,s_k,m_k)$ and assume that $w$ is not a power of $b$. Let $l\in L_{n}(w)$, $\{m\}=M_{n}(w)$ and $r\in R_{n}(w)$. Then $u=lmr$ is in the n-support of $w$ and by \ref{lmrred} we know that \begin{equation}\lvert u\rvert_S=\lvert l\rvert_S+\lvert m\rvert_S+\lvert r\rvert_S.\end{equation} So it makes sense to study the length of the elements in $L_{n}(w)$ and $R_{n}(w)$. The partition of these sets in \ref{partition} helps us to do so.

\newlength{\lenf}
\settowidth{\lenf}{$1+\lvert m_k+j\rvert$,}

\begin{cor}\label{LRlength}
Let $w\equiv_S (m_0,s_1,...,s_k,m_k)$ and assume that $w$ is not a power of $b$.
\begin{enumerate}[(i)]
\item Let $i=0$ and $l_i\in dL_{i}(w)$. Then we have $\lvert l_i \rvert_S=\lvert m_0\rvert$.
\item Let $i\geq 1$ and $l_i\in dL_{i}(w)$. Then we have \begin{equation}\lvert l_i \rvert_S=\begin{cases}
\mathmakebox[\lenf][l]{1+\lvert m_0-i\rvert,}& \text{if } m_0>0, s_1\neq a^{-1}\\
m_0+i,& \text{if } m_0>0, s_1=a^{-1}\\
-m_0+i,& \text{if } m_0<0, s_1\neq a^{-1}\\
1+\lvert m_0+i\rvert,& \text{if } m_0<0, s_1=a^{-1}
\end{cases}.\end{equation}
\item Let $j=0$ and $r_j\in dR_{j}(w)$. Then we have $\lvert r_j \rvert_S=\lvert m_k\rvert$.
\item Let $j\geq 1$ and $r_j\in dR_{j}(w)$. Then we have \begin{equation}\lvert r_j \rvert_S=\begin{cases}
\mathmakebox[\lenf][l]{1+\lvert m_k-j\rvert,}& \text{if } m_k>0, s_k=a\\
m_k+j,& \text{if } m_k>0, s_k\neq a\\
-m_k+j,& \text{if } m_k<0, s_k=a\\
1+\lvert m_k+j\rvert,& \text{if } m_k<0, s_k\neq a
\end{cases}.\end{equation}
\end{enumerate}
\end{cor}

\begin{proof}
This follows from \ref{LRexpl}.
\end{proof}

\begin{cor}\label{ijlength}
Let $w\equiv_S (m_0,s_1,...,s_k,m_k)$ be a reduced word and assume that $w$ is not a power of $b$. We have: \begin{enumerate}[(i)]
\item Let $l_1, l_2\in dL_{i}(w)$. Then $\lvert l_1\rvert_S=\lvert l_2\rvert_S$.
\item Let $r_1, r_2\in dR_{j}(w)$. Then $\lvert r_1\rvert_S=\lvert r_2\rvert_S$.
\end{enumerate}
\end{cor}

\begin{proof}
This follows from \ref{LRlength}.
\end{proof}

So elements of the n-support that are of the same type have the same length. This motivates the following definition:

\begin{defn}\label{def:ijlength}
Assume that $w$ is not a power of $b$. Let $u\in W_n(w)$ be of type $(i,j)$. Then \ref{ijlength} and \ref{lmrred} allow us to define \[\lvert (i,j)\rvert_{W_n(w)}=\lvert u\rvert_S.\] We call $\lvert (i,j)\rvert_{W_n(w)}$ the \emph{length of the square $(i,j)$}.
\end{defn}

Now we want to locate the elements of $W_n(w)$, which are not $b$-truncated, in the rectangle $W_n(w)$. We see that they sit in the leftmost or highest squares.

\begin{lem}\label{ijbendWnw}
Let $w\equiv_S (m_0,s_1,...,s_k,m_k)$ be a reduced word and assume that $w$ is not a power of $b$. \begin{enumerate}[(i)]
\item If $w$ is $b$-truncated, then the element of $W_n(w)$ is $b$-truncated.
\item If $w$ is $b$-left, then $v\in W_n(w)$ is not $b$-truncated if and only if $v$ is of type $(0,0)$.
\item If $w$ is right-$b$, then $v\in W_n(w)$ is not $b$-truncated if and only if $v$ is of type $(0,0)$.
\item If $w$ is $b$-and-$b$, then $v\in W_n(w)$ is not $b$-truncated if and only if $v$ is of type $(i,0)$ or $(0,j)$ for $0\leq i,j\leq n$.\end{enumerate}
\end{lem}

\begin{proof}
Let $v=lmr\in W_n(w)$ with $l\in L_{n}(w)$, $\{m\}=M_{n}(w)$ and $r\in R_{n}(w)$.
\begin{enumerate}[(i)]
\item If $w$ is $b$-truncated, then $v=m$ by \ref{def:phiwn} and we see that $v$ is also $b$-truncated.
\item If $w$ is $b$-left, then $v=lm$ by \ref{def:phiwn}. If $v$ is of type $(0,0)$, then $l=b^{m_0}$ and we see that $v$ is not $b$-truncated. If $v$ is not of type $(0,0)$, then $l\neq b^{m_0}$. By \ref{lrlook} we get that $l=sb^i$ for $s\in \bar{S}\setminus \{s_1^{-1}\}$ and $i\in \mathbb{Z}$ and so we see that $v$ is $b$-truncated.
\item If $w$ is right-$b$, then $v=mr$ by \ref{def:phiwn}. If $v$ is of type $(0,0)$, then $r=b^{m_k}$ and we see that $v$ is not $b$-truncated. If $v$ is not of type $(0,0)$, then $r\neq b^{m_k}$. By \ref{lrlook} we get that $r=b^js$ for $s\neq \bar{S}\setminus \{s_k^{-1}\}$ and $j\in \mathbb{Z}$ and so we see that $v$ is $b$-truncated.
\item If $w$ be $b$-and-$b$, then $v=lmr$ by \ref{def:phiwn}. If $v$ is of type $(i,0)$ for $0\leq i\leq n$, then $l=b^{m_0}$. If $v$ is of type $(0,j)$ for $0\leq j\leq n$, then $r=b^{m_k}$. In both cases we see that $v$ is not $b$-truncated. If $v$ is not of type $(i,0)$ or $(0,j)$ for $0\leq i,j\leq n$, then $l\neq b^{m_0}$ and $r\neq b^{m_k}$. By \ref{lrlook} we get that $l=sb^i$ for $s\in \bar{S}\setminus \{s_1^{-1}\}$ and $i\in \mathbb{Z}$ and that $r=b^js$ for $s\neq \bar{S}\setminus \{s_k^{-1}\}$ and $j\in \mathbb{Z}$. So we see that $v$ is $b$-truncated.\end{enumerate}
\end{proof}

\begin{exmp}
We now want to illustrate the result \ref{ijbendWnw} in a concrete example. Again we consider the words $w_1=abab^3a$, $w_2=bab^2a$, $w_3=ab^2ab^2$ and $w_4=bab^2$. In this example we examine the rectangles $W_8(w_i)$ for $i\in \{1,2,3,4\}$ and locate the elements of $W_8(w_i)$, that are not $b$-truncated. We start with $w_1=abab^3a$. By \ref{ijbendWnw} the element of $W_8(w_1)$ is $b$-truncated. So we visualize $W_8(w_1)$ as before:

\begin{center}
\begin{tikzpicture}
\draw[step=1cm] (0,0) grid (1,1);
\node at (0.5,0.5) {(0,0)};
\end{tikzpicture}
\end{center}Now we consider $w_2=bab^2a$. By \ref{ijbendWnw} an element $v\in W_8(w_2)$ is not $b$-truncated if and only if $v$ is of type $(0,0)$. To indicate this, we colour the square $(0,0)$ black in the visualization of $W_8(w_2)$.

\begin{center}
\begin{tikzpicture}
\filldraw[fill=black!30!white] (0,8) rectangle (1,9);
\draw[step=1cm] (0,0) grid (1,9);
\node at (0.5,0.5) {(8,0)};
\node at (0.5,1.5) {(7,0)};
\node at (0.5,2.5) {(6,0)};
\node at (0.5,3.5) {(5,0)};
\node at (0.5,4.5) {(4,0)};
\node at (0.5,5.5) {(3,0)};
\node at (0.5,6.5) {(2,0)};
\node at (0.5,7.5) {(1,0)};
\node at (0.5,8.5) {(0,0)};
\end{tikzpicture}
\end{center}Now we consider $w_3=ab^2ab^2$. By \ref{ijbendWnw} an element $v\in W_8(w_3)$ is not $b$-truncated if and only if $v$ is of type $(0,0)$. We visualize as above:

\begin{center}
\begin{tikzpicture}
\filldraw[fill=black!30!white] (0,0) rectangle (1,1);
\draw[step=1cm] (0,0) grid (9,1);
\node at (0.5,0.5) {(0,0)};
\node at (1.5,0.5) {(0,1)};
\node at (2.5,0.5) {(0,2)};
\node at (3.5,0.5) {(0,3)};
\node at (4.5,0.5) {(0,4)};
\node at (5.5,0.5) {(0,5)};
\node at (6.5,0.5) {(0,6)};
\node at (7.5,0.5) {(0,7)};
\node at (8.5,0.5) {(0,8)};
\end{tikzpicture}
\end{center}Finally we consider $w_4=bab^2$. Then by \ref{ijbendWnw} an element $v\in W_8(w_4)$ is not $b$-truncated if and only if $v$ is of type $(i,0)$ or $(0,j)$ for $0\leq i,j\leq 8$. We visualize as above:

\begin{center}
\begin{tikzpicture}
\filldraw[fill=black!30!white] (0,8) rectangle (9,9);
\filldraw[fill=black!30!white] (0,0) rectangle (1,9);
\draw[step=1cm] (0,0) grid (9,9);

\node at (0.5,8.5) {(0,0)};
\node at (1.5,8.5) {(0,1)};
\node at (2.5,8.5) {(0,2)};
\node at (3.5,8.5) {(0,3)};
\node at (4.5,8.5) {(0,4)};
\node at (5.5,8.5) {(0,5)};
\node at (6.5,8.5) {(0,6)};
\node at (7.5,8.5) {(0,7)};
\node at (8.5,8.5) {(0,8)};

\node at (0.5,7.5) {(1,0)};
\node at (1.5,7.5) {(1,1)};
\node at (2.5,7.5) {(1,2)};
\node at (3.5,7.5) {(1,3)};
\node at (4.5,7.5) {(1,4)};
\node at (5.5,7.5) {(1,5)};
\node at (6.5,7.5) {(1,6)};
\node at (7.5,7.5) {(1,7)};
\node at (8.5,7.5) {(1,8)};

\node at (0.5,6.5) {(2,0)};
\node at (1.5,6.5) {(2,1)};
\node at (2.5,6.5) {(2,2)};
\node at (3.5,6.5) {(2,3)};
\node at (4.5,6.5) {(2,4)};
\node at (5.5,6.5) {(2,5)};
\node at (6.5,6.5) {(2,6)};
\node at (7.5,6.5) {(2,7)};
\node at (8.5,6.5) {(2,8)};

\node at (0.5,5.5) {(3,0)};
\node at (1.5,5.5) {(3,1)};
\node at (2.5,5.5) {(3,2)};
\node at (3.5,5.5) {(3,3)};
\node at (4.5,5.5) {(3,4)};
\node at (5.5,5.5) {(3,5)};
\node at (6.5,5.5) {(3,6)};
\node at (7.5,5.5) {(3,7)};
\node at (8.5,5.5) {(3,8)};

\node at (0.5,4.5) {(4,0)};
\node at (1.5,4.5) {(4,1)};
\node at (2.5,4.5) {(4,2)};
\node at (3.5,4.5) {(4,3)};
\node at (4.5,4.5) {(4,4)};
\node at (5.5,4.5) {(4,5)};
\node at (6.5,4.5) {(4,6)};
\node at (7.5,4.5) {(4,7)};
\node at (8.5,4.5) {(4,8)};

\node at (0.5,3.5) {(5,0)};
\node at (1.5,3.5) {(5,1)};
\node at (2.5,3.5) {(5,2)};
\node at (3.5,3.5) {(5,3)};
\node at (4.5,3.5) {(5,4)};
\node at (5.5,3.5) {(5,5)};
\node at (6.5,3.5) {(5,6)};
\node at (7.5,3.5) {(5,7)};
\node at (8.5,3.5) {(5,8)};

\node at (0.5,2.5) {(6,0)};
\node at (1.5,2.5) {(6,1)};
\node at (2.5,2.5) {(6,2)};
\node at (3.5,2.5) {(6,3)};
\node at (4.5,2.5) {(6,4)};
\node at (5.5,2.5) {(6,5)};
\node at (6.5,2.5) {(6,6)};
\node at (7.5,2.5) {(6,7)};
\node at (8.5,2.5) {(6,8)};

\node at (0.5,1.5) {(7,0)};
\node at (1.5,1.5) {(7,1)};
\node at (2.5,1.5) {(7,2)};
\node at (3.5,1.5) {(7,3)};
\node at (4.5,1.5) {(7,4)};
\node at (5.5,1.5) {(7,5)};
\node at (6.5,1.5) {(7,6)};
\node at (7.5,1.5) {(7,7)};
\node at (8.5,1.5) {(7,8)};

\node at (0.5,0.5) {(8,0)};
\node at (1.5,0.5) {(8,1)};
\node at (2.5,0.5) {(8,2)};
\node at (3.5,0.5) {(8,3)};
\node at (4.5,0.5) {(8,4)};
\node at (5.5,0.5) {(8,5)};
\node at (6.5,0.5) {(8,6)};
\node at (7.5,0.5) {(8,7)};
\node at (8.5,0.5) {(8,8)};
\end{tikzpicture}
\end{center}
\end{exmp}

Now we start proving results about the length of squares. These results look quite technical at a first view, but we hope to have made the intuition behind them clear in the outline of the program of this subsection. We will also give concrete examples following the proofs. The results also might give the impression of being unnecessary complicated. This is in fact true if we only consider the goal of this subsection, but they will be very useful in Subsection \ref{T-1speedsingle} and Section \ref{Tsums}. The basic tool to compute with the length of squares is the following statement:

\begin{lem}\label{ijdiff}
Let $w\equiv_S (m_0,s_1,...,s_k,m_k)$ be a reduced word and assume that $w$ is not a power of $b$.\begin{enumerate}[(i)]
\item Assume there are $u_1,u_2\in W_n(w)$ such that $u_1$ is of type $(i_1,j_1)$ and $u_2$ is of type $(i_2,j_2)$. Then there are elements\begin{align} l_{i_1}&\in dL_{i_1}(w),\\
l_{i_2}&\in dL_{i_2}(w),\\
r_{j_1}&\in dR_{j_1}(w),\\
r_{j_2}&\in dR_{j_2}(w).\end{align} We have \begin{equation}\label{ijdiff1} \lvert (i_1,j_1)\rvert_{W_n(w)}-\lvert (i_2,j_2)\rvert_{W_n(w)}=\lvert l_{i_1} \rvert_S-\lvert l_{i_2} \rvert_S+\lvert r_{j_1} \rvert_S-\lvert r_{j_2} \rvert_S.\end{equation}
\item Assume there are $u_1\in W_{n_1}(w)$ and $u_2\in W_{n_2}(w)$ such that $u_1$ and $u_2$ are of type $(i,j)$. Then there are elements \begin{align} M_{n_1}(w)&=\{m_{n_1}\},\\
M_{n_2}(w)&=\{m_{n_2}\}.\end{align} We have \begin{equation}\label{ijdiff2}\lvert (i,j)\rvert_{W_{n_1}(w)}-\lvert (i,j)\rvert_{W_{n_2}(w)}=\lvert m_{n_1} \rvert_S-\lvert m_{n_2} \rvert_S.\end{equation}
\end{enumerate} 
\end{lem}

\begin{proof}
Assume we have $l\in dL_{i}(w)$, $\{m\}=M_{n}(w)$ and $r\in dR_{j}$. Then $u=lmr\in W_n(w)$ is of type $(i,j)$. By \ref{def:ijlength} and by \ref{lmrred} we have \begin{equation} \lvert (i,j)\rvert_{W_n(w)}=\lvert u \rvert_S=\lvert l \rvert_S+\lvert m \rvert_S+\lvert r \rvert_S.\end{equation} By \ref{ijlength} we have for every $l_{i}\in dL_{i}(w)$ and $r_{j}\in dR_{j}$ \begin{equation}\label{ij=l+m+r} \lvert (i,j)\rvert_{W_n(w)}=\lvert l_i \rvert_S+\lvert m \rvert_S+\lvert r_j \rvert_S.\end{equation} The statements follow:\begin{enumerate}[(i)]
\item Let $\{m_n\}=M_{n}(w)$. By \ref{ij=l+m+r} we have \begin{equation}\begin{split} \lvert (i_1,j_1)\rvert_{W_n(w)}-\lvert (i_2,j_2)\rvert_{W_n(w)}=&(\lvert l_{i_1} \rvert_S+\lvert m_n \rvert_S+\lvert r_{j_1} \rvert_S)\\
&-(\lvert l_{i_2} \rvert_S+\lvert m_n \rvert_S+\lvert r_{j_2} \rvert_S)\\
=&\lvert l_{i_1} \rvert_S-\lvert l_{i_2} \rvert_S+\lvert r_{j_1} \rvert_S-\lvert r_{j_2} \rvert_S.\end{split}\end{equation}
\item Let $l\in dL_{i}(w)$ and $r\in dR_{j}$. By \ref{ij=l+m+r} we have \begin{equation}\begin{split}\lvert (i,j)\rvert_{W_{n_1}(w)}-\lvert (i,j)\rvert_{W_{n_2}(w)}=&(\lvert l_i \rvert_S+\lvert m_{n_1} \rvert_S+\lvert r_j \rvert_S)\\
&-(\lvert l_i \rvert_S+\lvert m_{n_2} \rvert_S+\lvert r_j \rvert_S)\\
=&\lvert m_{n_1} \rvert_S-\lvert m_{n_2} \rvert_S.\end{split}\end{equation}
\end{enumerate}
\end{proof}

Now we can prove the first of the promised lemmas on the length of squares. We recall the statement given in the outline of this subsection: Let $w\equiv_S (m_0,s_1,...,s_k,m_k)$ be a reduced word and assume that $w$ is not a power of $b$. Let $N=\max (\lvert m_0\rvert, \lvert m_k\rvert)$ and assume $n>N$. Then we draw a red line under the squares whose horizontal coordinate is $\lvert m_0\rvert$ if $m_0\neq 0$ and we draw a blue line right of the squares whose vertical coordinate is $\lvert m_k\rvert$ if $m_k\neq 0$. Imagine you start at the upper, left corner of a big enough rectangle and want to walk to the lower, right corner. Then you have to cross the lines we drew before. After you crossed the lines the length of the squares you walk on grows with every step down or to the right by 1. To prove the statement we have to formulate it in a more technical way:

\begin{lem}\label{LRgrow}
Let $w\equiv_S (m_0,s_1,...,s_k,m_k)$ be a reduced word and assume that $w$ is not a power of $b$. Let $N=\max (\lvert m_0\rvert, \lvert m_k\rvert)$. Let $n>N$. Assume there are $u_1,u_2\in W_n(w)$ such that $u_1$ is of type $(i_1,j_1)$ and $u_2$ is of type $(i_2,j_2)$. Then we have the following:\begin{enumerate}[(i)]
\item \label{LRgrow1} Assume that $j_1=j_2$ and $\max(1,\lvert m_0\rvert)\leq i_1, i_2\leq n$. Then \begin{equation}\lvert (i_1,j_1)\rvert_{W_n(w)}-\lvert (i_2,j_2)\rvert_{W_n(w)}=i_1-i_2. \end{equation}
\item \label{LRgrow2} Assume that $i_1=i_2$ and $\max(1,\lvert m_k\rvert)\leq j_1, j_2\leq n$. Then \begin{equation}\lvert (i_1,j_1)\rvert_{W_n(w)}-\lvert (i_2,j_2)\rvert_{W_n(w)}=j_1-j_2.\end{equation}
\item Assume that $\max(1,N)\leq i_1, i_2, j_1, j_2\leq n$. Then \begin{equation}\lvert (i_1,j_1)\rvert_{W_n(w)}-\lvert (i_2,j_2)\rvert_{W_n(w)}=i_1-i_2+j_1-j_2.\end{equation}\end{enumerate}
\end{lem}

\begin{proof}
Let $l_{i_1}\in dL_{i_1}(w)$, $l_{i_2}\in dL_{i_2}(w)$, $r_{j_1}\in dR_{j_1}(w)$ and $r_{j_2}\in dR_{j_2}(w)$. By \ref{ijdiff1} we have \begin{equation} \lvert (i_1,j_1)\rvert_{W_n(w)}-\lvert (i_2,j_2)\rvert_{W_n(w)}=\lvert l_{i_1} \rvert_S-\lvert l_{i_2} \rvert_S+\lvert r_{j_1} \rvert_S-\lvert r_{j_2} \rvert_S.\end{equation} Now we prove the three statements:\begin{enumerate}[(i)]
\item Since $j_1=j_2$, we have by \ref{ijlength} that $\lvert r_{j_1} \rvert_S-\lvert r_{j_2} \rvert_S=0$. So it is enough to compute $\lvert l_{i_1} \rvert_S-\lvert l_{i_2}\rvert_S$. For every $\max(1,\lvert m_0\rvert)\leq i$ we can define $i'=i-\lvert m_0\rvert$. Since $i>0$, we get by \ref{LRlength} that\begin{equation}
\lvert l_{i} \rvert_S=\begin{cases}
\mathmakebox[\lenf][l]{1+i',}& \text{if } m_0>0, s_1\neq a^{-1}\\
2m_0+i',& \text{if } m_0>0, s_1=a^{-1}\\
-2m_0+i',& \text{if } m_0<0, s_1\neq a^{-1}\\
1+i',& \text{if } m_0<0, s_1=a^{-1}
\end{cases}.
\end{equation} It follows that \begin{equation} \lvert l_{i_1} \rvert_S-\lvert l_{i_2} \rvert_S=i_1'-i_2'=i_1-i_2.\end{equation}
\item Since $i_1=i_2$, we have by \ref{ijlength} that $\lvert l_{i_1} \rvert_S-\lvert l_{i_2} \rvert_S=0$. So it is enough to compute $\lvert r_{j_1} \rvert_S-\lvert r_{j_2} \rvert_S$. For every $\max(1,\lvert m_0\rvert)\leq j$ we can define $j'=j-\lvert m_k\rvert$. Since $j>0$, we get by \ref{LRlength} that\begin{equation}\lvert r_j \rvert_S=\begin{cases}
\mathmakebox[\lenf][l]{1+j',}& \text{if } m_k>0, s_k=a\\
2m_k+j',& \text{if } m_k>0, s_k\neq a\\
-2m_k+j',& \text{if } m_k<0, s_k=a\\
1+j',& \text{if } m_k<0, s_k\neq a
\end{cases}.\end{equation} It follows that \begin{equation} \lvert r_{j_1} \rvert_S-\lvert r_{j_2} \rvert_S=j_1'-j_2'=j_1-j_2.\end{equation}
\item Using \ref{LRgrow1} and \ref{LRgrow2} we get that \begin{equation}\begin{split}
\lvert (i_1,j_1)\rvert_{W_n(w)}-\lvert (i_2,j_2)\rvert_{W_n(w)}
=&\lvert (i_1,j_1)\rvert_{W_n(w)}-\lvert (i_2,j_1)\rvert_{W_n(w)}\\
&+\lvert (i_2,j_1)\rvert_{W_n(w)}-\lvert (i_2,j_2)\rvert_{W_n(w)}\\
=& i_1-i_2+ j_1-j_2
\end{split}\end{equation}
\end{enumerate}
\end{proof}

\begin{exmp}
We want to visualize the result of \ref{LRgrow} in a concrete example: Again we consider the words $w_1=abab^3a$, $w_2=bab^2a$, $w_3=ab^2ab^2$ and $w_4=bab^2$. For convenience we set again $n=8$, i.e. we consider the rectangles $W_8(w_i)$ for $i\in \{1,2,3,4\}$. Since the rectangle $W_8(w_1)$ consists of just one square $(0,0)$, we can not apply the result. We move on to $w_2=bab^2a\equiv_S (1,a,2,a,0)$. Above we visualized $W_8(w_2)$ as a column with a black top square. Now we add a red horizontal line under the square whose horizontal coordinate is $\lvert m_0\rvert=1$. By \ref{LRgrow} the distance of any two squares under this red line measures the difference in length of these squares, where the lower square has the bigger length.

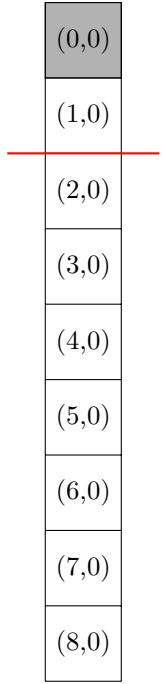
\begin{center}
\begin{tikzpicture}
\filldraw[fill=black!30!white] (0,8) rectangle (1,9);
\draw[step=1cm] (0,0) grid (1,9);
\node at (0.5,0.5) {(8,0)};
\node at (0.5,1.5) {(7,0)};
\node at (0.5,2.5) {(6,0)};
\node at (0.5,3.5) {(5,0)};
\node at (0.5,4.5) {(4,0)};
\node at (0.5,5.5) {(3,0)};
\node at (0.5,6.5) {(2,0)};
\node at (0.5,7.5) {(1,0)};
\node at (0.5,8.5) {(0,0)};
\draw[thick, red] (-0.5,7)--(1.5,7);
\end{tikzpicture}
\end{center}

Now we consider $w_3=ab^2ab^2\equiv_S (0,a,2,a,2)$. Above we visualized $W_8(w)$ as a row with a black leftmost square. Now we add a blue vertical line right of the square whose vertical coordinate is $\lvert m_k\rvert=2$. By \ref{LRgrow} the distance of any two squares right of this blue line measures the difference in length of these squares, where the square on the right of the other has the bigger length.

\begin{center}
\begin{tikzpicture}
\filldraw[fill=black!30!white] (0,0) rectangle (1,1);
\draw[step=1cm] (0,0) grid (9,1);
\node at (0.5,0.5) {(0,0)};
\node at (1.5,0.5) {(0,1)};
\node at (2.5,0.5) {(0,2)};
\node at (3.5,0.5) {(0,3)};
\node at (4.5,0.5) {(0,4)};
\node at (5.5,0.5) {(0,5)};
\node at (6.5,0.5) {(0,6)};
\node at (7.5,0.5) {(0,7)};
\node at (8.5,0.5) {(0,8)};
\draw[thick, blue] (3,-0.5)--(3,1.5);
\end{tikzpicture}
\end{center}

Lastly we consider $w_4=bab^2\equiv_S (1,a,2)$. We visualized $W_8(w_4)$ before as a box, where the leftmost and highest squares are coloured black. Now we draw a red horizontal and a blue vertical line as above: The red horizontal line is drawn under the squares whose horizontal coordinate is $\lvert m_0\rvert=1$, the blue vertical line is drawn right of the squares whose vertical coordinate is $\lvert m_k\rvert=2$. Here \ref{LRgrow} gives the following results: If two square with the same vertical coordinate are under the red line, the distance of these two squares measures the difference in length of these squares, where the lower square has the bigger length. If two square with the same horizontal coordinate are right of the blue line, the distance of these two squares measures the difference in length of these squares, where the square on the right of the other has the bigger length. In particular the distance of any two squares under the red line and right of the blue line measures the difference in length of these squares.

\begin{center}
\begin{tikzpicture}
\filldraw[fill=black!30!white] (0,8) rectangle (9,9);
\filldraw[fill=black!30!white] (0,0) rectangle (1,9);
\draw[step=1cm] (0,0) grid (9,9);

\node at (0.5,8.5) {(0,0)};
\node at (1.5,8.5) {(0,1)};
\node at (2.5,8.5) {(0,2)};
\node at (3.5,8.5) {(0,3)};
\node at (4.5,8.5) {(0,4)};
\node at (5.5,8.5) {(0,5)};
\node at (6.5,8.5) {(0,6)};
\node at (7.5,8.5) {(0,7)};
\node at (8.5,8.5) {(0,8)};

\node at (0.5,7.5) {(1,0)};
\node at (1.5,7.5) {(1,1)};
\node at (2.5,7.5) {(1,2)};
\node at (3.5,7.5) {(1,3)};
\node at (4.5,7.5) {(1,4)};
\node at (5.5,7.5) {(1,5)};
\node at (6.5,7.5) {(1,6)};
\node at (7.5,7.5) {(1,7)};
\node at (8.5,7.5) {(1,8)};

\node at (0.5,6.5) {(2,0)};
\node at (1.5,6.5) {(2,1)};
\node at (2.5,6.5) {(2,2)};
\node at (3.5,6.5) {(2,3)};
\node at (4.5,6.5) {(2,4)};
\node at (5.5,6.5) {(2,5)};
\node at (6.5,6.5) {(2,6)};
\node at (7.5,6.5) {(2,7)};
\node at (8.5,6.5) {(2,8)};

\node at (0.5,5.5) {(3,0)};
\node at (1.5,5.5) {(3,1)};
\node at (2.5,5.5) {(3,2)};
\node at (3.5,5.5) {(3,3)};
\node at (4.5,5.5) {(3,4)};
\node at (5.5,5.5) {(3,5)};
\node at (6.5,5.5) {(3,6)};
\node at (7.5,5.5) {(3,7)};
\node at (8.5,5.5) {(3,8)};

\node at (0.5,4.5) {(4,0)};
\node at (1.5,4.5) {(4,1)};
\node at (2.5,4.5) {(4,2)};
\node at (3.5,4.5) {(4,3)};
\node at (4.5,4.5) {(4,4)};
\node at (5.5,4.5) {(4,5)};
\node at (6.5,4.5) {(4,6)};
\node at (7.5,4.5) {(4,7)};
\node at (8.5,4.5) {(4,8)};

\node at (0.5,3.5) {(5,0)};
\node at (1.5,3.5) {(5,1)};
\node at (2.5,3.5) {(5,2)};
\node at (3.5,3.5) {(5,3)};
\node at (4.5,3.5) {(5,4)};
\node at (5.5,3.5) {(5,5)};
\node at (6.5,3.5) {(5,6)};
\node at (7.5,3.5) {(5,7)};
\node at (8.5,3.5) {(5,8)};

\node at (0.5,2.5) {(6,0)};
\node at (1.5,2.5) {(6,1)};
\node at (2.5,2.5) {(6,2)};
\node at (3.5,2.5) {(6,3)};
\node at (4.5,2.5) {(6,4)};
\node at (5.5,2.5) {(6,5)};
\node at (6.5,2.5) {(6,6)};
\node at (7.5,2.5) {(6,7)};
\node at (8.5,2.5) {(6,8)};

\node at (0.5,1.5) {(7,0)};
\node at (1.5,1.5) {(7,1)};
\node at (2.5,1.5) {(7,2)};
\node at (3.5,1.5) {(7,3)};
\node at (4.5,1.5) {(7,4)};
\node at (5.5,1.5) {(7,5)};
\node at (6.5,1.5) {(7,6)};
\node at (7.5,1.5) {(7,7)};
\node at (8.5,1.5) {(7,8)};

\node at (0.5,0.5) {(8,0)};
\node at (1.5,0.5) {(8,1)};
\node at (2.5,0.5) {(8,2)};
\node at (3.5,0.5) {(8,3)};
\node at (4.5,0.5) {(8,4)};
\node at (5.5,0.5) {(8,5)};
\node at (6.5,0.5) {(8,6)};
\node at (7.5,0.5) {(8,7)};
\node at (8.5,0.5) {(8,8)};

\draw[thick, red] (-0.5,7)--(9.5,7);
\draw[thick, blue] (3,-0.5)--(3,9.5);
\end{tikzpicture}
\end{center}
\end{exmp}

Now we want to prove the second of the promised lemmas on the length of squares. Again we recall the statement given in the outline of this subsection before proving a more technical formulation. Let $w\equiv_S (m_0,s_1,...,s_k,m_k)$ be a reduced word and assume that $w$ is not a power of $b$. Let $N=\max (\lvert m_0\rvert, \lvert m_k\rvert)$ and assume $n>2N$. Then we colour the squares red whose horizontal coordinate is $2\lvert m_0\rvert$ if $m_0\neq 0$ and we colour the squares blue whose vertical coordinate is $2\lvert m_k\rvert$ if $m_k\neq 0$. Imagine you start at the upper, left corner of a big enough rectangle and want to walk to the lower, right corner. Then you have to walk over squares we have coloured before. Once you walked over these squares the length of the square you stand on is bigger than the length of any square you walked over before.

\begin{lem}\label{LRanlauf}
Let $w\equiv_S (m_0,s_1,...,s_k,m_k)$ be a reduced word and assume that $w$ is not a power of $b$. Let $N=\max (\lvert m_0\rvert, \lvert m_k\rvert)$. Then we get for $n\geq 2N$:\begin{enumerate}[(i)]
\item Let $m_0\neq 0$. Assume that $dR_{j}(w)\neq \varnothing$. Then for $0\leq i<2\lvert m_0\rvert$ we get \begin{equation}\lvert (2\lvert m_0\rvert,j)\rvert_{W_n(w)}>\lvert (i,j)\rvert_{W_n(w)}.\end{equation}
\item Let $m_k\neq 0$. Assume that $dL_{i}(w)\neq \varnothing$. Then for $0\leq j<2\lvert m_k\rvert$ we get \begin{equation}\lvert (i,2\lvert m_k\rvert)\rvert_{W_n(w)}>\lvert (i,j)\rvert_{W_n(w)}.\end{equation}
\end{enumerate} 
\end{lem}

\newlength{\leng}
\settowidth{\leng}{$-2m_k-j'$,}

\begin{proof}
This proof consists mainly in a few computations using \ref{LRlength}. \begin{enumerate}[(i)]
\item Since $m_0\neq 0$, we know that $1\leq \lvert m_0\rvert$. So if $\lvert m_0\rvert\leq i\leq n$, then we have by \ref{LRgrow1} that \begin{equation}\lvert (2\lvert m_0\rvert,j)\rvert_{W_n(w)}-\lvert (i,j)\rvert_{W_n(w)}=2\lvert m_0\rvert-i>0.\end{equation} We can assume from now that $i <\lvert m_0\rvert$. Using \ref{ijdiff1} we get as in the proof of \ref{LRgrow1}, that it is enough to compute $\lvert l_{2\lvert m_0\rvert} \rvert_S-\lvert l_{i}\rvert_S$ for $l_{2\lvert m_0\rvert}\in dL_{2\lvert m_0\rvert}(w)$ and $l_{i}\in dL_{i}(w)$. We do this by calculating the respective length explicitely:\begin{enumerate}
\item Let $i=0$. Then by \ref{LRlength} we get $\lvert l_{i}\rvert_S=\lvert m_0\rvert$
\item Let $1\leq i<\lvert m_0\rvert$. Then we can define $i'=\lvert m_0\rvert-i$. We have $1<i'<\lvert m_0\rvert$. We get by \ref{LRlength} that \begin{equation}\lvert l_i \rvert_S=\begin{cases}
\mathmakebox[\leng][l]{1+i',}& \text{if } m_0>0, s_1\neq a^{-1}\\
2m_0-i',& \text{if } m_0>0, s_1=a^{-1}\\
-2m_0-i',& \text{if } m_0<0, s_1\neq a^{-1}\\
1+i',& \text{if } m_0<0, s_1=a^{-1}
\end{cases}.\end{equation} 
\item We get by \ref{LRlength} that \begin{equation}\lvert l_{2\lvert m_0\rvert} \rvert_S=\begin{cases}
\mathmakebox[\leng][l]{1+m_0,}& \text{if } m_0>0, s_1\neq a^{-1}\\
3m_0,& \text{if } m_0>0, s_1=a^{-1}\\
-3m_0,& \text{if } m_0<0, s_1\neq a^{-1}\\
1-m_0,& \text{if } m_0<0, s_1=a^{-1}
\end{cases}.\end{equation} 
\end{enumerate} It follows that \begin{equation} \lvert (2\lvert m_0\rvert,j)\rvert_{W_n(w)}-\lvert (i,j)\rvert_{W_n(w)}=\lvert l_{2\lvert m_0\rvert} \rvert_S-\lvert l_i \rvert_S>0.\end{equation}
\item Since $m_k\neq 0$, we know that $1\leq \lvert m_k\rvert$. So if $\lvert m_k\rvert\leq j\leq n$, then we have by \ref{LRgrow2} that \begin{equation}\lvert (i,2\lvert m_k\rvert)\rvert_{W_n(w)}-\lvert (i,j)\rvert_{W_n(w)}=2\lvert m_k\rvert-j>0.\end{equation} We can assume from now that $j <\lvert m_k\rvert$. Using \ref{ijdiff1} we get as in the proof of \ref{LRgrow2}, that it is enough to compute $\lvert r_{2\lvert m_k\rvert} \rvert_S-\lvert r_{j}\rvert_S$ for $r_{2\lvert m_k\rvert}\in dR_{2\lvert m_k\rvert}(w)$ and $r_{j}\in dR_{j}(w)$. We do this by calculating the respective length explicitely:\begin{enumerate}
\item Let $j=0$. Then by \ref{LRlength} we get $\lvert r_{j}\rvert_S=\lvert m_k\rvert$
\item Let $1\leq j<\lvert m_k\rvert$. Then we can define $j'=\lvert m_k\rvert-j$. We have $1<j'<\lvert m_k\rvert$. We get by \ref{LRlength} that \begin{equation}\lvert r_j \rvert_S=\begin{cases}
\mathmakebox[\leng][l]{1+j',}& \text{if } m_k>0, s_k=a\\
2m_k-j',& \text{if } m_k>0, s_k\neq a\\
-2m_k-j',& \text{if } m_k<0, s_k=a\\
1+j',& \text{if } m_k<0, s_k\neq a
\end{cases}.\end{equation}
\item We get by \ref{LRlength} that \begin{equation}\lvert r_{2\lvert m_k\rvert} \rvert_S=\begin{cases}
\mathmakebox[\leng][l]{1+m_k,}& \text{if } m_k>0, s_k=a\\
3m_k,& \text{if } m_k>0, s_k\neq a\\
-3m_k,& \text{if } m_k<0, s_k=a\\
1-m_k,& \text{if } m_k<0, s_k\neq a
\end{cases}.\end{equation}
\end{enumerate} It follows that \begin{equation} \lvert (i,2\lvert m_k\rvert)\rvert_{W_n(w)}-\lvert (i,j)\rvert_{W_n(w)}=\lvert r_{2\lvert m_k\rvert} \rvert_S-\lvert r_j \rvert_S>0.\end{equation}\end{enumerate}
\end{proof}

\begin{exmp}
Now we want to visualize the result of \ref{LRanlauf} in a concrete example. For convenience we stay with the same example as before: We consider the rectangles $W_8(w_i)$ for $i\in \{1,2,3,4\}$, where $w_1=abab^3a$, $w_2=bab^2a$, $w_3=ab^2ab^2$ and $w_4=bab^2$. Since the rectangle $W_8(w_1)$ consists of just one square $(0,0)$, we can not apply the result. We move on to $w_2=bab^2a\equiv_S (1,a,2,a,0)$. We visualized $W_8(w_2)$ as a column with a black top square and a red line due to \ref{LRgrow}. Now we colour the square red whose horizontal coordinate is $2\lvert m_0\rvert=2$. By \ref{LRanlauf} the length of this square is bigger than of all of the squares above him.

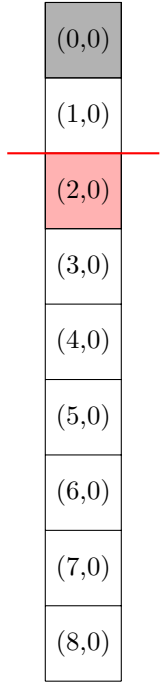
\begin{center}
\begin{tikzpicture}
\filldraw[fill=black!30!white] (0,8) rectangle (1,9);
\filldraw[fill=red!30!white] (0,6) rectangle (1,7);
\draw[step=1cm] (0,0) grid (1,9);
\node at (0.5,0.5) {(8,0)};
\node at (0.5,1.5) {(7,0)};
\node at (0.5,2.5) {(6,0)};
\node at (0.5,3.5) {(5,0)};
\node at (0.5,4.5) {(4,0)};
\node at (0.5,5.5) {(3,0)};
\node at (0.5,6.5) {(2,0)};
\node at (0.5,7.5) {(1,0)};
\node at (0.5,8.5) {(0,0)};
\draw[thick, red] (-0.5,7)--(1.5,7);
\end{tikzpicture}
\end{center}

Now we consider $w_3=ab^2ab^2\equiv_S (0,a,2,a,2)$. We visualized $W_8(w_3)$ as a row with a black left square and a blue line due to \ref{LRgrow}. Now we colour the square blue whose vertical coordinate is $2\lvert m_k\rvert=4$. By \ref{LRanlauf} the length of this square is bigger than of all of the squares left of him.

\begin{center}
\begin{tikzpicture}
\filldraw[fill=black!30!white] (0,0) rectangle (1,1);
\filldraw[fill=blue!30!white] (4,0) rectangle (5,1);
\draw[step=1cm] (0,0) grid (9,1);
\node at (0.5,0.5) {(0,0)};
\node at (1.5,0.5) {(0,1)};
\node at (2.5,0.5) {(0,2)};
\node at (3.5,0.5) {(0,3)};
\node at (4.5,0.5) {(0,4)};
\node at (5.5,0.5) {(0,5)};
\node at (6.5,0.5) {(0,6)};
\node at (7.5,0.5) {(0,7)};
\node at (8.5,0.5) {(0,8)};
\draw[thick, blue] (3,-0.5)--(3,1.5);
\end{tikzpicture}
\end{center}

Lastly we consider $w_4=bab^2\equiv_S (1,a,2)$. We visualized $W_8(w_4)$ as a box, where the leftmost and highest squares are coloured black with a red and a blue line due to \ref{LRgrow}. Parallel to the red line of \ref{LRgrow}, we now colour every square with horizontal coordinate $2\lvert m_0\rvert=2$ in red. Parallel to the blue line of \ref{LRgrow}, we colour every square with vertical coordinate $2\lvert m_k\rvert=4$ in blue. By \ref{LRanlauf} the length of a red square is bigger than the length of the squares above it. Likewise the length of a blue square is bigger than the length of the squares left of him.

\begin{center}
\begin{tikzpicture}
\filldraw[fill=black!30!white] (0,8) rectangle (9,9);
\filldraw[fill=black!30!white] (0,0) rectangle (1,9);
\filldraw[fill=red!30!white] (0,6) rectangle (4,7);
\filldraw[fill=red!30!white] (5,6) rectangle (9,7);
\filldraw[fill=blue!30!white] (4,0) rectangle (5,6);
\filldraw[fill=blue!30!white] (4,7) rectangle (5,9);
\filldraw[fill=blue!50!red!50!white] (4,6) rectangle (5,7);
\draw[step=1cm] (0,0) grid (9,9);

\node at (0.5,8.5) {(0,0)};
\node at (1.5,8.5) {(0,1)};
\node at (2.5,8.5) {(0,2)};
\node at (3.5,8.5) {(0,3)};
\node at (4.5,8.5) {(0,4)};
\node at (5.5,8.5) {(0,5)};
\node at (6.5,8.5) {(0,6)};
\node at (7.5,8.5) {(0,7)};
\node at (8.5,8.5) {(0,8)};

\node at (0.5,7.5) {(1,0)};
\node at (1.5,7.5) {(1,1)};
\node at (2.5,7.5) {(1,2)};
\node at (3.5,7.5) {(1,3)};
\node at (4.5,7.5) {(1,4)};
\node at (5.5,7.5) {(1,5)};
\node at (6.5,7.5) {(1,6)};
\node at (7.5,7.5) {(1,7)};
\node at (8.5,7.5) {(1,8)};

\node at (0.5,6.5) {(2,0)};
\node at (1.5,6.5) {(2,1)};
\node at (2.5,6.5) {(2,2)};
\node at (3.5,6.5) {(2,3)};
\node at (4.5,6.5) {(2,4)};
\node at (5.5,6.5) {(2,5)};
\node at (6.5,6.5) {(2,6)};
\node at (7.5,6.5) {(2,7)};
\node at (8.5,6.5) {(2,8)};

\node at (0.5,5.5) {(3,0)};
\node at (1.5,5.5) {(3,1)};
\node at (2.5,5.5) {(3,2)};
\node at (3.5,5.5) {(3,3)};
\node at (4.5,5.5) {(3,4)};
\node at (5.5,5.5) {(3,5)};
\node at (6.5,5.5) {(3,6)};
\node at (7.5,5.5) {(3,7)};
\node at (8.5,5.5) {(3,8)};

\node at (0.5,4.5) {(4,0)};
\node at (1.5,4.5) {(4,1)};
\node at (2.5,4.5) {(4,2)};
\node at (3.5,4.5) {(4,3)};
\node at (4.5,4.5) {(4,4)};
\node at (5.5,4.5) {(4,5)};
\node at (6.5,4.5) {(4,6)};
\node at (7.5,4.5) {(4,7)};
\node at (8.5,4.5) {(4,8)};

\node at (0.5,3.5) {(5,0)};
\node at (1.5,3.5) {(5,1)};
\node at (2.5,3.5) {(5,2)};
\node at (3.5,3.5) {(5,3)};
\node at (4.5,3.5) {(5,4)};
\node at (5.5,3.5) {(5,5)};
\node at (6.5,3.5) {(5,6)};
\node at (7.5,3.5) {(5,7)};
\node at (8.5,3.5) {(5,8)};

\node at (0.5,2.5) {(6,0)};
\node at (1.5,2.5) {(6,1)};
\node at (2.5,2.5) {(6,2)};
\node at (3.5,2.5) {(6,3)};
\node at (4.5,2.5) {(6,4)};
\node at (5.5,2.5) {(6,5)};
\node at (6.5,2.5) {(6,6)};
\node at (7.5,2.5) {(6,7)};
\node at (8.5,2.5) {(6,8)};

\node at (0.5,1.5) {(7,0)};
\node at (1.5,1.5) {(7,1)};
\node at (2.5,1.5) {(7,2)};
\node at (3.5,1.5) {(7,3)};
\node at (4.5,1.5) {(7,4)};
\node at (5.5,1.5) {(7,5)};
\node at (6.5,1.5) {(7,6)};
\node at (7.5,1.5) {(7,7)};
\node at (8.5,1.5) {(7,8)};

\node at (0.5,0.5) {(8,0)};
\node at (1.5,0.5) {(8,1)};
\node at (2.5,0.5) {(8,2)};
\node at (3.5,0.5) {(8,3)};
\node at (4.5,0.5) {(8,4)};
\node at (5.5,0.5) {(8,5)};
\node at (6.5,0.5) {(8,6)};
\node at (7.5,0.5) {(8,7)};
\node at (8.5,0.5) {(8,8)};

\draw[thick, red] (-0.5,7)--(9.5,7);
\draw[thick, blue] (3,-0.5)--(3,9.5);
\end{tikzpicture}
\end{center}
\end{exmp}

Now having proved the Lemmas \ref{LRgrow} and \ref{LRanlauf}, we are finally prepared to show that the $b$-truncated end of the n-support is not empty. Recall that the $b$-truncated end of the n-support of a word $w$ is \begin{equation}\begin{split}
E_b(W_n(w))&=\{v\in W_n(w)\mid \lvert v\rvert_S>n_b(W_n(w))\}\\
&=\{v\in W_n(w)\mid \lvert v\rvert_S>\max \{\lvert v\rvert_S\mid v\in W_n(w), \tau_b(v)\neq v\}\}.\end{split}\end{equation}

\begin{lem}\label{btruncendWnw}
Let $w\equiv_S (m_0,s_1,...,s_k,m_k)$ be a reduced word and assume that $w$ is not a power of $b$. Let $N=\max (\lvert m_0\rvert, \lvert m_k\rvert)$. Then for $n>2N$, we get $E_b(W_n(w))\neq \varnothing$.
\end{lem}

\begin{proof}
Suppose otherwise. Then there is $n>2N$ and $x\in W_n(w)$ such that $x$ is not $b$-truncated and for every $v\in W_n(w)$ we have $\lvert v\rvert_S\leq \lvert x\rvert_S$.\begin{enumerate}
\item Let $w$ be $b$-truncated. Then by \ref{ijbendWnw} we get that the element of $W_n(w)$ is $b$-truncated. This is a contradiction.
\item Let $w$ be $b$-left. Then by \ref{ijbendWnw} we get that $x$ is of type $(0,0)$. Let $v\in W_n(w)$ be of type $(2\lvert m_0\rvert,0)$. Then by \ref{LRanlauf} we get that $\lvert v\rvert_S=\lvert (2\lvert m_0\rvert,0)\rvert_{W_n(w)}>\lvert (0,0)\rvert_{W_n(w)}=\lvert x\rvert_S.$ This is a contradiction.
\item Let $w$ be right-$b$. Then by \ref{ijbendWnw} we get that $x$ is of type $(0,0)$. Let $v\in W_n(w)$ be of type $(0,2\lvert m_k\rvert)$. Then by \ref{LRanlauf} we get that $\lvert v\rvert_S=\lvert (0,2\lvert m_k\rvert)\rvert_{W_n(w)}>\lvert (0,0)\rvert_{W_n(w)}=\lvert x\rvert_S.$ This is a contradiction.
\item Let $w$ be $b$-and-$b$. Then by \ref{ijbendWnw} we get that $x$ is of type $(i,0)$ or $(0,j)$ for $0\leq i,j\leq n$. We consider two cases seperately:\begin{enumerate}
\item Let $x$ be of type $(i,0)$. Then let $v\in W_n(w)$ be of type $(i,2\lvert m_k\rvert)$. We get by \ref{LRanlauf} that $\lvert v\rvert_S=\lvert (i,2\lvert m_k\rvert)\rvert_{W_n(w)}>\lvert (i,0)\rvert_{W_n(w)}=\lvert x\rvert_S.$ This is a contradiction. 
\item Let $x$ be of type $(0,j)$. Then let $v\in W_n(w)$ be of type $(2\lvert m_0\rvert,j)$. We get by \ref{LRanlauf} that $\lvert v\rvert_S=\lvert (2\lvert m_0\rvert,j)\rvert_{W_n(w)}>\lvert (0,j)\rvert_{W_n(w)}=\lvert x\rvert_S.$ This is a contradiction. 
\end{enumerate}
\end{enumerate}
\end{proof}

\begin{thm}\label{tentacle}
Let $w\equiv_S (m_0,s_1,...,s_k,m_k)$ be a reduced word and assume that $w$ is not a power of $b$. Let $N=\max (\lvert m_0\rvert, \lvert m_k\rvert)$. Then $(\phi w)_n$ is reduced for all $n>2N$.
\end{thm}

\begin{proof}
By \ref{btruncendWnw} we get that for $n>2N$ there is $v\in E_b(W_n(w))$. We have by \ref{phiwn} \[(\phi w)_n=\sum_{u\in W_{n}(w)}\alpha(u)\phi u\ \text{with}\ \alpha(u)=\begin{cases}
1,& \text{if } u\in W^{+}_{n}(w)\\
-1,& \text{if } u\in W^{-}_{n}(w)\end{cases}.\] It follows that $\alpha(v)\neq 0$. This implies by \ref{redtree} that $(\phi w)_n$ is reduced.
\end{proof}

\subsection{The speed of $T^{-1}$}\label{T-1speedsingle}

Let $w\equiv_S (m_0,s_1,...,s_k,m_k)$ be a reduced word and assume that $w$ is not a power of $b$. In this subsection we want to compute the speed of $T^{-1}$ on $[\phi w]$. We proved in Subsection \ref{studynsupportw} that the n-representative of $[\phi w]$ is reduced if $n$ is big enough. So by Proposition \ref{phiwnequiv} it is enough to know the length of the n-representative of $[\phi w]$ to compute the speed of $T^{-1}$ on $[\phi w]$. With our results on the length of squares, it is not hard to collect information about $\lVert(\phi w)_n\rVert_S$: We will show in \ref{nnbiggest} that for $n$ big enough the square of length $\lVert(\phi w)_n\rVert_S$ is in the lower, right corner of the rectangle. We colour this square green in the visualization.

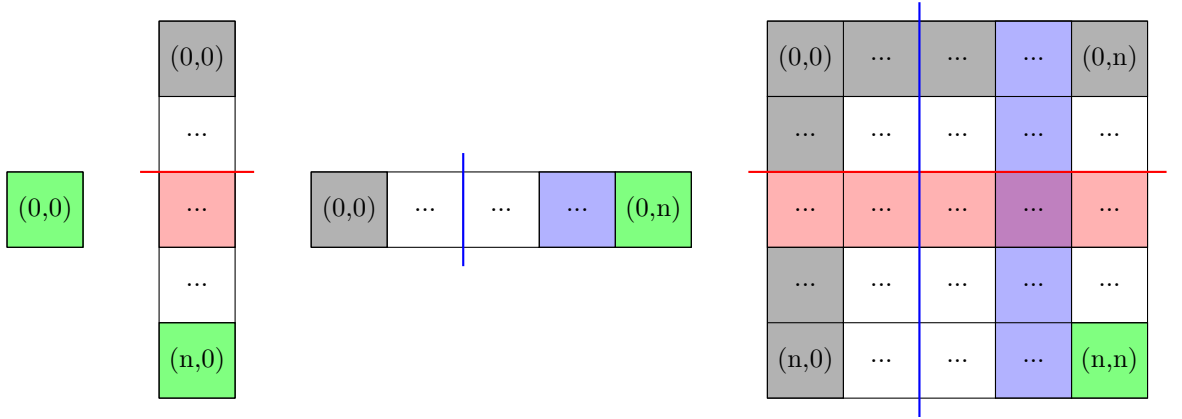
\begin{center}
\begin{tikzpicture}
\filldraw[fill=green!50!white] (0,2) rectangle (1,3);
\draw[step=1cm] (0,2) grid (1,3);
\node at (0.5,2.5) {(0,0)};

\filldraw[fill=green!50!white] (2,0) rectangle (3,1);
\filldraw[fill=red!30!white] (2,2) rectangle (3,3);
\filldraw[fill=black!30!white] (2,4) rectangle (3,5);
\draw[step=1cm] (2,0) grid (3,5);
\node at (2.5,0.5) {(n,0)};
\node at (2.5,1.5) {...};
\node at (2.5,2.5) {...};
\node at (2.5,3.5) {...};
\node at (2.5,4.5) {(0,0)};
\draw[thick, red] (1.75,3)--(3.25,3);

\filldraw[fill=green!50!white] (8,2) rectangle (9,3);
\filldraw[fill=blue!30!white] (7,2) rectangle (8,3);
\filldraw[fill=black!30!white] (4,2) rectangle (5,3);
\draw[step=1cm] (4,2) grid (9,3);
\node at (4.5,2.5) {(0,0)};
\node at (5.5,2.5) {...};
\node at (6.5,2.5) {...};
\node at (7.5,2.5) {...};
\node at (8.5,2.5) {(0,n)};
\draw[thick, blue] (6,1.75)--(6,3.25);

\filldraw[fill=green!50!white] (14,0) rectangle (15,1);
\filldraw[fill=black!30!white] (10,4) rectangle (15,5);
\filldraw[fill=black!30!white] (10,4) rectangle (11,0);
\filldraw[fill=red!30!white] (10,2) rectangle (13,3);
\filldraw[fill=red!30!white] (14,2) rectangle (15,3);
\filldraw[fill=blue!30!white] (13,0) rectangle (14,2);
\filldraw[fill=blue!30!white] (13,3) rectangle (14,5);
\filldraw[fill=blue!50!red!50!white] (13,2) rectangle (14,3);
\draw[step=1cm] (10,0) grid (15,5);

\node at (10.5,4.5) {(0,0)};
\node at (11.5,4.5) {...};
\node at (12.5,4.5) {...};
\node at (13.5,4.5) {...};
\node at (14.5,4.5) {(0,n)};

\node at (10.5,3.5) {...};
\node at (11.5,3.5) {...};
\node at (12.5,3.5) {...};
\node at (13.5,3.5) {...};
\node at (14.5,3.5) {...};

\node at (10.5,2.5) {...};
\node at (11.5,2.5) {...};
\node at (12.5,2.5) {...};
\node at (13.5,2.5) {...};
\node at (14.5,2.5) {...};

\node at (10.5,1.5) {...};
\node at (11.5,1.5) {...};
\node at (12.5,1.5) {...};
\node at (13.5,1.5) {...};
\node at (14.5,1.5) {...};

\node at (10.5,0.5) {(n,0)};
\node at (11.5,0.5) {...};
\node at (12.5,0.5) {...};
\node at (13.5,0.5) {...};
\node at (14.5,0.5) {(n,n)};
\draw[thick, red] (9.75,3)--(15.25,3);
\draw[thick, blue] (12,-0.25)--(12,5.25);
\end{tikzpicture}
\end{center}

\begin{thm}\label{nnbiggest}
Let $w\equiv_S (m_0,s_1,...,s_k,m_k)$ be a reduced word and assume that $w$ is not a power of $b$. Let $N=\max (\lvert m_0\rvert, \lvert m_k\rvert)$. Then we get for $n>2N$:
\begin{enumerate}[(i)]
\item Let $w$ be $b$-truncated. Then $\lvert(\phi w)_n\rvert_S=\lvert (0,0)\rvert_{W_n(w)}$.
\item Let $w$ be $b$-left. Then $\lvert(\phi w)_n\rvert_S=\lvert (n,0)\rvert_{W_n(w)}$.
\item Let $w$ be right-$b$. Then $\lvert(\phi w)_n\rvert_S=\lvert (0,n)\rvert_{W_n(w)}$.
\item Let $w$ be $b$-and-$b$. Then $\lvert(\phi w)_n\rvert_S=\lvert (n,n)\rvert_{W_n(w)}$.
\end{enumerate}
\end{thm}

\begin{proof}
By \ref{tentacle} we know that $\lvert(\phi w)_n\rvert_S=\lVert(\phi w)_n\rVert_S$ for $n>2N$. So it is enough to prove the four equalities for $\lVert(\phi w)_n\rVert_S$. Assume $\lVert(\phi w)_n\rVert_S=\lvert (i,j)\rvert_{W_n(w)}$ for $0\leq i,j\leq n$. \begin{enumerate}[(i)]
\item Let $w$ be $b$-truncated. Then by \ref{4shapes} $(i,j)=(0,0)$.
\item Let $w$ be $b$-left. Then by \ref{4shapes} $(i,j)=(i,0)$. \begin{enumerate}
\item If $\lvert m_0\rvert\leq i<n$, we get by \ref{LRgrow} that \begin{equation}\lvert (n,0)\rvert_{W_n(w)}-\lvert (i,0)\rvert_{W_n(w)}=n-i>0.\end{equation} This contradicts $\lVert(\phi w)_n\rVert_S=\lvert (i,j)\rvert_{W_n(w)}$.
\item If $i<2\lvert m_0\rvert$, we get by \ref{LRanlauf} that \begin{equation}\lvert (2\lvert m_0\rvert,0)\rvert_{W_n(w)}>\lvert (i,0)\rvert_{W_n(w)}.\end{equation} This contradicts $\lVert(\phi w)_n\rVert_S=\lvert (i,j)\rvert_{W_n(w)}$.
\end{enumerate} It follows that $i=n$.
\item Let $w$ be right-$b$. Then by \ref{4shapes} $(i,j)=(0,j)$. \begin{enumerate}
\item If $\lvert m_k\rvert\leq j<n$, we get by \ref{LRgrow} that \begin{equation}\lvert (0,n)\rvert_{W_n(w)}-\lvert (0,j)\rvert_{W_n(w)}=n-j>0.\end{equation} This contradicts $\lVert(\phi w)_n\rVert_S=\lvert (i,j)\rvert_{W_n(w)}$.
\item If $j<2\lvert m_k\rvert$, we get by \ref{LRanlauf} that \begin{equation}\lvert (0,2\lvert m_k\rvert)\rvert_{W_n(w)}>\lvert (0,j)\rvert_{W_n(w)}.\end{equation} This contradicts $\lVert(\phi w)_n\rVert_S=\lvert (i,j)\rvert_{W_n(w)}$.
\end{enumerate} It follows that $j=n$.
\item Let $w$ be $b$-and-$b$. We can prove $i=n$ and $j=n$ exactly as above.\end{enumerate} 
\end{proof}

Now we are in shape to compute the speed of $T^{-1}$ on $[\phi w]$ explicitly:

\begin{defn}\label{def:sp}
Let $w\equiv_S (m_0,s_1,...,s_k,m_k)$ be a reduced word and assume that $w$ is not a power of $b$. We define \begin{align}
A^+&=\{j\in \{1,...,l-1\}\mid s_j\neq a, s_{j+1}=a^{-1}\},\\
A^-&=\{j\in \{1,...,l-1\}\mid s_j=a, s_{j+1}\neq a^{-1}\}.
\end{align} Using this we define \begin{equation}
sp(w)=\lvert A^+\rvert+\lvert A^-\rvert+\begin{cases}
0,& \text{if $w$ is $b$-truncated}\\
1,& \text{if $w$ is $b$-left or right-$b$}\\
2,& \text{if $w$ is $b$-and-$b$}
\end{cases}.\end{equation}
\end{defn}

\begin{prop}\label{diff=sp}
Let $w\equiv_S (m_0,s_1,...,s_k,m_k)$ be a reduced word and assume that $w$ is not a power of $b$. Let $N=2\max\{\lvert m_j\rvert \mid j\in \{0,...,k\}\}$. Then for all $n>N$ we have \begin{equation}
\lvert(\phi w)_{n+1}\rvert_S-\lvert(\phi w)_{n}\rvert_S=sp(w).
\end{equation}
\end{prop}

\begin{proof}
For $n>N$ we set $M_{n}(w)=\{m_{n}\}$.\begin{enumerate}
\item Let $w$ be $b$-truncated. Then for $n>N$ we get by \ref{nnbiggest} that $\lvert(\phi w)_n\rvert_S=\lvert (0,0)\rvert_{W_n(w)}$. By \ref{ijdiff2} we get for $n>N$ that \begin{equation}\begin{split}
\lvert (\phi w)_{n+1}\rvert_S-\lvert(\phi w)_{n}\rvert_S&=\lvert (0,0)\rvert_{W_{n+1}(w)}-\lvert (0,0)\rvert_{W_n(w)}\\
&=\lvert m_{n+1} \rvert_S-\lvert m_{n} \rvert_S
\end{split}\end{equation}
\item Let $w$ be $b$-left. Then for $n>N$ we get by \ref{nnbiggest} that $\lvert(\phi w)_n\rvert_S=\lvert (n,0)\rvert_{W_n(w)}$. Then by \ref{LRgrow1} and \ref{ijdiff2} we get for $n>N$ that\begin{equation}\begin{split}
\lvert(\phi w)_{n+1}\rvert_S-\lvert(\phi w)_{n}\rvert_S=&\lvert (n+1,0)\rvert_{W_{n+1}(w)}-\lvert (n,0)\rvert_{W_n(w)}\\
=&\lvert (n+1,0)\rvert_{W_{n+1}(w)}-\lvert (n,0)\rvert_{W_{n+1}(w)}\\
&+\lvert (n,0)\rvert_{W_{n+1}(w)}-\lvert (n,0)\rvert_{W_n(w)}\\
=&n+1-n+\lvert m_{n+1} \rvert_S-\lvert m_{n} \rvert_S\\
=&1+\lvert m_{n+1} \rvert_S-\lvert m_{n} \rvert_S
\end{split}\end{equation} 
\item Let $w$ be right-$b$. Then for $n>N$ we get by \ref{nnbiggest} that $\lvert(\phi w)_n\rvert_S=\lvert (0,n)\rvert_{W_n(w)}$. Then by \ref{LRgrow2} and \ref{ijdiff2} we get for $n>N$ that\begin{equation}\begin{split}
\lvert(\phi w)_{n+1}\rvert_S-\lvert(\phi w)_{n}\rvert_S=&\lvert (0,n+1)\rvert_{W_{n+1}(w)}-\lvert (0,n)\rvert_{W_n(w)}\\
=&\lvert (0,n+1)\rvert_{W_{n+1}(w)}-\lvert (0,n)\rvert_{W_{n+1}(w)}\\
&+\lvert (0,n)\rvert_{W_{n+1}(w)}-\lvert (0,n)\rvert_{W_n(w)}\\
=&n+1-n+\lvert m_{n+1} \rvert_S-\lvert m_{n} \rvert_S\\
=&1+\lvert m_{n+1} \rvert_S-\lvert m_{n} \rvert_S
\end{split}\end{equation} 
\item Let $w$ be $b$-and-$b$. Then for $n>N$ we get by \ref{nnbiggest} that $\lvert(\phi w)_n\rvert_S=\lvert (n,n)\rvert_{W_n(w)}$. Then by \ref{LRgrow} and \ref{ijdiff2} we get for $n>N$ that\begin{equation}\begin{split}
\lvert(\phi w)_{n+1}\rvert_S-\lvert(\phi w)_{n}\rvert_S=&\lvert (n+1,n+1)\rvert_{W_{n+1}(w)}-\lvert (n,n)\rvert_{W_n(w)}\\
=&\lvert (n+1,n+1)\rvert_{W_{n+1}(w)}-\lvert (n,n)\rvert_{W_{n+1}(w)}\\
&+\lvert (n,n)\rvert_{W_{n+1}(w)}-\lvert (n,n)\rvert_{W_n(w)}\\
=&(n+1-n)+(n+1-n)+\lvert m_{n+1} \rvert_S-\lvert m_{n} \rvert_S\\
=&2+\lvert m_{n+1} \rvert_S-\lvert m_{n} \rvert_S
\end{split}\end{equation}\end{enumerate} It remains to show that \begin{equation}
\lvert m_{n+1} \rvert_S-\lvert m_{n} \rvert_S=\lvert A^+\rvert+\lvert A^-\rvert.
\end{equation} For easier notation, define $A^0=\{1,...,k-1\}\setminus (A^+\cup A^-)$. By \ref{btrsubwords} we get that \begin{equation}\begin{split}
\lvert m_{n} \rvert_S=& k+\sum_{j=1}^{k-1} \lvert m_j+n(\#a(s_j)-\#a^{-1}(s_j))\rvert\\
=& k+\sum_{j\in A^+} \lvert m_j+n\rvert+\sum_{j\in A^-} \lvert m_j-n\rvert+\sum_{j\in A^0} \lvert m_j\rvert\end{split}\end{equation} Since $n>N$, we know that $m_j+n>0$ and $m_j-n<0$ for all $j\in \{1,...,k-1\}$. So we have \begin{equation}\begin{split} \lvert m_{n+1} \rvert_S-\lvert m_{n} \rvert_S=& (k+\sum_{j\in A^+} (m_j+n+1)-\sum_{j\in A^-} (m_j-n-1)+\sum_{j\in A^0} \lvert m_j\rvert)\\
&-(k+\sum_{j\in A^+} (m_j+n)-\sum_{j\in A^-} (m_j-n)+\sum_{j\in A^0} \lvert m_j\rvert)\\
=& \sum_{j\in A^+} 1 + \sum_{j\in A^-} 1\\
=& \lvert A^+\rvert+\lvert A^-\rvert \end{split}.\end{equation}
\end{proof}

\begin{cor}\label{speed=sp}
Let $w\equiv_S (m_0,s_1,...,s_k,m_k)$ be a reduced word and assume that $w$ is not a power of $b$. Then $\mathrm{sp}(T^{-1},[\phi w])=sp(w)$ for every counting function $\phi w$. In particular the speed of $T^{-1}$ on $[\phi w]$ exists. 
\end{cor}

\begin{proof}
By \ref{diff=sp} there is $N\in \mathbb{N}$ such that $\lvert(\phi w)_{n+1}\rvert_S-\lvert(\phi w)_{n}\rvert_S=sp(w)$ for all $n>N$. Then \begin{equation}\begin{split}\lim_n \frac{\lvert T^{-n}[\phi w]\rvert_S}{n}
&=\lim_n \frac{\lvert (\phi w)_n\rvert_S}{n}\\
&=\lim_n \frac{(n-N)sp(w)+\lvert (\phi w)_N\rvert_S}{n}\\
&=sp(w)+\lim_n \frac{-Nsp(w)+\lvert (\phi w)_N\rvert_S}{n}\\
&=sp(w).\end{split}\end{equation}
\end{proof}

\begin{exmp}
We now can quite compute the speed of $T^{-1}$ on $[\phi w]$ as long as $w$ is not a power of $b$. We give three examples for $w\in F_2$:\begin{enumerate}
\item Let $w=aba^{-1}b^7a$. Then $sp(T^{-1},[\phi w])=0$.
\item Let $w=b^3a^5$. Then $sp(T^{-1},[\phi w])=5$.
\item Let $w=baba^2b^{-2}$. Then $sp(T^{-1},[\phi w])=4$.
\end{enumerate}
\end{exmp}

\begin{cor}\label{T-1strongly}
Let $w\equiv_S (m_0,s_1,...,s_k,m_k)$ be a reduced word and assume that $w$ is not a power of $b$. Then $T^{-1}$ is strongly of linear speed on $[\phi w]$ if and only if $sp(w)>0$.
\end{cor}

\begin{proof}
This follows from \ref{speed=sp}.
\end{proof}

If $w$ is not a power of $b$ or $w=b$, we know now that the speed of $T^{-1}$ on $[\phi w]$ exist and we actually can compute it explicitly. But for a general element $[f]\in \tilde{\mathcal{B}}(F_n)$ we do not know at this point if the speed of $T^{-1}$ on $[f]$ exist. (For example we do not even know how the length of $[\phi b^k]$ for some $k\in \mathbb{Z}\setminus \{0,1\}$ behaves asymptotically under applications of $T^{-1}$.) This will be the topic of Section \ref{Tsums}. Nevertheless we already can establish some results concerning the speed of $T^{-1}$ on $\tilde{\mathcal{B}}(F_n)$. The key for this is the following lemma:

\begin{lem}\label{normalform1}
Every sum of counting quasimorphisms $f=\sum_{v\in I} \alpha(v)\phi v$ is equivalent to an element $f_0=\sum_{v\in I_0} \gamma_0(v)\phi v \in \mathcal{B}(F_n,S)$ such that $v\in I_0$ and $\tau_b(v)=\varnothing$ imply $v=b$. 
\end{lem}

\begin{proof}
Let $f=\sum_{v\in I} \alpha(v)\phi v$. Assume there are $k$ elements $v\in I$ such that $\tau_b(v)=\varnothing$, but $v\neq b$. Then pick one such element $v_0$. We have that $v_0=b^{m_0}$ for some $m_0\in \mathbb{Z}\setminus \{0,1\}$. By \ref{powerofbtob} we know that there is a weight $\beta_0$ and a finite set $J_0$ such that no element of $J_0$ is a power of $b$ and \begin{equation} f\sim g_0=\sum_{v\in I\setminus \{v_0\}} \alpha(v)\phi v+\alpha(v_0)\beta_0(b)\phi b + \sum_{y\in J_0}\alpha(v_0)\beta_0(y)\phi y.\end{equation} There are at most $k-1$ elements $v\in (I\setminus \{v_0\})\cup \{b\}\cup J_0$ such that $\tau_b(v)=\varnothing$ and $v\neq b$. Repeating this argument at most $k-1$ times, we get that $f$ is equivalent to a sum of counting quasimorphisms $f_0=\sum_{v\in I_0} \gamma_0(v)\phi v$ such that $v\in I_0$ and $\tau_b(v)=\varnothing$ imply $v=b$.
\end{proof}

Our results so far culminate in the following corollaries:

\begin{cor}
$T^{-1}$ has linear speed.
\end{cor}

\begin{proof}
We have to show that $T^{-1}$ satisfies\begin{enumerate}
\item $\lvert T^{-n}[f]\rvert_S=O(n)$ for all $[f]\in \tilde{\mathcal{B}}(F_n)$.
\item There is $[f_0]\in \tilde{\mathcal{B}}(F_n)$ such that $T^{-1}$ has linear speed on $[f_0]$.
\end{enumerate} The second condition is satisfied by \ref{T-1strongly}. For the first condition we have to work a bit. Let $[f]\in \tilde{\mathcal{B}}(F_n)$ for $f=\sum_{v\in I} \alpha(v)\phi v$. By \ref{normalform1} we can assume that $v\in I$ and $\tau_b(v)=\varnothing$ imply $v=b$. So we know that the speed of $T^{-1}$ on $[\phi v]$ exists for all $v\in I$. In particular we know that $\lvert T^{-n}[\phi v]\rvert_S=O(n)$ for all $v\in I$. Then we get using the properties of $\lvert -\rvert_S$ that \begin{equation}\begin{split}
\lvert T^{-n}[f]\rvert_S=&\lvert \sum_{v\in I} \alpha(v)T^{-n}[\phi v]\rvert_S\\
\leq& \sup_{v\in I} \lvert T^{-n}[\phi v]\rvert_S\\
=&O(n)\end{split}\end{equation}
\end{proof}

\begin{cor}\label{ulseminorm}
Endow $\mathbb{R}$ with the trivial absolute value $\lvert -\rvert_0$. Then $\mathrm{usp}_S(T^{-1},-)$ is a seminorm on $\tilde{\mathcal{B}}(F_n)$.
\end{cor}

\begin{proof}
By \ref{speedseminorm} it is enough to show that $\mathrm{usp}_S(T^{-1},[f])<\infty$ for every $[f]\in \tilde{\mathcal{B}}(F_n)$. Let $[f]\in \tilde{\mathcal{B}}(F_n)$ for $f=\sum_{v\in I} \alpha(v)\phi v$. By \ref{normalform1} we can assume that $v\in I$ and $\tau_b(v)=\varnothing$ imply $v=b$. Then by \ref{speedhomogeneous} and \ref{ultraspeed} we get \begin{equation}\begin{split}
\mathrm{usp}_S(T^{-1},[f])=&\mathrm{usp}_S(T^{-1},[\sum_{v\in I} \alpha(v)\phi v])\\
\leq& \sup_{v\in I} \mathrm{usp}_S(T^{-1},[\alpha(v)\phi v])\\
=& \sup_{v\in I} \mathrm{usp}_S(T^{-1},[\phi v]).\end{split}\end{equation}
By \ref{speed=sp} we know that $\mathrm{usp}_S(T^{-1},[\phi v])<\infty$ for every $v\in I$. This completes the proof.
\end{proof}

\section{The speed of $T^{-1}$ on a sum of counting quasimorphisms}\label{Tsums}

In this section we want to describe a way to compute the speed of $T^{-1}$ on any element of $\tilde{\mathcal{B}}(F_n)$. Our goal is to find for any equivalence class $[f] \in \tilde{\mathcal{B}}(F_n)$ an element $f=\sum_{v\in I} \alpha(v)\phi v$ that is speed reduced for $T^{-1}$. Then we have reduced the problem of computing the speed of $T^{-1}$ on $[f]$ to the problem we successfully tackled in Section \ref{Tsingle}. 

To check if an element $f=\sum_{v\in I} \alpha(v)\phi v$ is speed reduced for $T^{-1}$, we need to compare the speed of $T^{-1}$ on $[f]$ with the speed of $T^{-1}$ on $[\phi v]$ for $v\in I$. By the results in Section \ref{Tsingle} we can compute the speed of $T^{-1}$ on $[\phi v]$ as long as $v\in I$ with $\tau_b(v)=\varnothing$ implies $v=b$. This gives us a first condition on $f$. We get a second condition on $f$ by our need to understand the speed of $T^{-1}$ on $[f]$: We need $f$ to give us a good representative of $T^{-n}[f]$, that enables us to say something about $\lvert T^{-n}[f]\rvert_S$.

To make the presentation a bit easier we turn this natural order of the section up side down: We will start by stating the conditions on $f$ in the first subsection. Then we will show in the second and third subsection that every $f$ satisfying these conditions gives us a good representative of $T^{-n}[f]$ and ultimately that $f$ is speed reduced for $T^{-1}$.

\subsection{Finding the n-representative}

Let $f=\sum_{v\in I} \alpha(v)\phi v$ be a sum of counting quasimorphisms such that $v\in I$ with $\tau_b(v)=\varnothing$ implies $v=b$. Then using results from Section \ref{Tsingle} it is easy to find a representative $f_n$ of $T^{-n}[f]$.

\begin{defn}\label{def:fn}
Let $f=\sum_{v\in I} \alpha(v)\phi v$ be a sum of counting quasimorphisms such that $v\in I$ with $\tau_b(v)=\varnothing$ implies $v=b$. Then we denote \[f_n=\sum_{v\in I} \alpha(v)(\phi v)_n.\]
\end{defn}

\begin{cor}\label{fnequiv}
Let $f=\sum_{v\in I} \alpha(v)\phi v$ be a sum of counting quasimorphisms such that $v\in I$ with $\tau_b(v)=\varnothing$ implies $v=b$. Then $f\circ T^n$ is equivalent to $f_n$.
\end{cor}

\begin{proof}
By \ref{phibnequiv} we know that $\phi b\circ T^n$ is equivalent to $(\phi b)_n$. By \ref{phiwnequiv} we know that $\phi v\circ T^n$ is equivalent to $(\phi v)_n$ if $v$ is not a power of $b$. So we get \begin{equation}f\circ T^n=\sum_{v\in I} \alpha(v)\phi v\circ T^n\sim\sum_{v\in I} \alpha(v)(\phi v)_n=f_n.\end{equation}
\end{proof}

This gives us a representative of $T^{-n}[f]$ for every $[f] \in \tilde{\mathcal{B}}(F_n)$: By \ref{normalform1} we always find a sum of counting quasimorphisms $f_0=\sum_{v\in I_0} \gamma_0(v)\phi v$ equivalent to $f$ such that $v\in I_0$ and $\tau_b(v)=\varnothing$ imply $v=b$. Then $(f_0)_n$ is a representative of $T^{-n}[f]$ by \ref{fnequiv}. For convenience we will in the following always consider $f_0$ instead of $f$, i.e. we assume that $f=\sum_{v\in I} \alpha(v)\phi v$ has the property that $v\in I$ with $\tau_b(v)=\varnothing$ implies $v=b$. So $f_n$ is a representative for $T^{-n}[f]$. We now reformulate $f_n$ as one sum of counting quasimorphisms.

\begin{defn}
Let $I$ be a finite subset of $F_n\setminus \{e\}$ such that $v\in I$ with $\tau_b(v)=\varnothing$ implies $v=b$. Then we call $W_n(I)=\bigcup_{v\in I} W_n(v)$ the \emph{n-support} of $I$.
\end{defn}

To visualize the n-support of a set $I$ we use the visualization of the n-support of every $v\in I$. We see $W_n(I)$ as a union of possibly intersecting rectangles.

\begin{exmp}
Let $I=\{bab^2,bab^2a,ab^2ab^2,abab^3a\}\subset F_2$. Then $W_n(I)$ is the union of the box $W_n(bab^2)$, the column $W_n(bab^2a)$, the row $W_n(ab^2ab^2)$ and the cell $W_n(abab^3a)$.
\end{exmp}

\begin{cor}\label{fn}
Let $f=\sum_{v\in I} \alpha(v)\phi v$ be a sum of counting quasimorphisms such that $v\in I$ with $\tau_b(v)=\varnothing$ implies $v=b$. Define \[\alpha_n(u)=\sum_{v\in I}\alpha(v)v_n(u).\] Then we have $f_n=\sum_{u\in W_n(I)} \alpha_n(u)\phi u$.
\end{cor}

\begin{proof}
Using \ref{phiwn} and \ref{phibnequiv} we get \begin{equation}\begin{split}
f_n=& \sum_{v\in I} \alpha(v)(\phi v)_n\\
=&\sum_{v\in I} \alpha(v)(\sum_{u\in W_{n}(v)}v_n(u)\phi u)\\
=&\sum_{u\in W_n(I)}\sum_{v\in I}\alpha(v)v_n(u)\phi u\\
=&\sum_{u\in W_n(I)} \alpha_n(u)\phi u
\end{split}\end{equation}
\end{proof}

While the representative $f_n$ is good for some $f$, it is unfortunately bad in many cases. It can fail to be good in two ways: Firstly for some $f\in \mathcal{B}(F_n,S)$ there is no $N\in\mathbb{N}$ such that $f_n$ is reduced for $n>N$.

\begin{exmp}
Let $f=\phi a+ \sum_{s\in \bar{S}\setminus \{a^{-1}\}}\phi as$. Then $\lim_{n\to \infty} \lVert f_n \rVert_S=\infty$, but $\lvert f_n \rvert_S=0$ for all $n\in \mathbb{N}$ by a right-extension relation.
\end{exmp}

Secondly even if we know that $f_n$ is reduced for $f=\sum_{v\in I} \alpha(v)\phi v$ and $n$ big enough, we may have a hard time saying something about the reduced length of $f_n$. The reason is that we can not just study the the n-support of $I$ to study the length of $f_n$ as in Section \ref{Tsingle}: In general we know only that the support of $\alpha_n$ is included in the n-support of $I$, but it is not necessarily equal. This is true even if $I=\mathrm{supp}(\alpha)$.

\begin{exmp}
Let $f=\phi b^{-2}a-\phi b^{-1}a+\phi ab^{-1}a$. We have \begin{equation}\begin{split}
W_2(\{b^{-2}a,b^{-1}a,ab^{-1}a\})=&W_2(b^{-2}a)\cup W_2(b^{-1}a)\cup W_2(ab^{-1}a)\\
=&\{b^{-2}a,ab^{-2}a,ab^{-3}a\}\cup \{b^{-1}a,ab^{-1}a,ab^{-2}a\}\cup \{ab^{-3}a\}\\
=&\{b^{-2}a,ab^{-2}a,ab^{-3}a,b^{-1}a,ab^{-1}a\}.
\end{split}\end{equation} In particular the element of maximal length in $W_2(\{b^{-2}a,b^{-1}a,ab^{-1}a\})$ has length $5$. But on the other hand \begin{equation}\begin{split}
f_2=&\phi b^{-2}a-\phi ab^{-2}a-\phi ab^{-3}a\\
&-\phi b^{-1}a+\phi ab^{-1}a+\phi ab^{-2}a\\
&+\phi ab^{-3}a\\
=&\phi b^{-2}a-\phi b^{-1}a+\phi ab^{-1}a.
\end{split}\end{equation} So $\mathrm{supp}(\alpha_2)=\{b^{-2}a, b^{-1}a, ab^{-1}a\}$ and in particular $\lVert f_2 \rVert_S=3$.
\end{exmp}

So we have to impose a further condition on the representative $f$ of $[f] \in \tilde{\mathcal{B}}(F_n)$ to make sure that $f_n$ is a good representative of $T^{-n}[f]$. This condition will ensure that $f_n$ is reduced for $n$ big enough and that we can say something about $\lVert f_n\rVert_S$.

\begin{defn}
We say that $v,w\in F_n\setminus \{e\}$ are of \emph{one kind} if $v,w$ are both $b$-truncated or both $b$-left or both right-$b$ or both $b$-and-$b$.
\end{defn}

\begin{defn}
We say that $v,w\in F_n\setminus \{e\}$ are of \emph{opposite kind} if $v,w$ are both $b$-truncated, or both $b$-and-$b$, or one of them is $b$-left and the other one is right-$b$.
\end{defn}

\begin{defn}
Let $I$ be a finite subset of $F_n\setminus \{e\}$. We say $I$ is in \emph{normal form} if the following requirements are satisfied: \begin{enumerate}
\item Let $v\in I$ and $\tau_b(v)=\varnothing$. Then $v=b$.
\item Let $v,w\in I$ be of one kind and let $\tau_b(v)=\tau_b(w)$. Then $v=w$.
\item Let $v, w$ be of opposite kind and let $\tau_b(v^{-1})=\tau_b(w)$. If $v\in I$, then $w\notin I$.
\end{enumerate} Let $f=\sum_{v\in I} \alpha(v)\phi v$ be a sum of counting quasimorphisms. We say $f$ is in \emph{normal form} if $\mathrm{supp}(\alpha)$ is in normal form. In the following we will always demand wlog that $I$ is in normal form, if $f=\sum_{v\in I} \alpha(v)\phi v$ is in normal form.
\end{defn}

The first condition is necessary to make sure that $f_n$ is defined. We admit that at the current stage it is not clear at all why the second and third condition make sure that $f_n$ is a good representative of $T^{-n}[f]$. A full explanation of this fact will be given in the following subsections. Now we can only give the idea: We consider $W_n(I)$ as a union of rectangles and by Section \ref{Tsingle} we understand one rectangle on its own very well. So our problem is essentially controlling the intersection of rectangles. We will see in \ref{kindintempty} that the second condition gives us that the n-supports of different elements of one kind do not intersect. The third condition guarantees that this stays true, if we use $\phi v=-\phi v^{-1}$. So considering $f$ in normal form, we reduce the problem to controlling intersections of the n-supports of elements which are not of the same kind. This is in fact an easier problem, because we can use the different shapes of cells, columns, rows and boxes.

Now want to prove that every sum of counting quasimorphisms is equivalent to a sum of counting quasimorphisms in normal form. This is unfortunately very technical, the idea is to use a lot of left- and right-extension relations. We start with the following lemma.

\begin{lem}\label{rewrite}
Let $x=x_1...x_t$ be a $b$-truncated word.\begin{enumerate}[(i)]
\item\label{rewrite1} Let $m_k\neq 0$, $n_k\neq 0$ and $m_0\in \mathbb{Z}$. Then \begin{equation}\phi b^{m_0}xb^{m_k}\sim \alpha(b^{m_0}xb^{n_k})\phi b^{m_0}xb^{n_k}+\sum_{v\in J}\alpha(v)\phi v,\end{equation} where $\alpha$ is a weight on $T_n$ and $J$ is a finite set such that $y\in J$ implies $y=b^{m_0}y'$ for some $b$-truncated word $y'$.
\item\label{rewrite2} Let $m_0\neq 0$, $n_0\neq 0$ and $m_k\in \mathbb{Z}$. Then \begin{equation}\phi b^{m_0}xb^{m_k}\sim \alpha(b^{n_0}xb^{m_k})\phi b^{n_0}xb^{m_k}+\sum_{v\in J}\alpha(v)\phi v,\end{equation} where $\alpha$ is a weight on $T_n$ and $J$ is a finite set such that $y\in J$ implies $y=y'b^{m_k}$ for some $b$-truncated word $y'$.
\end{enumerate}
\end{lem}

\begin{proof}
Assume \ref{rewrite1} is true for $m_k>n_k$, i.e. assume we have $\alpha$ and $J$ as in \ref{rewrite1}. Then \begin{equation}\phi b^{m_0}xb^{n_k}\sim \alpha(b^{m_0}xb^{n_k})^{-1}\phi b^{m_0}xb^{m_k}+\sum_{v\in J}-\alpha(v)\alpha(b^{m_0}xb^{n_k})^{-1}\phi v.\end{equation} So we get that that \ref{rewrite1} is true for $n_k>m_k$. It is therefore enough to prove \ref{rewrite1} for $m_k>n_k$. We consider three cases seperately:

\begin{enumerate}
\item Let $m_k>n_k>0$. Using right-extension relations $m_k-n_k$ times we get:\begin{equation}\label{r1}\begin{split}\phi b^{m_0}xb^{m_k}&\sim \phi b^{m_0}xb^{m_k-1}-\sum_{s\in S_b}\phi b^{m_0}xb^{m_k-1}s\\
&\sim ...\\
&\sim \phi b^{m_0}xb^{m_k-(m_k-n_k)}-\sum_{i=1}^{m_k-n_k}\sum_{s\in S_b}\phi b^{m_0}xb^{m_k-i}s\\
&=\phi b^{m_0}xb^{n_k}-\sum_{i=n_k}^{m_k-1}\sum_{s\in S_b}\phi b^{m_0}xb^{i}s.\end{split}\end{equation}

\item Let $0>m_k>n_k$. Using right-extension relations $m_k-n_k$ times we get:\begin{equation}\label{r2}\begin{split}\phi b^{m_0}xb^{n_k}&\sim \phi b^{m_0}xb^{n_k+1}-\sum_{s\in S_b}\phi b^{m_0}xb^{n_k+1}s\\
&\sim ...\\
&\sim \phi b^{m_0}xb^{n_k+(m_k-n_k)}-\sum_{i=1}^{m_k-n_k}\sum_{s\in S_b}\phi b^{m_0}xb^{n_k+i}s\\
&=\phi b^{m_0}xb^{m_k}-\sum_{i=n_k+1}^{m_k}\sum_{s\in S_b}\phi b^{m_0}xb^{i}s.\end{split}\end{equation} This shows \begin{equation}\phi b^{m_0}xb^{m_k}\sim\phi b^{m_0}xb^{n_k}+\sum_{i=n_k+1}^{m_k}\sum_{s\in S_b}\phi b^{m_0}xb^{i}s.\end{equation}

\item Let $m_k>0>n_k$. By \ref{r1} we know \begin{equation}\phi b^{m_0}xb^{m_k}\sim\phi b^{m_0}xb^{1}-\sum_{i=1}^{m_k-1}\sum_{s\in S_b}\phi b^{m_0}xb^is.\end{equation} Using right-extension relations, we get \begin{equation}\phi b^{m_0}xb^{1}\sim \phi b^{m_0}x-\phi b^{m_0}xb^{-1}-\sum_{s\in S_b\setminus\{x_t^{-1}\}}\phi b^{m_0}xs.\end{equation} By \ref{r2}, we know \begin{equation}\phi b^{m_0}xb^{-1}\sim\phi b^{m_0}xb^{n_k}+\sum_{i=n_n+1}^{-1}\sum_{s\in S_b}\phi b^{m_0}xb^{i}s.\end{equation} Together these three statements give us\begin{equation}\label{r3}\begin{split}\phi b^{m_0}xb^{m_k}
&\sim\phi b^{m_0}xb^{1}-\sum_{i=1}^{m_k-1}\sum_{s\in S_b}\phi b^{m_0}xb^is\\
&\sim\phi b^{m_0}x-\phi b^{m_0}xb^{-1}-\sum_{s\in S_b\setminus\{x_t^{-1}\}}\phi b^{m_0}xs-\sum_{i=1}^{m_k-1}\sum_{s\in S_b}\phi b^{m_0}xb^is\\
&\sim\phi b^{m_0}x-\phi b^{m_0}xb^{n_k}\\
&\quad -\sum_{i=n_k+1}^{-1}\sum_{s\in S_b}\phi b^{m_0}xb^{i}s-\sum_{s\in S_b\setminus\{x_t^{-1}\}}\phi b^{m_0}xs-\sum_{i=1}^{m_k-1}\sum_{s\in S_b}\phi b^{m_0}xb^is.\end{split}\end{equation}
\end{enumerate} The proof of \ref{rewrite2} is completely analogous, using left-extension relations instead of right-extension relations. Again it is enough to prove \ref{rewrite2} for $m_0>n_0$. We consider three cases seperately:

\begin{enumerate}
\item Let $m_0>n_0>0$. Using left-extension relations $m_0-n_0$ times we get:\begin{equation}\label{t1}\begin{split}\phi b^{m_0}xb^{m_k}&\sim \phi b^{m_0-1}xb^{m_k}-\sum_{s\in S_b}\phi sb^{m_0-1}xb^{m_k}\\
&\sim ...\\
&\sim \phi b^{m_0-(m_0-n_0)}xb^{m_k}-\sum_{i=1}^{m_0-n_0}\sum_{s\in S_b}\phi sb^{m_0-i}xb^{m_k}\\
&=\phi b^{n_0}xb^{m_k}-\sum_{i=n_0}^{m_0-1}\sum_{s\in S_b}\phi sb^ixb^{m_k}.\end{split}\end{equation}

\item Let $0>m_0>n_0$. Using left-extension relations $m_0-n_0$ times we get:\begin{equation}\label{t2}\begin{split}\phi b^{n_0}xb^{m_k}&\sim \phi b^{n_0+1}xb^{m_k}-\sum_{s\in S_b}\phi sb^{n_0+1}xb^{m_k}\\
&\sim ...\\
&\sim \phi b^{n_0+(m_0-n_0)}xb^{m_k}-\sum_{i=1}^{m_0-n_0}\sum_{s\in S_b}\phi sb^{n_0+i}xb^{m_k}\\
&=\phi b^{m_0}xb^{m_k}-\sum_{i=n_0+1}^{m_0}\sum_{s\in S_b}\phi sb^ixb^{m_k}.\end{split}\end{equation} This shows \begin{equation}\phi b^{m_0}xb^{m_k}\sim\phi b^{n_0}xb^{m_k}+\sum_{i=n_0+1}^{m_0}\sum_{s\in S_b}\phi sb^ixb^{m_k}.\end{equation}

\item Let $m_0>0>n_0$. By \ref{t1} we know \begin{equation}\phi b^{m_0}xb^{m_k}\sim\phi b^{1}xb^{m_k}-\sum_{i=1}^{m_0-1}\sum_{s\in S_b}\phi sb^ixb^{m_k}.\end{equation} Using right-extension relations, we get \begin{equation}\phi b^{1}xb^{m_k}\sim \phi xb^{m_k}-\phi b^{-1}xb^{m_k}-\sum_{s\in S_b\setminus\{x_1^{-1}\}}\phi sxb^{m_k}.\end{equation} By \ref{t2}, we know \begin{equation}\phi b^{-1}xb^{m_k}\sim\phi b^{n_0}xb^{m_k}+\sum_{i=n_0+1}^{-1}\sum_{s\in S_b}\phi sb^ixb^{m_k}.\end{equation} Together these three statements give us \begin{equation}\label{t3}\begin{split}\phi b^{m_0}xb^{m_k}
&\sim\phi b^{1}xb^{m_k}-\sum_{i=1}^{m_0-1}\sum_{s\in S_b}\phi sb^ixb^{m_k}\\
&\sim \phi xb^{m_k}-\phi b^{-1}xb^{m_k}-\sum_{s\in S_b\setminus\{x_1^{-1}\}}\phi sxb^{m_k}-\sum_{i=1}^{m_0-1}\sum_{s\in S_b}\phi sb^ixb^{m_k}\\
&\sim \phi xb^{m_k}-\phi b^{n_0}xb^{m_k}\\
&\quad -\sum_{i=n_0+1}^{-1}\sum_{s\in S_b}\phi sb^ixb^{m_k}-\sum_{s\in S_b\setminus\{x_1^{-1}\}}\phi sxb^{m_k}-\sum_{i=1}^{m_0-1}\sum_{s\in S_b}\phi sb^ixb^{m_k}.\end{split}\end{equation}
\end{enumerate}
\end{proof}

\begin{cor}\label{cor:rewrite}
Let $v,w \in F_n\setminus \{e\}$ and $\tau_b(v)=\tau_b(w)$. Then we have:\begin{enumerate}[(i)]
\item Assume $v,w$ are both $b$-left. Then there is a weight $\alpha$ and a finite set $J$ such that all elements of $J$ are $b$-truncated and \begin{equation}
\phi v\sim \alpha(w)\phi w + \sum_{y\in J}\alpha(y)\phi y\end{equation}
\item Assume $v,w$ are both right-$b$. Then there is a weight $\alpha$ and a finite set $J$ such that all elements of $J$ are $b$-truncated and \begin{equation}
\phi v\sim \alpha(w)\phi w + \sum_{y\in J}\alpha(y)\phi y\end{equation}
\item Assume $v,w$ are both $b$-and-$b$. Then there is a weight $\alpha$ and a finite set $J$ such that no element of $J$ is $b$-and-$b$ and \begin{equation}
\phi v\sim \alpha(w)\phi w + \sum_{y\in J}\alpha(y)\phi y\end{equation}
\end{enumerate}
\end{cor}

\begin{proof}
Let $\tau_b(v)=\tau_b(w)=x$. We use \ref{rewrite}:
\begin{enumerate}[(i)]
\item Let $v,w$ be both $b$-left. Then there are $m_0\neq 0$ and $n_0\neq 0$ such that $v=b^{m_0}x$ and $w=b^{n_0}x$. The statement follows from \ref{rewrite2}.
\item Let $v,w$ be both right-$b$. Then there are $m_k\neq 0$ and $n_k\neq 0$ such that $v=xb^{m_k}$ and $w=xb^{n_k}$. The statement follows from \ref{rewrite1}.
\item Let $v,w$ be both $b$-and-$b$. Then there are $m_0\neq 0$,$m_k\neq 0$ and $n_0\neq 0$, $n_k\neq 0$ such that $v=b^{m_0}xb^{m_k}$ and $w=b^{n_0}xb^{n_k}$. By \ref{rewrite2} we have \begin{equation}\phi v\sim \alpha_1(b^{n_0}xb^{m_k})\phi b^{n_0}xb^{m_k}+\sum_{y_1\in J_1}\alpha_1(y_1)\phi y_1,\end{equation} where $\alpha_1$ is a weight on $T_n$ and $J_1$ is a set such all elements of $J_1$ are right-$b$. By \ref{rewrite1} we have \begin{equation}\phi b^{n_0}xb^{m_k}\sim \alpha_2(w)\phi w+\sum_{y_2\in J_2}\alpha_2(y_2)\phi y_2,\end{equation} where $\alpha_2$ is a weight on $T_n$ and $J_2$ is a set such that all elements of $J_1$ are $b$-left. Together this gives us \begin{equation}\phi v\sim \alpha_1(b^{n_0}xb^{m_k})\alpha_2(w)\phi w+\sum_{y_2\in J_2}\alpha_1(b^{n_0}xb^{m_k})\alpha_2(y_2)\phi y_2+\sum_{y_1\in J_1}\alpha_1(y_1)\phi y_1.\end{equation}
\end{enumerate}
\end{proof}

\begin{exmp}
Let $v=ab^5$ and $w=ab^{-1}$ be elements of $F_2$. Then \begin{equation}v\sim -w+a-aa-\sum_{i=1}^{4}ab^ia-\sum_{i=1}^{4}ab^ia^{-1}\end{equation}
\end{exmp}

\begin{exmp}
Let $v=b^3ab^{-2}$ and $w=b^2ab^1$ be elements of $F_2$. Then \begin{equation}v \sim -w+b^2a-b^2a^2-ab^2ab^{-2}-a^{-1}b^2ab^{-2}-b^2ab^{-1}a-b^2ab^{-1}a^{-1}\end{equation}
\end{exmp}

\begin{thm}\label{normalform}
Let $f=\sum_{v\in I} \alpha(v)\phi v$ be a sum of counting quasimorphisms. Then $f$ is equivalent to a sum of counting quasimorphisms in normal form.
\end{thm}

\begin{proof}
By \ref{normalform1} we know that $f\in \mathcal{B}(F_n,S)$ is equivalent to an element $f_0=\sum_{v\in J} \beta(v)\phi v \in \mathcal{B}(F_n,S)$ such that $v\in \mathrm{supp}(\beta)$ and $\tau_b(v)=\varnothing$ imply $v=b$. 

By considering $f_0$ instead of $f$, we can assume wlog that $f$ satisfies the above condition. If $v, w\in I$ are $b$-and-$b$ and $\tau_b(v^{-1})=\tau_b(w)$, we use $\phi w=-\phi w^{-1}$ to rewrite $f$ such that $w\notin I$. Assume there are $k$ unordered pairs $\{v,w\}\subset I$ such that $v,w$ are both $b$-and-$b$ and $\tau_b(v)=\tau_b(w)$. Then pick one such pair $\{v_1,w_1\}$. By \ref{cor:rewrite} there is a weight $\beta_1$ and a set $J_1$ such that no element of $J_1$ is $b$-and-$b$ and \begin{equation} f\sim g_1=\sum_{v\in I\setminus \{v_1\}} \alpha(v)\phi v+\alpha(v_1)\beta_1(w_1)\phi w_1 + \sum_{y\in J_1}\alpha(v_1)\beta(y)\phi y.\end{equation} Since no element of $J_1$ is $b$-and-$b$ there are at most $k-1$ unordered pairs $\{v,w\}\subset I\setminus \{v_1\}\cup J_1$ such that $v,w$ are both $b$-and-$b$ and $\tau_b(v)=\tau_b(w)$. Repeating this argument at most $k-1$ times, we get that $f$ is equivalent to a sum of counting quasimorphisms $f_1=\sum_{v\in I_1} \gamma_1(v)\phi v$ such that \begin{enumerate} 
\item $v\in I_1$ with $\tau_b(v)=\varnothing$ implies $v=b$,
\item $v,w\in I_1$ are both $b$-and-$b$, $\tau_b(v)=\tau_b(w)$ implies $v=w$.
\item Let $v\in I$ and let $\tau_b(v^{-1})=\tau_b(w)$. If $v,w$ are both $b$-and-$b$, then $w\notin I$.\end{enumerate}

By considering $f_1$ instead of $f$, we can assume wlog that $f$ satisfies these conditions. If $v\in I$ is $b$-left, $w\in I$ is right-$b$ and $\tau_b(v^{-1})=\tau_b(w)$, we use $\phi w=-\phi w^{-1}$ to rewrite $f$ such that $w\notin I$. Assume there are $k$ unordered pairs $\{v,w\}\subset I$ such that $v,w$ are both $b$-left and $\tau_b(v)=\tau_b(w)$. Then pick one such pair $\{v_2,w_2\}$. By \ref{cor:rewrite} there is a weight $\beta_2$ and a set $J_2$ such that all elements of $J_2$ are $b$-truncated and \begin{equation} f\sim g_2=\sum_{v\in I\setminus \{v_2\}} \alpha(v)\phi v+\alpha(v_2)\beta(w_2)\phi w_2 + \sum_{y\in J_2}\alpha(v_2)\beta_2(y)\phi y.\end{equation} Since all elements of $J_2$ are $b$-truncated there are at most $k-1$ unordered pairs $\{v,w\}\subset I\setminus \{v_2\}\cup J_2$ such that $v,w$ are both $b$-left and $\tau_b(v)=\tau_b(w)$. Repeating this argument $k-1$ times, we get that $f$ is equivalent to a sum of counting quasimorphisms $f_2=\sum_{v\in I_2} \gamma_2(v)\phi v$ such that\begin{enumerate} 
\item $v\in I_2$ with $\tau_b(v)=\varnothing$ implies $v=b$.
\item $v,w\in I_2$ are both $b$-and-$b$, $\tau_b(v)=\tau_b(w)$ implies $v=w$.
\item $v,w\in I_2$ are both $b$-left, $\tau_b(v)=\tau_b(w)$ implies $v=w$.
\item Let $v\in I$ and let $\tau_b(v^{-1})=\tau_b(w)$. If $v,w$ are both $b$-and-$b$, or one of them is $b$-left and the other one is right-$b$, then $w\notin I$.
\end{enumerate}

By considering $f_2$ instead of $f$, we can assume wlog that $f$ satisfies these conditions. Assume there are $k$ unordered pairs $\{v,w\}\subset I$ such that $v,w$ are both right-$b$ and $\tau_b(v)=\tau_b(w)$. Then pick one such pair $\{v_3,w_3\}$. By \ref{cor:rewrite} there is a weight $\beta_3$ and a set $J_3$ such that all elements of $J_3$ are $b$-truncated and \begin{equation} f\sim g_3=\sum_{v\in I\setminus \{v_3\}} \alpha(v)\phi v+\alpha(v_3)\beta(w_3)\phi w_3 + \sum_{y\in J_3}\alpha(v_3)\beta_3(y)\phi y.\end{equation} Since all elements of $J_3$ are $b$-truncated there are at most $k-1$ unordered pairs $\{v,w\}\subset I\setminus \{v_3\}\cup J_3$ such that $v,w$ are both $b$-left and $\tau_b(v)=\tau_b(w)$. Repeating this argument at most $k-1$ times, we get that $f$ is equivalent to a sum of counting quasimorphisms $f_3=\sum_{v\in I_3} \gamma_3(v)\phi v$ such that \begin{enumerate} 
\item $v\in I_3$ with $\tau_b(v)=\varnothing$ implies $v=b$.
\item $v,w\in I_3$ are both $b$-and-$b$, $\tau_b(v)=\tau_b(w)$ implies $v=w$.
\item $v,w\in I_3$ are both $b$-left, $\tau_b(v)=\tau_b(w)$ implies $v=w$.
\item $v,w\in I_3$ are both right-$b$, $\tau_b(v)=\tau_b(w)$ implies $v=w$.
\item Let $v\in I$ and let $\tau_b(v^{-1})=\tau_b(w)$. If $v,w$ are both $b$-and-$b$, or one of them is $b$-left and the other one is right-$b$, then $w\notin I$.
\end{enumerate}

If $v,w$ are both $b$-truncated and $\tau_b(v)=\tau_b(w)$, then $v=\tau_b(v)=\tau_b(w)=w$. So $f_3$ also satisfies this demand. If $v,w\in I$ are both $b$-truncated and $\tau_b(v^{-1})=\tau_b(w)$, then $v^{-1}=w$. So using $\phi w=-\phi w^{-1}$, we can ensure that $w\notin I$ in this case. We get that we can write $f_3$ in normal form. This finishes the proof.
\end{proof}

Note that we actually gave an algorithm to find an element in normal form in every equivalence class $[f] \in \tilde{\mathcal{B}}(F_n)$.

\begin{defn}
Let $f\in \mathcal{B}(F_n,S)$. By \ref{normalform} there is a sum of counting functions $g\in \mathcal{B}(F_n,S)$ such that $g$ is in normal form and $f\sim g$. Then we call $g_n$ a \emph{n-representative} of $[f]$.
\end{defn}

Note that a sum of counting quasimorphisms $f\in \mathcal{B}(F_n,S)$ can have more than one n-representative. In this regard there is difference between a single counting quasimorphism $\phi v$ and sums of counting quasimorphisms $f$.

\subsection{Study of the n-support}\label{studynsuppI}

By \ref{normalform} we can assume from now on that $f=\sum_{v\in I} \alpha(v)\phi v$ is in normal form. Then $f_n$ is a n-representative of $[f]$. In this subsection we will show that $f_n$ is indeed a good representative of $T^{-n}[f]$. We will do so by studying the n-support of $I$. This course of action is only made possible by choosing $I$ in normal form.

\begin{rem}
We actually only consider sets $I$ in normal form such that $b\notin I$. This condition could be lifted, but that would make everything even more unreadable than it already is.
\end{rem}

We start by trying to compute the length of $f_n$. This is the maximal word length of any element in $\mathrm{supp}(\alpha_n)$, where $\alpha_n(u)=\sum_{v\in I}\alpha(v)v_n(u)$. So we can bound the length of $f_n$ from below if we find elements in $\mathrm{supp}(\alpha_n)$. At the moment we do not know more than that $\mathrm{supp}(\alpha_n)\subset W_n(I)$. Our strategy is to find an element $u_n\in W_n(I)$ such that $u_n \notin W_n(I\setminus \{v_0\})\cup W_n(I^{-1}\setminus \{v_0^{-1}\})$ for some $v_0\in \mathrm{supp}(\alpha)$. Then we have by definition that $\alpha_n(u_n)=(v_0)_n(u_n)\neq 0$ and therefore $u_n \in \mathrm{supp}(\alpha_n)$.

We reformulate this strategy in the terms of our visualization: We see that $W_n(I)$ is a union of rectangles. If we find an element of one of these rectangles that does not lie in any other rectangle, we know that this element is in $\mathrm{supp}(\alpha_n)$. This motivates studying the intersection of rectangles. We start with this technical lemma:

\begin{lem}\label{lmrinWnv}
Let $v\in F_n\setminus \{e\}$ not be a power of $b$. Let $x=b^{m_0}s_1b^{m_1}...b^{m_{k-1}}s_kb^{m_k}$ be reduced and assume $x=lmr$ for $l\in L_{n}(v)$, $\{m\}=M_{n}(v)$ and $r\in R_{n}(v)$.\begin{enumerate}[(i)]
\item If $m_0\neq 0$, then $l=b^{m_0}$. If $m_0=0$, then $l=e$ or $l=s_1b^{m_1}$.
\item If $m_k\neq 0$, then $r=b^{m_k}$. If $m_k=0$, then $r=e$ or $r=b^{m_{k-1}}s_k$.
\end{enumerate}
\end{lem}

\begin{proof}
We know that $m$ is $b$-truncated. So the statement follows from \ref{lmrred} and \ref{lrlook}. 
\end{proof}

The following statement is the main reason behind our definition of \emph{normal form}.

\begin{prop}\label{kindintempty}
Let $I$ be in normal form and assume $b\notin I$. Let $v,w\in I$ be of one kind and $v\neq w$. Then $W_n(v)\cap W_n(w)=\varnothing$.
\end{prop}

\begin{proof}
Assume on the contrary that $x\in W_n(v)\cap W_n(w)$. We write $x=b^{m_0}s_1...s_kb^{m_k}$. We get that $x=l_vm_vr_v=l_wm_wr_w$ for some $l_v\in L_{n}(v)$, $m_v\in L_{n}(v)$, $r_v \in R_{n}(v)$ and $l_w\in L_{n}(w)$, $m_w\in L_{n}(w)$, $r_w \in R_{n}(w)$. Since $v$ and $w$ are of one kind, we get that $l_v=e$ if and only if $l_w=e$ and $r_v=e$ if and only if $r_w=e$. By \ref{lmrinWnv} we get therefore that $l_v=l_w$ and $r_v=r_w$. So we have $m_v=m_w$. By \ref{phiwisTinv} this gives us $\tau_b(v)=\tau_b(w)$. Since $I$ is in normal form, we get $v=w$. This is a contradiction.
\end{proof}

\begin{cor}\label{oppkindintempty}
Let $I$ be in normal form and assume $b\notin I$. Let $v,w\in I$ be of opposite kind. Then $W_n(v)\cap W_n(w^{-1})=\varnothing$.
\end{cor}

\begin{proof}
If $v,w$ are both $b$-truncated or $b$-and-$b$, then the statement follows directly from Proposition \ref{kindintempty}. Assume wlog that $v$ is $b$-left. Then $w$ is right-$b$, so $w^{-1}$ is $b$-left and the statement follows again from Proposition \ref{kindintempty}.
\end{proof}

The visualization of \ref{kindintempty} is straightforward: If $I$ is a set in normal form, then $W_n(I)$ is a union of rectangles such that two different rectangles of the same shape never intersect. So we do not have to worry about, for example, two boxes intersecting. Now we have to handle the intersection of rectangles of different shapes.

\begin{lem}\label{intrect}
Let $w$ be a reduced word and assume that $w$ is either $b$-truncated or $b$-and-$b$ or $\lvert w \rvert_b\geq 2$. Let $u_1\in W_n(w)$ be of type $(i_1,j_1)$ and let $u_2\in W_n(w)$ be of type $(i_2,j_2)$. Assume $u_1\in W_n(v)$ for a reduced word $v$ that is not a power of $b$. Then we get the following:
\begin{enumerate}[(i)]
\item Assume $v$ is $b$-truncated. If $i_1\neq i_2$ or $j_1\neq j_2$, then $u_2\notin W_n(v)$.
\item Assume $v$ is $b$-left. If $j_1\neq j_2$, then $u_2\notin W_n(v)$.
\item Assume $v$ is right-$b$. If $i_1\neq i_2$, then $u_2\notin W_n(v)$.
\end{enumerate}
\end{lem}

\begin{proof}
Since $u_1\in W_n(w)$ is of type $(i_1,j_1)$, we get that $u_1=l_{i_1}m_wr_{j_1}$ with $l_{i_1}\in dL_{i_1}(w)$, $\{m_w\}=M_n(w)$ and $r_{j_1}\in dR_{j_1}(w)$. Since $u_1\in W_n(v)$, we get that $u_1=lm_vr$ with $l\in L_{n}(v)$, $\{m_v\}=M_n(v)$ and $r\in R_{n}(v)$.
\begin{enumerate}
\item Let $v$ be $b$-truncated. Then $l=e$ and $r=e$, so $l_{i_1}m_wr_{j_1}=m_v$.\begin{enumerate}
\item If $i_1\neq i_2$, then $l_{i_1}\neq l_{i_2}$ for all $l_{i_2}\in dL_{i_2}(w)$. By \ref{lmrred} it follows that $u_2\neq m_v$ and therefore $u_2\notin W_n(v)$.
\item If $j_1\neq j_2$, then $r_{j_1}\neq r_{j_2}$ for all $r_{j_2}\in dR_{j_2}(w)$. By \ref{lmrred} it follows that $u_2\neq m_v$ and therefore $u_2\notin W_n(v)$.
\end{enumerate}
\item Let $v$ be $b$-left. Then $r=e$, so $l_{i_1}m_wr_{j_1}=lm_v$. If $r_{j_1}=e$ (amongst others if $w$ is $b$-truncated), then by \ref{LRexpl} $j_1=j_2=0$. So we can assume that $r_{j_1}\neq e$. Since $m_v$ is necessarily $b$-truncated, we also know that $r_{j_1}$ is not a power of $b$. So by \ref{lmrinWnv} we get that $\lvert r_{j_1}\rvert_b=1$. If $w$ is $b$-and-$b$, then either $l_{i_1}$ is a power of $b$ or $\lvert l_{i_1} \rvert_b=1$. In the first case also $l$ is a power of $b$ by \ref{lmrinWnv} and therefore $\lvert m_v\rvert_b=\lvert lm_v\rvert_b=\lvert l_{i_1}m_wr_{j_1}\rvert_b\geq 1+1=2$. In the second case $\lvert lm_v\rvert_b=\lvert l_{i_1}m_wr_{j_1}\rvert_b\geq 1+1+1=3$. Since $\lvert l\rvert_b\leq 1$, we get $\lvert m_v\rvert_b\geq 2$. If $\lvert w \rvert_b\geq 2$, then by \ref{lmrred} $\lvert lm_v\rvert_b=\lvert l_{i_1}m_wr_{j_1}\rvert_b\geq 2+1=3$. Again we get that $\lvert m_v\rvert_b\geq 2$.So we can conclude that $\lvert m_v\rvert_b\geq 2$. Since $\lvert r_{j_1}\rvert_b\leq 1$, there is $x_v$ such that $x_vr_{j_1}=m_v$ is reduced. If $j_1\neq j_2$, then $r_{j_1}\neq r_{j_2}$ for all $r_{j_2}\in dR_{j_2}(w)$. By \ref{lmrred} it follows that $u_2\neq l_vm_v$ for every $l_v\in L_n(v)$. So $u_2\notin W_n(v)$.
\item Let $v$ be right-$b$. Then $l=e$, so $l_{i_1}m_wr_{j_1}=m_vr$. If $l_{i_1}=e$, then by \ref{LRexpl} $i_1=i_2=0$. So we can assume that $l_{i_1}\neq e$. Since $m_v$ is necessarily $b$-truncated, we also know that $l_{i_1}$ is not a power of $b$. So by \ref{lmrinWnv} we get that $\lvert l_{i_1}\rvert_b=1$. If $w$ is $b$-and-$b$, then either $r_{j_1}$ is a power of $b$ or $\lvert r_{j_1} \rvert_b=1$. In the first case also $r$ is a power of $b$ by \ref{lmrinWnv} and therefore $\lvert m_v\rvert_b=\lvert m_vr\rvert_b=\lvert l_{i_1}m_wr_{j_1}\rvert_b\geq 1+1=2$. In the second case $\lvert m_vr\rvert_b=\lvert l_{i_1}m_wr_{j_1}\rvert_b\geq 1+1+1=3$. Since $\lvert r\rvert_b\leq 1$, we get $\lvert m_v\rvert_b\geq 2$. If $\lvert w \rvert_b\geq 2$, then by \ref{lmrred} $\lvert m_vr\rvert_b=\lvert l_{i_1}m_wr_{j_1}\rvert_b\geq 2+1=3$. Again we get that $\lvert m_v\rvert_b\geq 2$.
So we can conclude that $\lvert m_v\rvert_b\geq 2$. Since $\lvert l_{i_1}\rvert_b\leq 1$, there is $x_v$ such that $l_{i_1}x_v=m_v$ is reduced. If $i_1\neq i_2$, then $l_{i_1}\neq l_{i_2}$ for all $l_{i_2}\in dL_{i_2}(w)$. By \ref{lmrred} it follows that $u_2\neq m_vr_v$ for every $r_v\in R_n(v)$. So $u_2\notin W_n(v)$.
\end{enumerate}
\end{proof}

In terms of our visualization the Lemma \ref{intrect} gives us the shape of an intersection of rectangles. We get the following:\begin{enumerate}
\item The intersection of a cell with a rectangle is empty or a cell.
\item The intersection of a column with a rectangle is empty or a column or a cell
\item The intersection of a row with a rectangle is empty or a row or a cell.
\end{enumerate} What we proved so far would actually already be enough to prove the following: Let $I$ be a finite set in normal form and $b\notin I$. Then there is $v_0\in I$ such that for $n$ big enough a square in the rectangle $W_n(v_0)$ is not intersected by any other rectangle. So we could already find an element $u_n \in \mathrm{supp}(\alpha_n)$. But this is not all what we wanted: We also need information on the length of $u_n$. This will be provided by our study of the length of squares in Section \ref{Tsingle}:

\begin{cor}\label{lengthrect}
Let $w\equiv_S (m_0,s_1,...,s_k,m_k)$ be a reduced word and assume that $w$ is not a power of $b$. Let $N=\max (\lvert m_0\rvert, \lvert m_k\rvert)$. Then for $n>2N$ we get the following:\begin{enumerate}[(i)]
\item Let $w$ be $b$-truncated. Let $u\in W_n(w)$. Then \[\lvert u\rvert_S=\lvert(\phi w)_n\rvert_S.\]
\item Let $w$ be $b$-left. Let $u\in W_n(w)$ be of type $(i,0)$ with $\max (1,\lvert m_0 \rvert)<i\leq n$. Then \[\lvert u\rvert_S=\lvert(\phi w)_n\rvert_S-(n-i).\]
\item Let $w$ be right-$b$. Let $u\in W_n(w)$ be of type $(0,j)$ with $\max (1,\lvert m_k \rvert)<j\leq n$. Then \[\lvert u\rvert_S=\lvert(\phi w)_n\rvert_S-(n-j).\]
\item Let $w$ be $b$-and-$b$. Let $u\in W_n(w)$ be of type $(i,j)$ with $\max (1,\lvert m_0 \rvert)<i\leq n$ and $\max (1,\lvert m_k \rvert)<j\leq n$. Then \[\lvert u\rvert_S=\lvert(\phi w)_n\rvert_S-(n-i)-(n-j).\]
\end{enumerate}
\end{cor}

\begin{proof}
The proof consists in a few computations involving \ref{nnbiggest} and \ref{LRgrow}.
\begin{enumerate}[(i)]
\item Let $w$ be $b$-truncated. Then by \ref{4shapes} we know that $u$ is of type $(0,0)$. By \ref{nnbiggest} we know $\lvert(\phi w)_n\rvert_S=\lvert (0,0)\rvert_{W_n(w)}$. So \begin{equation}
\lvert u\rvert_S=\lvert(\phi w)_n\rvert_S.
\end{equation}
\item Let $w$ be $b$-left. Then by \ref{nnbiggest} we know $\lvert(\phi w)_n\rvert_S=\lvert (n,0)\rvert_{W_n(w)}$. By \ref{LRgrow} we know $\lvert (n,0)\rvert_{W_n(w)}-\lvert (i,0)\rvert_{W_n(w)}=n-i$. So \begin{equation}
\lvert u\rvert_S=\lvert (i,0)\rvert_{W_n(w)}=\lvert(\phi w)_n\rvert_S-(n-i).
\end{equation}
\item Let $w$ be right-$b$. Then by \ref{nnbiggest} we know $\lvert(\phi w)_n\rvert_S=\lvert (0,n)\rvert_{W_n(w)}$. By \ref{LRgrow} we know $\lvert (0,n)\rvert_{W_n(w)}-\lvert (0,j)\rvert_{W_n(w)}=n-j$. So \begin{equation}
\lvert u\rvert_S=\lvert (0,j)\rvert_{W_n(w)}=\lvert(\phi w)_n\rvert_S-(n-j).
\end{equation}
\item Let $w$ be $b$-and-$b$. Then by \ref{nnbiggest} we know $\lvert(\phi w)_n\rvert_S=\lvert (n,n)\rvert_{W_n(w)}$. By \ref{LRgrow} we know $\lvert (n,n)\rvert_{W_n(w)}-\lvert (i,j)\rvert_{W_n(w)}=n-i+n-j$. So \begin{equation}
\lvert u\rvert_S=\lvert (i,j)\rvert_{W_n(w)}=\lvert(\phi w)_n\rvert_S-(n-i)-(n-j).
\end{equation}
\end{enumerate}
\end{proof}

Now we are ready to prove the main result of this subsection. The intuition behind Proposition \ref{celllinebox} is the following:\begin{enumerate}
\item Using $k$ rows and cells to cover a long enough column, a square of distance at most $k$ to the lowest square in the column remains uncovered.
\item Using $k$ columns and cells to cover a long enough row, a square of distance at most $k$ to the rightmost square remains uncovered.
\item Using $k$ rows, columns and cells to cover a big enough box, a square of distance at most $k$ to the square in the lower, right corner of the box remains uncovered.
\end{enumerate}

\begin{prop}\label{celllinebox}
Let $I$ be a finite set in normal form and let $b\notin I$. If $I$ contains no $b$-and-$b$ element, but does contain a $b$-left (right-$b$) element, we assume that $I$ contains a $b$-left (right-$b$) element of $b$-length bigger than 1. Then there is $N\in \mathbb{N}$ and $v_0\in I$ such that for all $n>N$ we can find $u_n\in W_n(v_0)$ with \begin{enumerate}
\item\label{celllinebox1} $u_n \notin W_n(I\setminus \{v_0\})\cup W_n(I)^{-1}$
\item\label{celllinebox2} $\lvert u_n\rvert_S\geq \lvert(\phi v_0)_n\rvert_S-2\lvert I \rvert$
\end{enumerate}
\end{prop}

\begin{proof}
Every element $v_i\in I$ can be written as $v_i\equiv_S (m_{i,0},s_{i,0},...,s_{i,l},m_{i,l})$. So we can choose $N=2\lvert I\rvert+2\max \bigcup_{v_i\in I} \{\lvert m_{i,0}\rvert,\lvert m_{i,l}\rvert\}$.

\begin{case} All elements of $I$ are $b$-truncated.

Then choose a $b$-truncated element $v_0\in I$. For $n>N$, take $u_n \in W_n(v_0)$. Since $I$ is in normal form and all elements of $I$ are $b$-truncated, we get by \ref{kindintempty} that $u_n \notin W_n(I\setminus \{v_0\})$ and by \ref{oppkindintempty} that $u_n \notin W_n(I^{-1})=W_n(I)^{-1}$. This shows \ref{celllinebox1}. By \ref{lengthrect} we have $\lvert u_n\rvert_S=\lvert(\phi v_0)_n\rvert_S$. This shows \ref{celllinebox2}.
\end{case}

\begin{case} There is a $b$-left element, but no $b$-and-$b$ element in $I$.

Then choose a $b$-left element $v_0\in I$. We give an algorithm to find $u_n\in W_n(v_0)$ with the desired property \ref{celllinebox1}. (Note that our algorithm does not necessarily find the maximal $u_n\in W_n(v_0)$ that satisfies \ref{celllinebox1}.)

Step 1: Pick $x_1 \in W_n(v_0)$ such that $x_1$ is of type $(n,1)$. If $x_1 \notin W_n(I\setminus \{v_0\})\cup W_n(I)^{-1}$, then set $u_n=x_1$ and stop. If $x_1 \in W_n(I\setminus \{v_0\})\cup W_n(I)^{-1}$, we have wlog $x_1\in W_n(v_1)\cup W_n(v_1)^{-1}$. We know that $v_1$ is either $b$-truncated, $b$-left or right-$b$. Assume $x_1\in W_n(v_1)$ for $v_0\neq v_1$. Then by \ref{kindintempty} we know that $v_0$ and $v_1$ are not of one kind, so $v_1$ is either $b$-truncated or right-$b$. In both cases we get by \ref{intrect} that $y\notin W_n(v_1)$ for all $y\in W_n(v_0)$, which are of type $(i,1)$ for $i<n$. In this case set $D_1=\{v_1\}\subset I$ and $E_1=\varnothing \subset I^{-1}$ and pick $x_2 \in W_n(v_0)$ such that $x_2$ is of type $(n-1,1)$. Proceed with Step 2. Otherwise we have $x_1\in W_n(v_1^{-1})$. Then by \ref{oppkindintempty} we know that $v_0$ and $v_1$ are not of opposite kind, so $v_1$ is either $b$-truncated or $b$-left. In both cases we get by \ref{intrect} that $y\notin W_n(v_1^{-1})$ for all $y\in W_n(v_0)$, which are of type $(i,1)$ for $i<n$. In this case set $D_1=\varnothing\subset I$ and $E_1=\{v_1^{-1}\} \subset I^{-1}$ and pick $x_2 \in W_n(v_0)$ such that $x_2$ is of type $(n-1,1)$. Proceed with Step 2.

Step $k$ for $1<k\leq 2\lvert I\rvert-1$: In the previous step we got $x_k \in W_n(v_0)$ such that $x_k$ is of type $(n-k+1,1)$. We know by the previous steps that there are $D_{k-1}\subset I$ and $E_{k-1}\subset I^{-1}$ such that $\lvert D_{k-1}\rvert +\lvert E_{k-1}\rvert=k-1$ and $x_k\notin W_n(I\setminus D_{k-1})\cup W_n(I^{-1}\setminus E_{k-1})$.If $x_k \notin W_n(I\setminus \{v_0\})\cup W_n(I)^{-1}$, then set $u_n=x_k$ and stop. Otherwise we repeat the argument from above: We get that there are two sets $D_{k-1}\subset D_{k}\subset I$ and $E_{k-1}\subset E_{k}\subset I^{-1}$ such that $\lvert D_{k}\rvert +\lvert E_{k}\rvert=k$ and $y\notin W_n(I\setminus D_{k})\cup W_n(I^{-1}\setminus E_{k})$ for all $y\in W_n(v_0)$, which are of type $(i,1)$ for $i<n-k+1$. Pick $x_{k+1} \in W_n(v_0)$ such that $x_{k+1}$ is of type $(n-k,1)$.

Step $2\lvert I\rvert$: In the previous step we got $x_{2\lvert I\rvert} \in W_n(v_0)$ such that $x_{2\lvert I\rvert}$ is of type $(n-2\lvert I\rvert+1,1)$. We also got that there are $D_{2\lvert I\rvert-1}\subset I$ and $E_{2\lvert I\rvert-1}\subset I^{-1}$ such that $\lvert D_{2\lvert I\rvert-1}\rvert +\lvert E_{2\lvert I\rvert-1}\rvert=2\lvert I\rvert-1$ and $x_{2\lvert I\rvert}\notin W_n(I\setminus D_{2\lvert I\rvert-1})\cup W_n(I^{-1}\setminus E_{2\lvert I\rvert-1})$. Since $x_{2\lvert I\rvert} \in W_n(v_0)$, we have $W_n(I\setminus D_{2\lvert I\rvert-1})\cup W_n(I^{-1}\setminus E_{2\lvert I\rvert-1})=W_n(I\setminus \{v_0\})\cup W_n(I)^{-1}$. So set $u_n=x_{2\lvert I\rvert}$ and stop.

This algorithm terminates after at most $2\lvert I\rvert$ steps. So we found $u_n\in W_n(v_0)$ that satisfies \ref{celllinebox1}. By the nature of our algorithm, we know that $u_n$ is of type $(i,1)$ for $i\geq n-2\lvert I\rvert+1$. Since $n>N$, we can use \ref{lengthrect} and get: \begin{equation}\begin{split} \lvert u_n\rvert_S &=\lvert(\phi v_0)_n\rvert_S-(n-i)\\
&=\lvert(\phi v_0)_n\rvert_S-n+i\\
&\geq \lvert(\phi v_0)_n\rvert_S-n+(n-2\lvert I\rvert+1)\\
&=\lvert(\phi v_0)_n\rvert_S-2\lvert I\rvert+1 \end{split}\end{equation} So $u_n$ satisfies \ref{celllinebox2}.
\end{case}

\begin{case} There is a right-$b$ element, but no $b$-left element or $b$-and-$b$ element in $I$.

Then choose a right-$b$ element $v_0\in I$ with $\lvert v_0\rvert_b\geq 2$. We use essentially the same algorithm as above to find $u_n\in W_n(v_0)$ with the desired property \ref{celllinebox1}.

Step 1: Pick $x_1 \in W_n(v_0)$ such that $x_1$ is of type $(1,n)$. If $x_1 \notin W_n(I\setminus \{v_0\})\cup W_n(I)^{-1}$, then set $u_n=x_1$ and stop. If $x_1 \in W_n(I\setminus \{v_0\})\cup W_n(I)^{-1}$, we have wlog $x_1\in W_n(v_1)\cup W_n(v_1)^{-1}$. We know that $v_1$ is either $b$-truncated or right-$b$. Assume $x_1\in W_n(v_1)$ for $v_0\neq v_1$. Then by \ref{kindintempty} we know that $v_0$ and $v_1$ are not of one kind, so $v_1$ is $b$-truncated. By \ref{intrect} we get that $y\notin W_n(v_1)$ for all $y\in W_n(v_0)$, which are of type $(1,j)$ for $j<n$. In this case set $D_1=\{v_1\}\subset I$ and $E_1=\varnothing \subset I^{-1}$ and pick $x_2 \in W_n(v_0)$ such that $x_2$ is of type $(1,n-1)$. Proceed with Step 2. Otherwise we have $x_1\in W_n(v_1^{-1})$. Then $v_1$ is either $b$-truncated or right-$b$. In both cases we get by \ref{intrect} that $y\notin W_n(v_1^{-1})$ for all $y\in W_n(v_0)$, which are of type $(1,j)$ for $j<n$. In this case set $D_1=\varnothing\subset I$ and $E_1=\{v_1^{-1}\} \subset I^{-1}$ and pick $x_2 \in W_n(v_0)$ such that $x_2$ is of type $(1,n-1)$. Proceed with Step 2.

Step $k$ for $1<k\leq 2\lvert I\rvert-1$: In the previous step we got $x_k \in W_n(v_0)$ such that $x_k$ is of type $(1,n-k+1)$. We know by the previous steps that there are $D_{k-1}\subset I$ and $E_{k-1}\subset I^{-1}$ such that $\lvert D_{k-1}\rvert +\lvert E_{k-1}\rvert=k-1$ and $x_k\notin W_n(I\setminus D_{k-1})\cup W_n(I^{-1}\setminus E_{k-1})$. If $x_k \notin W_n(I\setminus \{v_0\})\cup W_n(I)^{-1}$, then set $u_n=x_k$ and stop. Otherwise we repeat the argument from above: We get that there are two sets $D_{k-1}\subset D_{k}\subset I$ and $E_{k-1}\subset E_{k}\subset I^{-1}$ such that $\lvert D_{k}\rvert +\lvert E_{k}\rvert=k$ and $y\notin W_n(I\setminus D_{k})\cup W_n(I^{-1}\setminus E_{k})$ for all $y\in W_n(v_0)$, which are of type $(1,j)$ for $j<n-k+1$. Pick $x_{k+1} \in W_n(v_0)$ such that $x_{k+1}$ is of type $(1,n-k)$.

Step $2\lvert I\rvert$: In the previous step we got $x_{2\lvert I\rvert} \in W_n(v_0)$ such that $x_{2\lvert I\rvert}$ is of type $(1,n-2\lvert I\rvert+1)$. We also got that there are $D_{2\lvert I\rvert-1}\subset I$ and $E_{2\lvert I\rvert-1}\subset I^{-1}$ such that $\lvert D_{2\lvert I\rvert-1}\rvert +\lvert E_{2\lvert I\rvert-1}\rvert=2\lvert I\rvert-1$ and $x_{2\lvert I\rvert}\notin W_n(I\setminus D_{2\lvert I\rvert-1})\cup W_n(I^{-1}\setminus E_{2\lvert I\rvert-1})$. Since $x_{2\lvert I\rvert} \in W_n(v_0)$, we have $W_n(I\setminus D_{2\lvert I\rvert-1})\cup W_n(I^{-1}\setminus E_{2\lvert I\rvert-1})=W_n(I\setminus \{v_0\})\cup W_n(I)^{-1}$. So set $u_n=x_{2\lvert I\rvert}$ and stop.

This algorithm terminates after at most $2\lvert I\rvert$ steps. So we found $u_n\in W_n(v_0)$ that satisfies \ref{celllinebox1}. By the nature of our algorithm, we know that $u_n$ is of type $(1,j)$ for $j\geq n-2\lvert I\rvert+1$. Since $n>N$, we can use \ref{lengthrect} and get: \begin{equation}\begin{split} \lvert u_n\rvert_S &=\lvert(\phi v_0)_n\rvert_S-(n-j)\\
&=\lvert(\phi v_0)_n\rvert_S-n+j\\
&\geq \lvert(\phi v_0)_n\rvert_S-n+(n-2\lvert I\rvert+1)\\
&=\lvert(\phi v_0)_n\rvert_S-2\lvert I\rvert+1 \end{split}\end{equation} So $u_n$ satisfies \ref{celllinebox2}.
\end{case}

\begin{case} There is a $b$-and-$b$ element in $I$.

Then choose a $b$-and-$b$ element $v_0\in I$. Again we use essentially the same algorithm as above to find $u_n\in W_n(v_0)$ with the desired property \ref{celllinebox1}.

Step 1: Pick $x_1 \in W_n(v_0)$ such that $x_1$ is of type $(n,n)$. If $x_1 \notin W_n(I\setminus \{v_0\})\cup W_n(I)^{-1}$, then set $u_n=x_1$ and stop. If $x_1 \in W_n(I\setminus \{v_0\})\cup W_n(I)^{-1}$, we have wlog $x_1\in W_n(v_1)\cup W_n(v_1)^{-1}$. By \ref{kindintempty} and \ref{oppkindintempty} we know that $v_1$ is either $b$-truncated, $b$-left or right-$b$. Assume $x_1\in W_n(v_1)$ for $v_0\neq v_1$. If $v_1$ is $b$-truncated or right-$b$, we get by \ref{intrect} that $y\notin W_n(v_1)$ for all $y\in W_n(v_0)$, which are of type $(i,n)$ for $i<n$. Then pick $x_2 \in W_n(v_0)$ such that $x_2$ is of type $(n-1,n)$. If $v_1$ is $b$-left, we get by \ref{intrect} that $y\notin W_n(v_1)$ for all $y\in W_n(v_0)$, which are of type $(n,j)$ for $j<n$. Then pick $x_2 \in W_n(v_0)$ such that $x_2$ is of type $(n,n-1)$. In both cases set $D_1=\{v_1\}\subset I$ and $E_1=\varnothing \subset I^{-1}$. Proceed with Step 2. Otherwise we have $x_1\in W_n(v_1^{-1})$. If $v_1$ is $b$-truncated or $b$-left, we get by \ref{intrect} that $y\notin W_n(v_1^{-1})$ for all $y\in W_n(v_0)$, which are of type $(i,n)$ for $i<n$. Then pick $x_2 \in W_n(v_0)$ such that $x_2$ is of type $(n-1,n)$. If $v_1$ is right-$b$, we get by \ref{intrect} that $y\notin W_n(v_1)$ for all $y\in W_n(v_0)$, which are of type $(n,j)$ for $j<n$. Then pick $x_2 \in W_n(v_0)$ such that $x_2$ is of type $(n,n-1)$. In both cases set $D_1=\varnothing\subset I$ and $E_1=\{v_1^{-1}\} \subset I^{-1}$. Proceed with Step 2.

Step $k$ for $1<k\leq 2\lvert I\rvert-1$: In the previous step we got $x_k \in W_n(v_0)$ such that $x_k$ is of type $(n-i_k,n-j_k)$ with $i_k+j_k=k-1$. We know by the previous steps that there are $D_{k-1}\subset I$ and $E_{k-1}\subset I^{-1}$ such that $\lvert D_{k-1}\rvert +\lvert E_{k-1}\rvert=k-1$ and $x_k\notin W_n(I\setminus D_{k-1})\cup W_n(I^{-1}\setminus E_{k-1})$. If $x_k \notin W_n(I\setminus \{v_0\})\cup W_n(I)^{-1}$, then set $u_n=x_k$ and stop. Otherwise we repeat the argument from above: We get that there are two sets $D_{k-1}\subset D_{k}\subset I$ and $E_{k-1}\subset E_{k}\subset I^{-1}$ such that $\lvert D_{k}\rvert +\lvert E_{k}\rvert=k$ and $y\notin W_n(I\setminus D_{k})\cup W_n(I^{-1}\setminus E_{k})$ for all $y\in W_n(v_0)$, which are of type $(i,j)$ for either $i<n-i_k,j\leq n-j_k$ or $i\leq n-i_k,j<n-j_k$. Pick $x_{k+1} \in W_n(v_0)$ accordingly of type $(n-i_{k+1},n-j_{k+1})$ equal $(n-i_k-1,n-j_k)$ or $(n-i_k,n-j_k-1)$.

Step $2\lvert I\rvert$: In the previous step we got $x_{2\lvert I\rvert} \in W_n(v_0)$ such that $x_{2\lvert I\rvert}$ is of type $(n-i_{2\lvert I\rvert},n-j_{2\lvert I\rvert})$ with $i_{2\lvert I\rvert}+j_{2\lvert I\rvert}=2\lvert I\rvert-1$. We also got that there are $D_{2\lvert I\rvert-1}\subset I$ and $E_{2\lvert I\rvert-1}\subset I^{-1}$ such that $\lvert D_{2\lvert I\rvert-1}\rvert +\lvert E_{2\lvert I\rvert-1}\rvert=2\lvert I\rvert-1$ and $x_{2\lvert I\rvert}\notin W_n(I\setminus D_{2\lvert I\rvert-1})\cup W_n(I^{-1}\setminus E_{2\lvert I\rvert-1})$. Since $x_{2\lvert I\rvert} \in W_n(v_0)$, we have $W_n(I\setminus D_{2\lvert I\rvert-1})\cup W_n(I^{-1}\setminus E_{2\lvert I\rvert-1})=W_n(I\setminus \{v_0\})\cup W_n(I)^{-1}$. So set $u_n=x_{2\lvert I\rvert}$ and stop.

This algorithm terminates after at most $2\lvert I\rvert$ steps. So we found $u_n\in W_n(v_0)$, that satisfies \ref{celllinebox1}. By the nature of our algorithm, we know that $u_n$ is of type $(i,j)$ for $i+j\geq 2n-2\lvert I\rvert+1$. Since $n>N$, we can use \ref{lengthrect} and get: \begin{equation}\begin{split} \lvert u_n\rvert_S &=\lvert(\phi v_0)_n\rvert_S-(n-i)-(n-j)\\
&=\lvert(\phi v_0)_n\rvert_S-2n+i+j\\
&\geq \lvert(\phi v_0)_n\rvert_S-2n+(2n-2\lvert I\rvert+1)\\
&=\lvert(\phi v_0)_n\rvert_S-2\lvert I\rvert+1 \end{split}\end{equation} So $u_n$ satisfies \ref{celllinebox2}.
\end{case}
\end{proof}

\subsection{The speed of $T^{-1}$}

Once we have found a long enough element in $\mathrm{supp}(\alpha_n)$, it is not too hard to prove that it is in the $b$-truncated end of $\mathrm{supp}(\alpha_n)$. We can use that $\mathrm{supp}(\alpha_n)$ is in the n-support of $I$ and we know the $b$-truncated end of rectangles by Section \ref{Tsingle}.

\begin{lem}\label{nbspeed}
Let $w\equiv_S (m_0,s_1,...,s_k,m_k)$ be a reduced word and assume that $w$ is not a power of $b$. Let $N=\max (\lvert m_0\rvert, \lvert m_k\rvert)$ and \begin{align}
w_0 &=b^{m_0}s_1...s_k,\\
w_k &=s_1...s_kb^{m_k}.
\end{align} Then we get for $n>2N$ \begin{enumerate}[(i)]
\item Let $w$ be $b$-truncated. Then\begin{equation} n_b(W^n(w))=-\infty\end{equation}
\item Let $w$ be $b$-left. Then\begin{equation} n_b(W^n(w))=\lvert m_0\rvert+ \lVert (\phi w_k)_n\rVert_S.\end{equation}
\item Let $w$ be right-$b$. Then\begin{equation} n_b(W^n(w))=\lVert (\phi w_0)_n\rVert_S+\lvert m_k\rvert.\end{equation}
\item Let $w$ be $b$-and-$b$. Then\begin{equation} n_b(W^n(w))=\max(\lvert m_0\rvert + \lVert (\phi w_k)_n\rVert_S, \lVert (\phi w_0)_n\rVert_S+\lvert m_k\rvert).\end{equation}
\end{enumerate}
\end{lem}

\begin{proof}
The main observation for this proof is that for $0\leq i,j,l\leq n$ we have that\begin{enumerate}
\item $L_i(w_0)=L_i(w)$, $M_l(w_0)=M_l(w)$ and $R_{j}(w_0)=\{e\}$,
\item $L_i(w_k)=\{e\}$, $M_l(w_k)=M_l(w)$ and $R_{j}(w_k)=R_{n}(w)$.\end{enumerate} Let $l_i\in dL_{i}(w)$ and $r_j\in dR_{j}$. Then it follows from the above that \begin{equation}\label{ww0wk}
\lvert (i,j)\rvert_{W_n(w)}=\lvert l_i\rvert_S + \lvert (0,j)\rvert_{W_n(w_k)}=\lvert (i,0)\rvert_{W_n(w_0)}+\lvert r_j\rvert_S.\end{equation} We get the following:
\begin{enumerate}[(i)]
\item Let $w$ be $b$-truncated. Then by \ref{ijbendWnw} the element of $W_n(w)$ is $b$-truncated. It follows that $n_b(W^n(w))=-\infty.$
\item Let $w$ be $b$-left. Then by \ref{ijbendWnw} every not $b$-truncated element of $W_n(w)$ is $(0,0)$. So $n_b(W^n(w))=\lvert (0,0)\rvert_{W_n(w)}$. By \ref{LRlength} we know that $\lvert l_0\rvert_S=\lvert m_0\rvert$ for $l_0\in dL_{0}(w)$. Since $w$ is $b$-left, we know that $w_k$ is $b$-truncated. By \ref{nnbiggest} we get that $\lvert(\phi w_k)_n\rvert_S=\lvert (0,0)\rvert_{W_n(w_k)}$. Combining these three facts with \ref{ww0wk} we get \begin{equation} n_b(W^n(w))=\lvert (0,0)\rvert_{W_n(w)}=\lvert l_0\rvert_S + \lvert (0,0)\rvert_{W_n(w_k)}=\lvert m_0\rvert+\lVert (\phi w_k)_n\rVert_S.\end{equation}
\item Let $w$ be right-$b$. Then by \ref{ijbendWnw} every not $b$-truncated element of $W_n(w)$ is $(0,0)$. So $n_b(W^n(w))=\lvert (0,0)\rvert_{W_n(w)}$. By \ref{LRlength} we know that $\lvert r_0\rvert_S=\lvert m_k\rvert$ for $r_0\in dR_{0}(w)$. Since $w$ is right-$b$, we know that $w_0$ is $b$-truncated. By \ref{nnbiggest} we get that $\lvert(\phi w_0)_n\rvert_S=\lvert (0,0)\rvert_{W_n(w_0)}$. Combining these three facts with \ref{ww0wk} we get \begin{equation} n_b(W^n(w))=\lvert (0,0)\rvert_{W_n(w)}=\lvert (0,0)\rvert_{W_n(w_0)}+\lvert r_0\rvert_S=\lVert (\phi w_k)_n\rVert_S+\lvert m_k\rvert.\end{equation}
\item Let $w$ be $b$-and-$b$. Then by \ref{ijbendWnw} every not $b$-truncated element of $W_n(w)$ is $(i,0)$ or $(0,j)$ for $0\leq i,j\leq n$. Using \ref{LRgrow} and \ref{LRanlauf} in the same way as in \ref{nnbiggest}, we get that \begin{align}
\lvert (0,j)\rvert_{W_n(w)}&\leq \lvert (0,n)\rvert_{W_n(w)}\\
\lvert (i,0)\rvert_{W_n(w)}&\leq \lvert (n,0)\rvert_{W_n(w)}
\end{align} So \begin{equation} n_b(W^n(w))=\max(\lvert (0,n)\rvert_{W_n(w)},\lvert (n,0)\rvert_{W_n(w)}).\end{equation} We compute both terms:\begin{enumerate}
\item By \ref{LRlength} we know that $\lvert l_0\rvert_S=\lvert m_0\rvert$ for $l_0\in dL_{0}(w)$. Since $w$ is $b$-and-$b$, we know that $w_k$ is right-$b$. By \ref{nnbiggest} we get that $\lvert(\phi w_k)_n\rvert_S=\lvert (0,n)\rvert_{W_n(w_k)}$. Combining these two facts with \ref{ww0wk} we get \begin{equation} \lvert (0,n)\rvert_{W_n(w)}=\lvert l_0\rvert_S + \lvert (0,n)\rvert_{W_n(w_k)}=\lvert m_0\rvert+\lVert (\phi w_k)_n\rVert_S.\end{equation}
\item By \ref{LRlength} we know that $\lvert r_0\rvert_S=\lvert m_k\rvert$ for $r_0\in dR_{0}(w)$. Since $w$ is $b$-and-$b$, we know that $w_0$ is $b$-left. By \ref{nnbiggest} we get that $\lvert(\phi w_0)_n\rvert_S=\lvert (n,0)\rvert_{W_n(w_0)}$. Combining these two facts with \ref{ww0wk} we get \begin{equation} \lvert (n,0)\rvert_{W_n(w)}=\lvert (n,0)\rvert_{W_n(w_0)}+\lvert r_0\rvert_S=\lVert (\phi w_0)_n\rVert_S+\lvert m_k\rvert.\end{equation}
\end{enumerate} It follows that \begin{equation}\begin{split} n_b(W^n(w))&=\max(\lvert (0,n)\rvert_{W_n(w)},\lvert (n,0)\rvert_{W_n(w)})\\
&=\max(\lvert m_0\rvert + \lVert (\phi w_k)_n\rVert_S, \lVert (\phi w_0)_n\rVert_S+\lvert m_k\rvert).\end{split}\end{equation}\end{enumerate}
\end{proof}

\begin{prop}\label{bendWnI}
Let $I$ be a finite set in normal form. Then by \ref{speed=sp} and \ref{speedonphib} the speed of $T^{-1}$ on $[\phi v]$ exists for all $v\in I$. So we can set \[a(I)=\max_{v\in I} \mathrm{sp}_S(T^{-1},[\phi v])\ \text{and}\ A(I)=\{v\in I\mid \mathrm{sp}_S(T^{-1},[\phi v])=a(I)\}.\] Let $a(I)>0$ and let $v\in A(I)$. Then for every $C\in\mathbb{N}$ there is $N\in\mathbb{N}$ such that $\lvert(\phi v)_n\rvert_S>n_b(W_n(I))+C$ for all $n>N$.
\end{prop}

\begin{proof}
By \ref{nbspeed} we can compute $n_b(W^n(w))$ for every element $w\in I$. Assume that $n_b(W^n(w))\neq -\infty$. Then using \ref{speed=sp} we can understand the result from \ref{nbspeed} the following way:\begin{equation}n_b(W^n(w))=k_w+ \lVert (\phi w')_n\rVert_S\end{equation} for some constant $k_w\in \mathbb{N}$ and some word $w'\in F_n\setminus \{e\}$ such that \begin{equation}\mathrm{sp}_S(T^{-1},[\phi w'])=\mathrm{sp}_S(T^{-1},[\phi w])-1.\end{equation} So in particular\begin{equation}\mathrm{sp}_S(T^{-1},[\phi w'])<a(I)=\mathrm{sp}_S(T^{-1},[\phi v]).\end{equation} By \ref{speedgiveslength} for every $w\in I$ there is $N_w\in \mathbb{N}$ such that for $n>N_w$ we have \begin{equation}
\lvert(\phi v)_n\rvert_S=\lvert T^{-n}[v]\rvert_S>n_b(W^n(w))+C
.\end{equation} Let $N=\max_{w\in I} N_w$. Then by \ref{nsmax} we get that $\lvert(\phi v)_n\rvert_S>n_b(W^n(I))+C$.
\end{proof}

We finally arrived at the main result of this section:

\begin{thm}\label{bigfish}
Let $f=\sum_{v\in I} \alpha(v)\phi v$ be a sum of counting quasimorphisms in normal form. Assume wlog that $I=\mathrm{supp}(\alpha)$. Since $f$ is in normal form, by \ref{speed=sp} and \ref{speedonphib} the speed of $T^{-1}$ on $[\phi v]$ exists for all $v\in I$. So we can define $a(I)$ and $A(I)$ as in \ref{bendWnI}. If $A(I)$ contains no $b$-and-$b$ element, but does contain a $b$-left (right-$b$) element, we assume that $A(I)$ contains a $b$-left (right-$b$) element of $b$-length bigger than 1. Then $f$ is speed reduced for $T^{-1}$.
\end{thm}

\begin{proof}
If $a(I)=0$, then $f$ is speed reduced for $T^{-1}$. So we can assume $a(I)>0$. Since $A(I)\subset I$, we know that $A(I)$ is in normal form and since $a(I)>0$, we know $b\notin A(I)$. By our assumption on $A(I)$ we can apply \ref{celllinebox}: There is $N_0\in\mathbb{N}$ and $v_0\in A(I)$ such that for all $n>N_0$ there is $u_n\in W_n(v_0)$ with\begin{enumerate}
\item\label{bigfish1} $u_n \notin W_n(A(I)\setminus \{v_0\})\cup W_n(A(I))^{-1}$,
\item\label{bigfish2} $\lvert u_n\rvert_S\geq \lvert(\phi v_0)_n\rvert_S-2\lvert A(I) \rvert$.
\end{enumerate} Since $v_0\in A(I)$, we know that $\mathrm{sp}_S(T^{-1},[\phi v_0])>\mathrm{sp}_S(T^{-1},[\phi v])$ for all $v\in I\setminus A(I)$. By \ref{tentacle} and \ref{speedgiveslength} it follows that there is $N_1\in\mathbb{N}$ such that for all $n>N_1$ and $v\in I\setminus A(I)$ we have\begin{equation}\label{bigfish3} \lvert (\phi v_0)_n\rvert_S=\lvert T^{-n}[\phi v_0]\rvert_S>\lvert T^{-n}[\phi v]\rvert_S+2\lvert A(I) \rvert=\lvert(\phi v)_n\rvert_S+2\lvert A(I)\rvert.\end{equation} Since $v_0\in A(I)$, we get by \ref{bendWnI} that there is $N_2\in\mathbb{N}$ such that for all $n>N_2$ we have\begin{equation}\label{bigfish4}\lvert(\phi v_0)_n\rvert_S>n_b(W_n(I))+2\lvert A(I)\rvert.\end{equation} Now let $N=\max (N_0,N_1,N_2)$. For all $n>N$, consider $u_n\in W_n(v_0)$ as in \ref{celllinebox}. By \ref{bigfish2} and \ref{bigfish3} we have $\lvert u_n\rvert_S\geq \lvert(\phi v_0)_n\rvert_S-2\lvert A(I) \rvert>\lvert(\phi v)_n\rvert_S$ for $v\in I\setminus A(I)$. So we have $u_n \notin W_n(I\setminus A(I))\cup W_n(I\setminus A(I))^{-1}$. By \ref{bigfish1} we know that $u_n \notin W_n(A(I)\setminus \{v_0\})\cup W_n(A(I))^{-1}$. So putting these two informations together we get $u_n \notin W_n(I\setminus \{v_0\})\cup W_n(I)^{-1}$. Since $f$ is in normal form, we know that the n-representative of $f$ is $f_n=\sum_{u\in W_n(I)} \alpha_n(u)\phi u$, where \[\alpha_n(u)=\sum_{v\in I}\alpha(v)v_n(u).\] We can compute \begin{equation} \alpha_n(u_n)=\sum_{v\in I}\alpha(v)v_n(u)=\begin{cases}
\alpha(v_0),& \text{if } u_n\in W^{+}_{n}(v_0)\\
-\alpha(v_0),& \text{if } u_n\in W^{-}_{n}(v_0)
\end{cases}.\end{equation} It follows that $\alpha_n(u_n)\neq 0$. Using \ref{bigfish2} and \ref{bigfish4}, we know that\begin{equation}\lvert u_n\rvert_S\geq \lvert(\phi v_0)_n\rvert_S-2\lvert A(I) \rvert>n_b(W_n(I)).\end{equation} It follows that $u_n\in E_b(W_n(I))$. Since $u_n\in E_b(W_n(I))$ and $\alpha_n(u_n)\neq 0$ we can apply the Criterion \ref{redtree}: We get that $f_n$ is reduced and much more usefully that \[\lvert f_n\rvert_S\geq \lvert u_n\rvert_S.\] Now we can prove that $f$ is speed reduced for $T^{-1}$ using \ref{howprovespeedred}:\begin{equation}\begin{split}\mathrm{lsp}_S(T^{-1},[f])&=\liminf_n \frac{\lvert f_n\rvert_S}{n}\\
&\geq \liminf_n \frac{\lvert u_n\rvert_S}{n}\\
&\geq \liminf_n \frac{\lvert(\phi v_0)_n\rvert_S-2\lvert A(I) \rvert}{n}\\
&=\liminf_n \frac{\lvert(\phi v_0)_n\rvert_S}{n}\\
&=\mathrm{lsp}_S(T^{-1},[\phi v_0])\\
&=\mathrm{sp}_S(T^{-1},[\phi v_0])\\
&=\sup_{v\in I} \mathrm{sp}_S(T^{-1},[\phi v])\\
&=\sup_{v\in I} \mathrm{usp}_S(T^{-1},[\phi v]).\end{split}\end{equation}
\end{proof}

We have to work a bit to remove the unnecessary assumptions.

\begin{defn}
Define $$\mathrm{rot}=\phi ab - \sum_{s\in S_b\setminus\{a^{-1}\}}\phi bs.$$ See Section 4.2.5 in \cite{scl}.
\end{defn}

\begin{lem}\label{annoyingproblem}
Let $f=\sum_{v\in I} \alpha(v)\phi v$ be a sum of counting quasimorphisms in normal form, that does not satisfy the conditions of Theorem \ref{bigfish}.  Then there is a real number $\lambda$, a weight $\beta$ and a finite set $J\subset F_n\setminus \{e\}$ such that all elements of $J$ are $b$-truncated and \begin{equation*}f\sim \lambda \mathrm{rot} + \sum_{s\in S_b\setminus\{a^{-1}\}} \beta(bs)\phi bs + \sum_{y\in J}\beta(y)\phi y\end{equation*}
\end{lem}

\begin{proof}
By assumption there are sets $N_l, N_r\subset \mathbb{N}$ and a finite set $J_0\subset F_n\setminus \{e\}$ such that all elements of $J_0$ are $b$-truncated such that $$A(I)\subset \bigcup_{s\in S_b}\bigcup_{n\in N_l} \{b^ns\} \cup \bigcup_{s\in S_b}\bigcup_{n\in N_r} \{sb^n\} \cup J_0.$$ Since $f$ is in normal form, the speed of $T^{-1}$ on $[\phi v]$ exists for all $v\in I$. By \ref{speed=sp} we get a finite set $J_1\subset F_n\setminus \{e\}$ such that all elements of $J$ are $b$-truncated and $$f=\sum_{s\in S_b}\sum_{n\in N_l} \alpha(b^ns)\phi b^ns + \sum_{s\in S_b}\sum_{n\in N_r} \alpha(sb^n)\phi sb^n+\sum_{y\in J_1}\alpha(y)\phi y.$$ The statement then follows from Corrollary \ref{cor:rewrite}.
\end{proof}

Note that the counting quasimorphism $\mathrm{rot}$ is $T^{-1}$-invariant and therefore in particular $\mathrm{sp}_S(T^{-1},[\mathrm{rot}])=0$.

\begin{lem}\label{annoyingproblem2}
Let $f=\sum_{s\in S_b\setminus\{a^{-1}\}} \beta(bs)\phi bs + \sum_{y\in J}\beta(y)\phi y$ for a weight $\beta$ and a finite set $J\subset F_n\setminus \{e\}$ such that all elements of $J$ are $b$-truncated. Let $f$ be in normal form. Then $f$ is speed reduced.
\end{lem}

\begin{proof}
Let $I=\bigcup_{s\in S_b\setminus\{a^{-1}\}}\{bs\}\cup J$. If $a(I)>1$ or $\beta(bs)=0$ for all $s\in S_b\setminus\{a^{-1}\}$, the statement follows from \ref{bigfish}. So we assume there is $s_0\in S_b\setminus\{a^{-1}\}$ such that $bs_0 \in A(I)$ and $\beta(bs_0)\neq 0$. We explicitely compute that $W_n(bs_0)=\{bs_0\}\cup \bigcup_{i=1}^n\{ab^{1-i}s_0\}$. So we can check that $W_n(bs_0)\cap W_n(bs_1)=\varnothing$ for $s_0\neq s_1\in S_b\setminus\{a^{-1}\}$ and $W_n(bs_0)\cap W_n(bs_1)^{-1}=\varnothing$ for $s_1\in S_b\setminus\{a^{-1}\}$. Since the elements of $J$ are $b$-truncated, we can use the algorithm from the proof of \ref{celllinebox} to get $N\in \mathbb{N}$ such that for all $n>N$ we can find $u_n\in W_n(bs_0)$ with \begin{enumerate}
\item $u_n \notin W_n(I\setminus \{bs_0\})\cup W_n(I)^{-1}$
\item $\lvert u_n\rvert_S\geq \lvert(\phi bs_0)_n\rvert_S-2\lvert I \rvert$
\end{enumerate} We can argue as in the proof of \ref{bigfish} to finish the proof.
\end{proof}

We summarize:

\begin{cor}\label{fequivspred}
Let $f=\sum_{v\in I} \alpha(v)\phi v$ be a sum of counting quasimorphisms. Then $f$ is equivalent to a $\lambda \mathrm{rot}+g$, where $\lambda$ is a real number and $g$ a sum of counting quasimorphisms that is speed reduced for $T^{-1}$. In particular $$\mathrm{sp}_S(T^{-1},[f])=\mathrm{sp}_S(T^{-1},[g])$$
\end{cor}

\begin{proof}
By \ref{normalform} we know that $f$ is equivalent to a sum of counting quasimorphisms $h$ in normal form. By \ref{bigfish} and \ref{annoyingproblem2} we know that $h$ is equivalent to $\lambda \mathrm{rot}+g$, where $\lambda$ is a real number and $g$ a sum of counting quasimorphisms that is speed reduced for $T^{-1}$.
\end{proof}

\begin{cor}\label{spexistseverywhere}
Let $f=\sum_{v\in I} \alpha(v)\phi v$ be a sum of counting quasimorphisms. Then the speed of $T^{-1}$ on $[f]$ exists.
\end{cor}

\begin{proof}
This follows from \ref{fequivspred}.
\end{proof}

\begin{exmp}\label{bpowersp}
In particular we can compute the speed of $T^{-1}$ on $[f]$ for every $[f]\in \tilde{\mathcal{B}}(F_n)$ using an equivalent sum of counting quasimorphisms in normal form. As an example we will compute $\mathrm{sp}_S(T^{-1},[\phi b^{m_0}])$ now. Since $[\phi b^{m_0}]=[-\phi b^{-m_0}]$ it is enough to consider $m_0>0$. If $m_0=1$, we know by \ref{speedonphib} that $\mathrm{sp}_S(T^{-1},[\phi b])=0$. So only the case $m_0>1$ is left: By applying $m_0-1$ right-extension relations we get \begin{equation}\phi b^{m_0}\sim \phi b-\sum_{s\in S_b}\sum_{i=1}^{m_0-1}\phi b^is.\end{equation} Then using \ref{t1} we get for every $\phi b^is$ with $1\leq i\leq m_0-1$\begin{equation}\phi b^is \sim 
\phi bs -\sum_{j=1}^{i-1}\sum_{s'\in S_b} \phi s'b^js.\end{equation} Putting the pieces together we get\begin{equation}\begin{split}\phi b^{m_0}&\sim \phi b-\sum_{s\in S_b}\sum_{i=1}^{m_0-1}\phi b^is\\
&\sim \phi b-\sum_{s\in S_b}\sum_{i=1}^{m_0-1}(\phi bs -\sum_{j=1}^{i-1}\sum_{s'\in S_b} \phi s'b^js)\\
&=\phi b-\sum_{s\in S_b}(m_0-1)\phi bs +\sum_{s\in S_b}\sum_{s'\in S_b}\sum_{k=1}^{m_0-2} (m_0-1-k)\phi s'b^ks=f_{m_0}.\end{split}\end{equation} Then $f_{m_0}$ is in normal form and so speed reduced for $T^{-1}$ by \ref{bigfish}. Using \ref{speed=sp} it is easy to see that $\mathrm{sp}_S(T^{-1},[\phi b^{m_0}])=1$.
\end{exmp}

\begin{cor}
Endow $\mathbb{R}$ with the trivial absolute value $\lvert -\rvert_0$. Then $\mathrm{sp}_S(T^{-1},-)$ is a seminorm on $\tilde{\mathcal{B}}(F_n)$.
\end{cor}

\begin{proof}
This follows from \ref{spexistseverywhere} and \ref{speedseminorm}.
\end{proof}

\section{Invariant Subspaces}\label{fixpts}

In this section we want to prove that there are no non-trivial finite-dimensional invariant subspaces in $H^2_b(F_n, \mathbb{R})_{\mathrm{fin}}$ under the action of $\mathrm{Out}(F_n)$. We will heavily rely on our results from the previous sections.

\begin{cor}\label{spnotfix}
Assume that $X\in \mathrm{Out}(F_n)$ has linear speed on $[f]$. Then $[f]$ is not contained in a finite-dimensional $\mathrm{Out}(F_n)$-invariant subspace of the Brooks space. In particular $[f]$ is not a fixpoint under the action of $\mathrm{Out}(F_n)$.
\end{cor}

\begin{proof}
Every finite-dimensional subspace of the Brooks space is bounded in the norm $\lvert -\rvert_S$. But if $X$ has linear speed on $[f]$, then there is $k_1\in \mathbb{R}_{+}$ and $N\in \mathbb{N}$ such that $k_1n\leq\lvert X^{n}[f]\rvert_S$ for all $n>N$.
\end{proof}

So the only candidates for fixpoints under the $\mathrm{Out}(F_n)$ action on $\tilde{\mathcal{B}}(F_n)$ are elements $[f]$ such that $\mathrm{sp}_S(T^{-1},[f])=0$. We know by \ref{speed=sp} and \ref{bpowersp} that $\mathrm{sp}_S(T^{-1},[\phi v])=0$ for $v\in F_n\setminus \{e\}$ if and only if \[v\in O=\{b,b^{-1}\}\cup \{w\mid w\ \text{is not a power of}\ b, sp(w)=0\}.\]

\begin{thm}\label{linspsums}
Let $f\in \mathcal{B}(F_n,S)$ be a sum of counting quasimorphisms.
Then exactly one of the following is true:\begin{enumerate}[(i)]
\item $T^{-1}$ has strongly linear speed on $[f]$.
\item $f$ is equivalent to $\lambda \mathrm{rot}+g$, where $\lambda$ is a real number and $g=\sum_{w\in J} \beta(w)\phi w$ in normal form with $J\subset O$.
\end{enumerate}
\end{thm}

\begin{proof}
By \ref{fequivspred} we know that $f$ is equivalent to $\lambda \mathrm{rot}+g$, where $\lambda$ is a real number and $g=\sum_{w\in J} \beta(w)\phi w$ is speed reduced for $T^{-1}$. We assume wlog that $\mathrm{supp}(\beta)=J$. Since $g$ is speed reduced for $T^{-1}$ exactly one of the following statements is true\begin{enumerate}[(i)]
\item $T^{-1}$ has strongly linear speed on $[g]$ and so by \ref{fequivspred} on $[f]$.
\item $\mathrm{sp}_S(T^{-1},[\phi w])=0$ for all $w\in \mathrm{supp}(\beta)=J$.
\end{enumerate}
By definition of $O$, we know that the second statement is true if and only if $J\subset O$.
\end{proof}

So to prove that no element of $\tilde{\mathcal{B}}(F_n)$ is fixed by $\mathrm{Out}(F_n)$, it is enough to show that $[f]$ is not fixed for $f=\lambda \mathrm{rot}+g$, with $\lambda \in \mathbb{R}$ and $g=\sum_{v\in I} \alpha(v)\phi v$ in normal form with $I\subset O$. This motivates looking at the elements of $O$.

\begin{lem}\label{stationarywexplicit}
Let $w\equiv_S (m_0,s_1,...,s_k,m_k)$ be a reduced word and assume that $w$ is not a power of $b$. Then $w\in O$ if and only if the following conditions are satisfied\begin{enumerate}
\item $w$ is $b$-truncated.
\item If $s_i=a$ for $1\leq i<k$, then $s_{i+1}=a^{-1}$.
\item If $s_{i+1}=a^{-1}$ for $1\leq i<k$, then $s_i=a$.
\end{enumerate}
\end{lem}

\begin{proof}
This follows from \ref{def:sp}.
\end{proof}

\begin{prop}\label{bigger2notfixed}
Let $[f]\in \tilde{\mathcal{B}}(F_n)$ and assume that $\lvert [f]\rvert_S>1$. Then there is an element $X\in \mathrm{Out}(F_n)$ such that $\mathrm{sp}_S(T^{-1},X[f])>0$.
\end{prop}

\begin{proof}
By \ref{linspsums} we only have to check the statement for $[f]\in \tilde{\mathcal{B}}(F_n)$, where $f=\lambda \mathrm{rot}+g$, for $\lambda\in \mathbb{R}$ and $g=\sum_{v\in I} \alpha(v)\phi v$ is in normal form with $I\subset O$. We will consider five cases separately:

\begin{case} Assume $\lambda\neq 0$.
By \ref{Xinv} we get that \begin{equation*}\begin{split} 
\mathrm{rot}\circ H &=\phi a^{-1}b - \sum_{s\in S_b\setminus\{a\}} \phi bs\\
&= -\phi b^{-1}a +\phi ab^{-1} - \sum_{s\in S_b\setminus\{a,a^{-1}\}} \phi bs\\
&\sim -\phi a + \phi ba + \sum_{s\in S_b\setminus\{a^{-1}\}} \phi sa +\phi a - \phi ab - \sum_{s\in S_b\setminus\{a^{-1}\}} \phi as - \sum_{s\in S_b\setminus\{a,a^{-1}\}} \phi bs\\
&= -\mathrm{rot} -2\sum_{s\in S_b\setminus\{a,a^{-1}\}} \phi bs+ \sum_{s\in S_b\setminus\{a,a^{-1}\}} \phi sa - \sum_{s\in S_b\setminus\{a,a^{-1}\}} \phi as
\end{split}\end{equation*} We now apply the Nielsen transformation $H$ to $[f]$. We get that $f\circ H\sim -\lambda \mathrm{rot}-2\lambda\sum_{s\in S_b\setminus\{a,a^{-1}\}} \phi bs+h$, where $h=\sum_{w\in J} \beta(w)\phi w$ for a weight $\beta$ and a finite set $J\subset F_n\setminus \{e\}$ such that all elements of $J$ are $b$-truncated. By \ref{annoyingproblem2} and \ref{fequivspred} it follows that $\mathrm{sp}_S(T^{-1},H[f])=\mathrm{sp}_S(T^{-1},[-2\sum_{s\in S_b\setminus\{a,a^{-1}\}} \phi bs+h])\geq 1$.
\end{case}

So in the following four cases we can assume that $f=\sum_{v\in I} \alpha(v)\phi v$ is in normal form and $I\subset O$. For this proof we assume wlog that $\mathrm{supp}(\alpha)=I$. By \ref{stationarywexplicit} every element of $I$ is either $b$ or $b$-truncated. 

\begin{case} There is $w\in I$ such that $w\equiv_S (m_0,s_1,...,s_k,m_k)$, where $s_i\in \{a,a^{-1}\}$ and $s_j\in \bar{S}\setminus \{a,a^{-1}\}$ for some $1\leq i<j\leq k$. 

Then there is $1\leq m<k$ such that $s_m\in \{a,a^{-1}\}$ and $s_{m+1}\in \bar{S}\setminus \{a,a^{-1}\}$. Since $w\in O$, we get by \ref{stationarywexplicit} that $s_m=a^{-1}$. Now we apply the Nielsen transformation $H$ to $[f]$. By \ref{Xinv} we get that $f\circ H=\sum_{v\in I} \alpha(v)\phi Hv$. We see that $Hw\equiv_S (m_0,t_1,...,t_k,m_k)$, where $Hs_i=t_i$ for all $1\leq i\leq k$. Then $t_m=a$ and $t_{m+1}\in \bar{S}\setminus \{a,a^{-1}\}$, so we get by \ref{stationarywexplicit} that $Hw\notin O$. We also see that $f\circ H$ is still in normal form: Since $v\in I$ is either $b$ or $b$-truncated, so is $Hv$. By \ref{bigfish} we get that $f\circ H$ is speed reduced for $T^{-1}$. It follows that $T^{-1}$ has strongly linear speed on $H[f]$.\end{case}

\begin{case} There is $w\in I$ such that $w\equiv_S (m_0,s_1,...,s_k,m_k)$, where $s_i\in \bar{S}\setminus \{a,a^{-1}\}$ and $s_j\in \{a,a^{-1}\}$ for some $1\leq i<j\leq k$. 

Then there is $1\leq m<k$ such that $s_m\in \bar{S}\setminus \{a,a^{-1}\}$ and $s_{m+1}\in \{a,a^{-1}\}$. Since $w\in O$, we get by \ref{stationarywexplicit} that $s_{m+1}=a$. Now we apply the Nielsen transformation $H$ to $[f]$. By \ref{Xinv} we get that $f\circ H=\sum_{v\in I} \alpha(v)\phi Hv$. We see that $Hw\equiv_S (m_0,t_1,...,t_k,m_k)$, where $Hs_i=t_i$ for all $1\leq i\leq k$. Then $t_m\in \bar{S}\setminus \{a,a^{-1}\}$ and $t_{m+1}=a^{-1}$, so we get by \ref{stationarywexplicit} that $Hw\notin O$. We also see that $f\circ H$ is still in normal form: Since $v\in I$ is either $b$ or $b$-truncated, so is $Hv$. By \ref{bigfish} we get that $f\circ H$ is speed reduced for $T^{-1}$. It follows that $T^{-1}$ has strongly linear speed on $H[f]$.\end{case}

So we can assume that for every word $w\in I\setminus \{b\}$ with $w\equiv_S (m_0,s_1,...,s_k,m_k)$ either $s_i\in \{a,a^{-1}\}$ or $s_i\in \bar{S}\setminus \{a,a^{-1}\}$ for all $1\leq i\leq k$.

\begin{case} There is $w\in I$ such that $w\equiv_S (m_0,s_1,...,s_k,m_k)$ with $k\geq 2$, where $s_i\in \{a,a^{-1}\}$ for all $1\leq i\leq k$. 

Then $P_1w$ is $b$-and-$b$ and therefore $P_1w\notin O$. We now want to prove that $f\circ P_1$ is still in normal form. This will need some work: Using that $\phi v=-\phi v^{-1}$ we can assume\begin{enumerate}
\item $a^{-1}\notin I$.
\item If $ab^ma^{-1}\in I$, then $m>0$.
\item If $a^{-1}b^ma\in I$, then $m<0$.
\end{enumerate} Let $v_1\neq v_2 \in I$ and write $v_1\equiv_S (m_0,s_1,...,s_k,m_k)$ and $v_2\equiv_S (n_0,t_1,...,t_l,n_l)$. Assume that $\tau_b(P_1v_1)=\tau_b(P_1v_2)$. By definition of $P_1$ we get that \begin{equation}P_1\tau_a(v_1)=\tau_b(P_1v_1)=\tau_b(P_1v_2)=P_1\tau_a(v_1).\end{equation} Since $P_1$ is an automorphism it follows that $\tau_a(v_1)=\tau_a(v_2)$. We want to show $v_1=v_2$. We will consider four cases:\begin{enumerate}
\item Assume that $\tau_a(v_1)=\tau_a(v_2)=\varnothing$. Since $a^{-1}\notin I$, we get that $v_1=v_2=a$.
\item Assume that $\#s(\tau_a(v_1))>0$ for some $s\in S_b\setminus \{a,a^{-1}\}$. Then we have $s_i\in \bar{S}\setminus \{a,a^{-1}\}$ for all $1\leq i\leq k$ and $t_j\in \bar{S}\setminus \{a,a^{-1}\}$ for all $1\leq j\leq l$. So $v_1=\tau_a(v_1)=\tau_a(v_2)=v_2$.
\item Assume that $\#s(\tau_a(v_1))>0$ for some $s\in \{a,a^{-1}\}$. Then we have $s_i\in \{a,a^{-1}\}$ for all $1\leq i\leq k$ and $t_j\in \{a,a^{-1}\}$ for all $1\leq j\leq l$. So $k=l$ and $s_1=t_1=s_2^{-1}$ and $s_k=t_k=s_{k-1}^{-1}$. So $v_1=v_2$.
\item Assume that $\tau_a(v_1)=\tau_a(v_2)=b^m$ for some $m\in \mathbb{Z}\setminus \{0\}$. By our assumptions on $I$ we get that $v_1=v_2=ab^ma^{-1}$ if $m>0$, and that $v_1=v_2=a^{-1}b^ma$ if $m<0$.\end{enumerate}

So we know that $f\circ P_1$ is still in normal form. By \ref{bigfish} we get that $f\circ P_1$ is speed reduced for $T^{-1}$. It follows that $T^{-1}$ has strongly linear speed on $P_1[f]$.\end{case}

So we can assume that for every word $w\in I$ with $w\equiv_S (m_0,s_1,...,s_k,m_k)$ either $w\in \{a,a^{-1}\}$ or $s_i\in \bar{S}\setminus \{a,a^{-1}\}$ for all $1\leq i\leq k$. Since $\lvert [f]\rvert_S>1$ only one case is left:

\begin{case} There is $w\in I$ such that $w\equiv_S (m_0,s_1,...,s_k,m_k)$ with $k\geq 2$, where $s_i\in \bar{S}\setminus \{a,a^{-1}\}$ for all $1\leq i\leq k$. 

Then there is a number $1<j<n$ such that $P_2^js_1\in \{a,a^{-1}\}$. So $P_2^j[f]$ falls in one of the other three cases. We can conclude that there is $n\in \mathbb{N}$ such that $\lvert T^{-n}P_2^j[f]\rvert_S\neq \lvert [f]\rvert_S$.\end{case}
\end{proof}

Now we can answer Ab\'{e}rt's question on the Brooks space.

\begin{cor}\label{halfAbert}
The action of $\mathrm{Out}(F_n)$ on the Brooks space $\tilde{\mathcal{B}}(F_n)$ has no non-trivial fixpoints.
\end{cor}

\begin{proof}
By \ref{bigger2notfixed} and \ref{spnotfix} we only have to check $[f]\in \tilde{\mathcal{B}}(F_n)$ with $\lvert [f]\rvert_S\leq 1$. If $\lvert [f]\rvert_S=0$, we get that $[f]=0\in \tilde{\mathcal{B}}(F_n)$. If $\lvert [f]\rvert_S=1$, we can write $f=\sum_{s\in S}\alpha(s)\phi s$. If $\alpha(b)\neq 0$, we know that $T^{-1}[f]\neq [f]$ by \ref{relations} and $T^{-1}\phi b=\phi a+\phi b$. If $\alpha(b)=0$, we choose $s\in S$ such that $\alpha(s)\neq 0$. There is $j\in \mathbb{N}$ such that $P_2^js=b$. We get by the above that $T^{-1}P_2^j[f]\neq [f]$.
\end{proof}

We can even give a stronger statement for the image of $\tilde{\mathcal{B}}(F_n)$ in $H_b^2(F_n,\mathbb{R})$.

\begin{cor}\label{nofindiminv}
The action of $\mathrm{Out}(F_n)$ on $H^2_b(F_n, \mathbb{R})_{\mathrm{fin}}$ admits no non-trivial finite-dimensional $\mathrm{Out}(F_n)$-invariant subspaces. In particular it has no non-trivial fixed point. 
\end{cor}

\begin{proof}
This follows from \ref{H2bnorm}, \ref{bigger2notfixed} and \ref{spnotfix}.
\end{proof}

\section{Open problems}\label{open}

Not only does this article leave questions open, but it leaves us even with more questions than we started with. To start with the obvious: Ab\'{e}rt's question \ref{47} is still open. We answered the question in \ref{halfAbert} for the Brooks space $\tilde{\mathcal{B}}(F_n)$, which is a dense subspace of $\tilde{\mathcal{Q}}(F_n)$. But the step up to the whole space $\tilde{\mathcal{Q}}(F_n)$ was not even attempted in this article. Here is still work to do. But apart from this, our introduction of the (general) notion of speed in \ref{speed} sparks several new questions concerning $\tilde{\mathcal{B}}(F_n)$ and $\mathrm{Out}(F_n)$.

Does the type of speed $g$ of an element $X\in \mathrm{Out}(F_n)$ depend on the generating set $S$? This boils down to the question if different generating sets $S_1$ and $S_2$ give us equivalent norms $\lvert -\rvert_{S_i}$ on $\tilde{\mathcal{B}}(F_n)$, which we think is true.

Every element $X\in \mathrm{Out}(F_n)$ has speed of some type $g$. So we get for every $X\in \mathrm{Out}(F_n)$ a seminorm on $\tilde{\mathcal{B}}(F_n)$ by looking at its upper $g$ speed. How do these seminorms look like? Can we extend this seminorms in any meaningful way to the bigger space $\tilde{\mathcal{Q}}(F_n)$?

Which types of speed $g$ exist? We saw an example of linear speed in this article and it is easy to find an example of exponential speed. But are there elements $X\in \mathrm{Out}(F_n)$ with intermediate speed?

Assume $X\in \mathrm{Out}(F_n)$ has speed of some type $g$. Does the $g$ speed of $X$ on $[f]$ exist for every $[f]\in \tilde{\mathcal{B}}(F_n)$? This is the case for $T^{-1}$. Once this would be proved the notions of upper and lower speed would become superfluous.

Can we see properties of elements in $\mathrm{Out}(F_n)$ in their speed? Examples in $\mathrm{Out}(F_2)$ suggest that there unipotent elements have linear speed, whereas semisimple elements of infinite order have exponential speed. So there seems to be some connection between the speed of $X$ and other already known properties. One might also try to use the map from $\mathrm{Out}(F_n)$ to the space of seminorms on $\tilde{\mathcal{B}}(F_n)$ given by the upper speed to get information on $\mathrm{Out}(F_n)$.

Last but not least there is of course the question of how far our results can be extended. Analogous constructions to the Brook space have been defined in a vast variety of environments (see page 6 of \cite{HT} for a list). These spaces are still widely mysterious. Can we define analogues of our norm $\lvert -\rvert_S$ or the speed $\mathrm{sp}_S(X,[f],g)$ in these cases? Given the resemblance of $\mathrm{Out}(F_n)$ and mapping class groups, can we extend our results to say something about surfaces (for example about pseudo-Anosov maps)?

We plan to investigate some of these questions in future work.

\appendix

\section{Speed}\label{speed}

In \ref{linspeed} we introduced linear speed. In this appendix we expand this to get a general notion of the speed of an element $X\in \mathrm{Out}(F_n)$. The definitions, results and proofs are exactly analogous to their counterparts in \ref{linspeed}. Therefore we will not write down any proofs in this appendix. We assume in the following that $g:\mathbb{N}\to \mathbb{R_+}$.

\begin{defn}
Let $X\in \mathrm{Out}(F_n)$ and $[f]\in \tilde{\mathcal{B}}(F_n)$. We say that $X$ has \emph{speed of type $g$ on $[f]$} if $\lvert X^{n}[f]\rvert_S=\Theta (g)$.
\end{defn}

\begin{defn}
Let $X\in \mathrm{Out}(F_n)$. We say that $X$ has \emph{speed of type $g$} if it satisfies:\begin{enumerate}
\item $\lvert X^{n}[f]\rvert_S=O(g)$ for all $[f]\in \tilde{\mathcal{B}}(F_n)$.
\item There is $[f_0]\in \tilde{\mathcal{B}}(F_n)$ such that $X$ has speed of type $g$ on $[f_0]$.\end{enumerate}
\end{defn}

\begin{exmp}
\begin{enumerate}
\item If $X$ has speed of type $Id$, we also say that $X$ has linear speed.
\item If $X$ has speed of type $p$ for some polynomial $p$, we also say that $X$ has polynomial speed.
\item If $X$ has speed of type $a^n$ for some $a>1$, we also say that $X$ has has exponential speed.
\end{enumerate}
\end{exmp}

\begin{defn}
Let $X\in \mathrm{Out}(F_n)$ and $[f]\in \tilde{\mathcal{B}}(F_n)$.\begin{enumerate}
\item We call $\mathrm{usp}_S(X,[f],g)=\limsup_n \frac{\lvert X^{n}[f]\rvert_S}{g(n)}$ the \emph{upper $g$ speed} of $X$ on $[f]$.
\item We call $\mathrm{lsp}_S(X,[f],g)=\liminf_n \frac{\lvert X^{n}[f]\rvert_S}{g(n)}$ the \emph{lower $g$ speed} of $X$ on $[f]$.
\item Assume $\lim_n \frac{\lvert X^{n}[f]\rvert_S}{g(n)}\in \mathbb{R}$. Then we call $\mathrm{sp}_S(X,[f],g)=\lim_n \frac{\lvert X^{n}[f]\rvert_S}{g(n)}$ the \emph{$g$ speed} of $X$ on $[f]$ and say that \emph{the $g$ speed of $X$ on $[f]$ exists}.
\end{enumerate}
\end{defn}

These three different speeds have the same properties as in the linear case.

\begin{lem}
Endow $\mathbb{R}$ with the trivial absolute value $\lvert -\rvert_0$. Let $[f]\in \tilde{\mathcal{B}}(F_n)$ and $a\in \mathbb{R}$. Then \begin{enumerate}[(i)]
\item $\mathrm{usp}_S(X,[af],g)=\lvert a\rvert_0\mathrm{usp}_S(X,[f],g)$
\item $\mathrm{lsp}_S(X,[af],g)=\lvert a\rvert_0\mathrm{lsp}_S(X,[f],g)$
\item $\mathrm{sp}_S(X,[af],g)=\lvert a\rvert_0\mathrm{sp}_S(X,[f],g)$ if both sides are defined.
\end{enumerate}
\end{lem}

\begin{lem}
Let $f=\sum_{v\in I} \alpha(v)\phi v$ be a sum of counting quasimorphisms. Then 
\begin{enumerate}[(i)]
\item $\mathrm{usp}_S(X,[f_1+f_2],g)\leq \sup (\mathrm{usp}_S(X,[f_1],g),\mathrm{usp}_S(X,[f_2],g)$
\item $\mathrm{sp}_S(X,[f_1+f_2],g)\leq \sup (\mathrm{sp}_S(X,[f_1],g),\mathrm{sp}_S(X,[f_2],g)$ if both sides are defined.
\end{enumerate}
\end{lem}

\begin{cor}
Let $X\in \mathrm{Out}(F_n)$. Then we get:\begin{enumerate}
\item Assume that $\mathrm{usp}_S(X,[f],g)<\infty$ for all $[f]\in \tilde{\mathcal{B}}(F_n)$. Then the upper speed $\mathrm{usp}_S(X,-,g)$ is a seminorm on $\tilde{\mathcal{B}}(F_n)$.
\item Assume that $\mathrm{sp}_S(X,[f],g)$ exists for all $[f]\in \tilde{\mathcal{B}}(F_n)$. Then the speed $\mathrm{sp}_S(X,-,g)$ is a seminorm on $\tilde{\mathcal{B}}(F_n)$.
\end{enumerate}
\end{cor}

By knowing the $g$ speed of $X$ on $[f]$, we get information on the length of $X^{n}[f]$ for big enough $n$.

\begin{cor}
Let $X\in \mathrm{Out}(F_n)$ and let $g:\mathbb{N}\to \mathbb{R}$ such that $\lim_{n\to \infty} g(n)=\infty$. Assume that the $g$ speed of $X$ exists on $[f_1], [f_2]\in \tilde{\mathcal{B}}(F_n)$. If $\mathrm{sp}_S(X,[f_1],g)>\mathrm{sp}_S(X,[f_2],g)$, then for every constant $C\in \mathbb{R}$ there is $N\in \mathbb{N}$ such that $\lvert X^{n}[f_1]\rvert_S>\lvert X^{n}[f_2]\rvert_S+C$ for $n>N$.
\end{cor}

We have the same concepts that we used to prove that $T^{-1}$ has linear speed.

\begin{defn}
Assume that the speed of $X$ on $[f]$ exists and $\mathrm{sp}_S(X,[f],g)>0$. Then we say that $X$ has \emph{strongly speed of type $g$ on $[f]$}.
\end{defn}

\begin{lem}
If $X$ has strongly speed of type $g$ on $[f]$, then $X$ has speed of type $g$ on $[f]$.
\end{lem}

\begin{defn}
Let $f=\sum_{v\in I} \alpha(v)\phi v$ be a sum of counting quasimorphisms. Then $f$ is \emph{$g$ speed reduced} for $X$ if $\mathrm{sp}_S(X,[f],g)=\sup_{v\in \mathrm{supp}(\alpha)} \mathrm{sp}_S(X,[\phi v],g)$ and both sites of the equation exist.
\end{defn}

\begin{lem}
Let $f=\sum_{v\in I} \alpha(v)\phi v$ be a sum of counting quasimorphisms. Then $f$ is $g$ speed reduced for $X$ if the $g$ speed of $X$ on $[\phi v]$ exists for all $v\in \mathrm{supp}(\alpha)$ and \begin{equation}
\mathrm{lsp}_S(X,[f],g)\geq\sup_{v\in \mathrm{supp}(\alpha)} \mathrm{usp}_S(X,[\phi v],g).
\end{equation}
\end{lem}

\end{document}